\documentclass[11pt]{article}

\usepackage{tikz}
\usetikzlibrary{shapes}
\usetikzlibrary{decorations.pathreplacing}
\usepackage{enumerate}
\usepackage[shortlabels]{enumitem}
\usepackage{verbatim}
\usepackage[margin=1in]{geometry}
\usepackage{amssymb}
\usepackage{caption}
\usepackage[normalem]{ulem}
\usepackage{bbm}
\usepackage{amsthm}
\usepackage{enumitem}
\usepackage{scalerel}    
\usepackage{stmaryrd}   

\usetikzlibrary{arrows.meta}

\usepackage[nottoc,notlot,notlof]{tocbibind}

\DeclareFontFamily{U}{mathx}{\hyphenchar\font45}
\DeclareFontShape{U}{mathx}{m}{n}{
      <5> <6> <7> <8> <9> <10>
      <10.95> <12> <14.4> <17.28> <20.74> <24.88>
      mathx10
      }{}
\DeclareSymbolFont{mathx}{U}{mathx}{m}{n}
\DeclareFontSubstitution{U}{mathx}{m}{n}
\DeclareMathAccent{\widecheck}{0}{mathx}{"71}
\DeclareMathAccent{\wideparen}{0}{mathx}{"75}

\usepackage{rotating,centernot,cancel}    

\newcommand{\myfootnote}[1]{
    \renewcommand{\thefootnote}{}
    \footnotetext{\scriptsize#1}
    \renewcommand{\thefootnote}{\arabic{footnote}}
}

\usepackage{amsmath}
\usepackage{tocloft}
\usepackage{float}

\usepackage{tcolorbox}   
\PassOptionsToPackage{hyphens}{url}
\usepackage[hidelinks]{hyperref}
\usepackage{babel}
\theoremstyle{plain}

\input amssym.def
\input amssym.tex

\def\beqn{\begin{eqnarray}}
\def\eeqn{\end{eqnarray}}

\newcommand{\olsi}[1]{\underline{{#1}}}

\def\beq{\begin{equation}}
\def\eeq{\end{equation}}
\usepackage{mathtools}
\DeclarePairedDelimiter\floor{\lfloor}{\rfloor}

\newtheorem{theorem}{Theorem}[section]
\newtheorem{corollary}[theorem]{Corollary} 
\newtheorem{lemma}[theorem]{Lemma} 
\newtheorem*{lemma*}{Lemma}
\newtheorem{proposition}[theorem]{Proposition} 

\theoremstyle{remark}
\newtheorem{remark}[theorem]{Remark}

\theoremstyle{definition}

\numberwithin{figure}{section}
\numberwithin{equation}{section}

\def\R{\mathbb R}

\def\Z{\mathbb Z}

\def\dlim[#1][#2]{\lim_{#1 \to #2, #1 \neq #2}}

\def\Var{\textup{$\mathbb{V}$ar}}

\def\Cov{\textup{$\mathbb{C}$ov}}
\def\Corr{\textup{$\mathbb{C}$orr}}

\newcommand{\be}{\begin{equation}}
\newcommand{\ee}{\end{equation}}

\newcommand{\nn}{\nonumber}

\def\ind{\mathbbm{1}}

\def\wt{\widetilde}    \def\wc{\widecheck}

\newcounter{cnstcnt}

\usepackage{xargs}
\usepackage[colorinlistoftodos,prependcaption,textsize=footnotesize]{todonotes}
\newcommandx{\addmath}[2][1=]{\todo[linecolor=red,backgroundcolor=red!25,bordercolor=red,#1]{#2}}
\newcommandx{\fixtext}[2][1=]{\todo[linecolor=blue,backgroundcolor=blue!25,bordercolor=blue,#1]{#2}}
\newcommandx{\note}[2][1=]{\todo[linecolor=yellow,backgroundcolor=yellow!25,bordercolor=yellow,#1]{#2}}

\def\stat{{\mu/2}}

\def\fg{\log \wt{Z}}

\def\fgs{\log Z^{W}}

\newcommand\bbullet{{{\scaleobj{0.6}{\bullet}}}} 

\title{Time correlations in KPZ models with diffusive initial conditions}
\author{
  Riddhipratim Basu\thanks{\scriptsize{International Centre for Theoretical Sciences, Tata Institute of Fundamental Research, Bengaluru, INDIA.
 \texttt{rbasu@icts.res.in}}}  \qquad  Xiao Shen\thanks{\scriptsize{Department of Mathematics, University of Utah, Utah, USA. \texttt{xiao.shen@utah.edu}}}
  }
\date{}
\begin{document}

\allowdisplaybreaks

\maketitle

\begin{abstract}
Temporal correlation for randomly growing interfaces in the KPZ universality class is a topic of recent interest. Most of the works so far have been concentrated on the zero temperature model of exponential last passage percolation, and three special initial conditions, namely droplet, flat and stationary. We focus on studying the time correlation problem for generic random initial conditions with diffusive growth. We formulate our results in terms of the positive temperature exactly solvable model of the inverse-gamma polymer and obtain up to constant upper and lower bounds for the correlation between the free energy of two polymers whose endpoints are close together or far apart. Our proofs apply almost verbatim to the zero temperature set-up of exponential LPP and are valid for a broad class of initial conditions.
Our work complements and completes the partial results obtained in \cite{ferocctimecorr}, following the conjectures of \cite{Ferr-Spoh-2016}. Moreover, our arguments rely on the one-point moderate deviation estimates which have recently been obtained using stationary polymer techniques and thus do not depend on complicated exact formulae. 
\end{abstract}

\myfootnote{Date: \today}
\myfootnote{2010 Mathematics Subject Classification. 60K35, 	60K37}
\myfootnote{Key words and phrases: correlation, coupling, directed polymer, Kardar-Parisi-Zhang, random growth model.}

\tableofcontents

\section{Introduction and main results}
A large number of planar random growth models are expected to share the universal features of the so-called Kardar-Parisi-Zhang (KPZ) universality class; these include planar last passage percolation (LPP) and polymer models on the plane under mild conditions on the underlying noise. Although the questions of universality remain out of reach, there are several exactly solvable lattice models that can be rigorously analyzed, and studying these models has been an important topic of research over the last twenty-five years. One of the problems that has attracted a lot of interest over the last few years is understanding the temporal decay of correlations in such models with different initial conditions. 

We undertake the study of time correlations in (1+1)-dimensional KPZ growth models started from a class of random initial conditions. For concreteness, we shall formulate and prove our main results in the framework of the exactly solvable positive temperature model of inverse-gamma polymer on $\Z^2$. Our arguments, nonetheless, carry over almost verbatim to the zero temperature model of exponential LPP on $\Z^2$ as well; a short discussion on this appears in Section \ref{s:lpp}.

Before proceeding further, let us explain the basic formulation of this problem. For simplicity, we describe this in the setup of the zero temperature model of exponential LPP. We consider an initial condition $f$, which is a function defined on the anti-diagonal line $x+y=0$. For ${\bf u}, {\bf v}\in \Z^2$, (${\bf u}\le {\bf v}$ coordinatewise) let $T_{{\bf u}, {\bf v}}$ denote the point-to-point last passage time (i.e., the maximum weight of a directed path joining $u$ and $v$ in an underlying field of i.i.d.\ exponential vertex weights). We then define the last-passage time to $(n,n)$  with the initial condition $f$, denoted $T_n^{f}$ by $\max_{{\bf v}} T_{{\bf v}, (n,n)}+f({\bf v})$ where the maximum is taken over all points $v$ on the line $x+y=0$. Then the two-time correlation is the correlation between $T^{f}_{r}$ and $T_n^{f}$ for $r,n\in \mathbb{Z}_{>0}$, and one usually studies the asymptotics of this quantity as $r,n \to \infty$ and $r/n\to 0$ or $1$. The role of the last passage time is played by the free energy (i.e., the log of the partition function) in positive temperature models; a precise formulation in the setup of the inverse-gamma polymer is given later.

Following the experimental and numerical studies by physicists \cite{Sing-2015, Take-2012, Take-Sano-2012}, and precise conjectures formulated by Ferrari and Spohn in \cite{Ferr-Spoh-2016}, a host of rigorous works have studied this problem in the context of the zero temperature exactly solvable models (particularly exponential LPP) \cite{timecorriid, timecorrflat, ferocctimecorr,fer-occ-2022, Ferr-Spoh-2016}. In the positive temperature setting, the problem has been studied for the KPZ equation in \cite{KPZcorr}. We shall not get into a detailed discussion of the existing literature; rather we refer the interested reader to the introduction of \cite{bas-sep-she-23} for a comprehensive discussion of available results and different methods employed to establish them. Three special initial conditions have been investigated in particular: (i) droplet initial condition, i.e., point-to-point passage time, (ii) flat initial condition, i.e., line-to-point passage time and (iii) stationary initial condition where $f$ is a certain two-sided random walk which is (increment)-stationary under the growth dynamics (ratio- stationary in the positive temperature setup; precise definitions later). More general (deterministic as well as random) initial conditions are also of interest but those are not very well understood even in the zero temperature case. Motivated by this, in the upcoming subsection, we discuss our contribution and provide a concise overview of the relevant literature.

\subsection{Our contribution and related literature}

As alluded to above, our focus in this paper is KPZ growth with general random initial conditions, which include the stationary ones. 
The only rigorous results for the time correlation study with the generic random initial condition so far 
appear in \cite{ferocctimecorr} in the zero-temperature setting. The generality of the random initial conditions studied there allows for the multiplication of the stationary initial condition by any non-negative constant, i.e.~``$f = C \times \text{the stationary initial condition}$" for any $C \geq 0$. 
The statistic studied in \cite{ferocctimecorr} is the limit $\rho(\tau) = \lim_{n \to \infty} \Corr(T^f_n, T^f_{\tau n})$ for $\tau\in (0,1)$. It was shown that for these random initial conditions $1-\rho(\tau)=\Theta((1-\tau)^{2/3})$ as $\tau\to 1$. Note that $\tau$ going to $1$ means the two times, $r$ and $n$,  are close relative to their sizes. In this scenario, the time correlation for the random, droplet and flat initial conditions all stem from similar reasons.

In contrast, in scenarios where 
$r$ is small compared to $n$, the underlying mechanisms governing time correlations differ fundamentally for random, droplet, and flat initial conditions. Here, 
the free energy at time
$r$ feels the effects from various initial conditions because it is much smaller compared to $n$. A part of our main contributions lies in establishing upper and lower bounds on the time correlation when two endpoints are far apart, for the random initial conditions. This aspect has not been previously known in any setting.

Expanding upon the random initial conditions studied in \cite{ferocctimecorr}, our work considers a very broad class of random initial conditions which satisfy certain ``diffusive growth" conditions. For this class of initial conditions, and for $r,n$ \textit{large but finite}, and for $r/n$ close to $0$ or close to $1$, we establish up to constant upper and lower bounds for the correlation establishing the same exponents as described above, see Theorem \ref{thm_r_large} and Theorem \ref{thm_r_small}. For the special case of the stationary initial condition, in the small $r$ regime, we also provide a different argument for the upper bound in Section \ref{s:duality} using duality which leads to a quantitatively better estimate.

Furthermore, the reliance of our arguments on the integrable nature of the models is rather weak, and this in turn, has enabled us to study a broad spectrum of diffusive initial conditions.  We rely only on the curvature of the limit shape and the one-point moderate deviation estimates for the point-to-point free energy (last passage time in the zero temperature setting). For the zero temperature set-up, the one-point estimates can be obtained either by analysis of Fredholm determinant formulae \cite{BFP12} (see also \cite{inc_up, inc_low, BXX01} for similar estimates in Poissonian and Geometric LPP) or using random matrix methods \cite{led-rid-2021}. More recently, these estimates have been obtained using softer stationary LPP/ coupling techniques \cite{cgm_low_up, rtail1}. For various positive temperature models, including the inverse-gamma polymer, these have recently been obtained in \cite{bas-sep-she-23, Emr-Jan-Xie-22-, OC_tail, Lan-Sos-22-a-} using the stationary polymer techniques; as far as we know these are the only moderate deviation results in the positive temperature discrete or semi-discrete models established so far; similar results are also known for the KPZ equation using Gibbsian line ensemble techniques (see, e.g.\ \cite{CG20}). Using these estimates, the time correlation in point-to-point inverse-gamma polymer (i.e., droplet initial condition) has been studied in \cite{bas-sep-she-23} by the current authors together with Sepp{\"a}l{\"a}inen and the same exponents as described above were obtained. Although the exponents are the same, and the ingredients are broadly similar, the arguments in this paper are largely different from \cite{bas-sep-she-23} for the case $r/n\ll 1$. As mentioned before, the contributing factors to the correlation in both cases are the same in the $r/n\to 1$ regime, and the arguments are rather similar. The main technical contribution of this paper is in the estimates in the  $r\ll n$ regime. When $r\ll n$, although the correlation in both the droplet and ``diffusive" initial conditions decay like $(r/n)^{1/3}$; the mechanisms are very different. In the former case, the correlation comes from the fluctuation in the bulk weights in the region between the lines $x+y=0$ and $x+y=2r$ while in the latter case, the primary contribution to the correlation is the fluctuation of the initial condition. Establishing this requires a different set of geometric arguments which is the heart of the new technical content of this paper. Notice that we work here in the annealed setting, that is we also average over the initial condition; one would expect a much smaller correlation (for $r\ll n$) for a fixed realization of the initial condition; see \cite{timecorrflat} for a discussion regarding this. 

We wrap up this section with some further comments relating our work to the relevant results of \cite{ferocctimecorr}. The argument in \cite{ferocctimecorr} is very different from the current article; the proof of the stationary correlation therein relies on the stationarity of the model as well as the convergence of the one point distribution to the Baik-Rains distribution in that case. The $\tau\to 1$ asymptotics for constant multiples of stationary initial condition considered there depended on the convergence of the passage time profiles with different initial conditions to Airy processes. Although one point convergence of the point-to-point free energy to the GUE Tracy-Widom distribution is known for the inverse-gamma polymer, finer results such as convergence to Airy processes appears to not have been established yet, therefore the approach from \cite{ferocctimecorr} does not seem applicable in the inverse-gamma case, for non-stationary initial conditions. For the stationary model, although we could not find the Baik-Rains convergence result in the literature in the inverse-gamma case; it is likely that one could obtain this result by arguing along the same lines as \cite{bk-oc} where the Baik-Rains limit was established for the O'Connell-Yor polymer. If so, one could prove an asymptotic version of our results, using the approach of \cite{ferocctimecorr} in the stationary case. Beyond the stationary case, even for the zero temperature setup, our results apply to a broader class of initial conditions, give finite size estimates, and also complete the picture by establishing the correlation for the $r\ll n$ case. For the case $\tau\to 1$ and for the subclass of initial conditions considered in \cite{ferocctimecorr}, however, our estimates are quantitatively somewhat weaker than the limiting estimates there. 

We now move towards precise definitions and formulation of the main results in this article.

\subsection{Polymer models on $\Z^2$ with a given initial condition}\label{def_poly_ad}

First, we define the bulk partition function without an initial condition.  Let $\mathbb{X}_{{\bf a}, {\bf b}}$ denote the collection of up-right paths between ${\bf a}$ and ${\bf b}$. Paths ${\bf x}_{\bbullet} \in\mathbb{X}_{{\bf a}, {\bf b}}$ are of the form ${\bf x}_0 = {\bf a}$ and ${\bf x}_{|{\bf a} - {\bf b}|_1} = {\bf b}$. We attach non-negative weights to the vertices $\{Y_{\bf z}\}_{{\bf z} \in \mathbb{Z}^2}$. If $\mathbb{X}_{{\bf a}, {\bf b}}$ is non-empty, define 
$$\wt Z_{{\bf a}, {\bf b}}= \sum_{{\bf x}_{\bbullet} \in \mathbb{X}_{{\bf a}, {\bf b}}} \prod_{i=1}^{|{\bf a} - {\bf b}|_1} Y_{{\bf x}_i}.$$
By our convention, we do not pick up any weights at the starting point of the paths ${\bf x}_0 = {\bf a}$. Thus, we will define $\wt Z_{{\bf a}, {\bf a}} = 1$. And if $\mathbb{X}_{{\bf a}, {\bf b}}$ is empty, define $\wt Z_{{\bf a}, {\bf b}} = 0$.

Next, we shall define a polymer model with an initial condition $W$ on the anti-diagonal line through the origin $\mathcal{L}_{\bf 0} = \{(j, -j): j \in \mathbb{Z}\}$. Let $\{W_{k}\}_{k\in \Z}$ denote a sequence of non-negative weights ($W_{k}$ should be interpreted as the boundary condition at $(k,-k)\in \mathcal{L}_{\bf 0}$) with $W_0=1$. We now define the partition function for the polymer to an endpoint ${\bf v}$ above $\mathcal{L}_{\bf 0}$, denoted $Z^{W}_{\bf v}$ by   
$$Z^{W}_{{\bf v}} = \sum_{k\in \mathbb{Z}} W_k \cdot \wt{Z}_{(k, -k), {\bf v}}.$$
The \textit{free energy} of the polymer between to  ${\bf v}$ with initial condition $W$ is then defined as $\log Z^W_{{\bf v}}$. Clearly, the above definition can also obviously be extended to initial conditions defined on $\mathcal{L}_{\bf a}$ for any ${\bf a}\in \Z^2$. \footnote{It is a common practice to define initial conditions in terms of a class of \textit{boundary weights} on the edges of  $\mathcal{S}_0$, the bi-infinite down-right staircase path $\{\dots, - {\bf e}_1 +\bf e_2,  - \bf e_1,{\bf 0} ,  \bf e_1, \bf e_1 - e_2, \dots \}$, but we shall not make much use of this formulation {in the first part of our paper. This definition will be delayed until Section \ref{stat_poly}.}}

Let $\mathbb{X}^\textup{line}_{{\bf v}}$ to denote the collection of up-right paths from $\mathcal{L}_{\bf 0}$ to ${\bf v}$.
Fix a directed path ${\bf x}_{\bbullet} \in \mathbb{X}^\textup{line}_{{\bf v}}$ and suppose $k\in \Z$ is such that ${\bf x}_{\bbullet} \in \mathbb{X}_{(k,-k), {\bf v}}$. The partition function over this particular path ${\bf x}_\bbullet$ is defined by 
$$Z^W_{{\bf v}}({\bf x}_\bbullet) = W_k \cdot \prod_{i=1}^{|{\bf v}|_1} Y_{{\bf x}_i}.$$
Similarly, for a subset $\mathfrak{A}\subset \mathbb{X}^\textup{line}_{{\bf v}}$, define
$Z^W_{{\bf v}}(\mathfrak{A}) = \sum_{{\bf x}_{\bbullet} \in \mathfrak{A}} Z^{W}_{{\bf v}}({\bf x}_\bbullet).$
Then, the quenched polymer measure is defined by $Q^W_{{\bf v}}\{\mathfrak{A}\} =  \frac{Z^W_{{\bf v}}(\mathfrak{A})}{Z^W_{{\bf v}}}$,
which is a probability measure on $\mathbb{X}^\textup{line}_{{\bf v}}$.

\subsection{Assumptions on bulk weights and initial condition} \label{as_weights}
Fix an anti-diagonal line $\mathcal{L}_{\bf a} = {\bf a} + \mathcal{L}_{\bf 0}$, and let us set $\mathcal{L}_{\bf a}^> = \bigcup_{{\bf b }\in \mathcal{L}_{\bf a}}({\bf b} + \mathbb{Z}^2_{>0})$. 
As already mentioned we shall work with the inverse-gamma polymer with an initial condition on $\mathcal{L}_{\bf 0}$, i.e., our boundary weights $\{Y_{\bf z}\}_{{\bf z } \in \mathcal{L}_{\bf 0}^{>}}$ will be distributed as 
 i.i.d.\ inverse-gamma random variables with shape parameter $\mu$. Recall the density function of the inverse-gamma distribution is defined by 
\begin{equation} \label{ig4}
f_\mu(x) = {\displaystyle {\frac {1}{\Gamma (\mu )}}x^{-\mu -1}e^{-{\frac {1}{x}}}}
\quad \text{for } x>0. 
\end{equation}
The shape parameter $\mu\in(0,\infty)$  plays the role of temperature in this polymer model, which will be fixed in this paper.

\noindent
\textbf{Assumptions on the initial condition:}
We shall assume that the initial condition $W = \{W_i\}_{i\in \mathbb{Z}}$ will be such that $\log W$ is a two-sided random walk with $\log W_0=0$.

\medskip
\noindent
\textbf{Assumption A:} 
The collection of random variables $\{\log W_{i}-\log W_{i-1}\}_{i\in \Z}$ are independent, have mean $0$ with the following properties:
\begin{enumerate}
    \item[(i)] $X_{i}:=\log W_{i}-\log W_{i-1}$ are uniformly sub-exponential 
    i.e., there exist parameters $\lambda_1,K_1>0$ such that $\log \mathbb E (\exp(\lambda |X_{i}-\mathbb E X_{i}|)) \le K_1\lambda^2$ for $\lambda\in [-\lambda_1,\lambda_1]$. 
    \item[(ii)]  $\inf_{i} \Var X_{i}>0$.
\end{enumerate}
In fact, we shall use weaker assumptions in our proofs, but we recorded \textbf{Assumption A} above since it is more straightforward to verify, and already contains many interesting cases. 
Now, let us introduce the specific assumptions employed in our paper, which shall be referred to as \textbf{Assumption B}. \textbf{Assumption B} comprises of three distinct components: \textbf{B1}, \textbf{B2}, and \textbf{B3}.

To simplify the notation, we will introduce a collection of random variables denoted as $\{X_i\}_{i \in \mathbb{Z}} = \{\log W_{i}-\log W_{i-1}\}_{i\in \Z}$.
\begin{itemize}
\item \textbf{Assumption B1:} 

\textit{Right tail decay:} There exist $C_1, N_0, t_0$ such that for each $N \geq N_0$ and $t\geq t_0$
\begin{equation}\label{up_bd}
\mathbb{P}\Big(\max_{-N \leq k \leq N} \log W_{k} \geq t\sqrt{N}\Big) \leq e^{-C_1t}.
\end{equation}

\item \textbf{Assumption B2:}
\begin{enumerate}[(i)]
\item  \textit{$\rho$-mixing:}
For each $A,B\subset \Z$ with $\textup{dist}(A, B) > \ell\geq 1$, we have 
\begin{equation}\label{mix}
 \Big|\Corr(f(X_i: i\in A), g(X_j : j\in B))\Big| \leq \frac{1}{\ell^{100}};
\end{equation}
for all square-integrable functions $f$ and $g$, i.e., the sequence $\{X_{i}\}_{i\in \Z}$ is fast $\rho$-mixing.\footnote{Note that the $\rho$-mixing is usually defined for stationary sequences, we consider the obvious extension of this definition in our case. For a precise definition of $\rho$-mixing coefficient and its relationship with other notions of mixing, see \cite{brad-05}.}

\item \textit{Left tail decay:} There exist $C_1, N_0, t_0$ such that for each $N \geq N_0$ and $t\geq t_0$
\begin{equation}\label{low_tail}
\mathbb{P}\Big(\log W_{N} \leq -t\sqrt{N}\Big) \leq e^{-C_1t}.
\end{equation}
\end{enumerate}

\item \textbf{Assumption B3:}

For $N\in \mathbb{Z}_{>0}$, let $\mathcal{F}_{r}=$ denote the $\sigma$-algebra generated by $\sigma(\{X_{i}\}_{i\notin [1, r]})$. 
\begin{enumerate}[(i)]

\item \textit{Conditional right tail decays:}
There exist $C_1, r_0, t_0$ such that for each $r \geq r_0$, $N\geq r$ and $t\geq t_0$
\begin{equation}\label{cond_up_bd}
\mathbb{P}\Big(\mathbb{P}\big(\max_{0\leq k \leq N} \log W_{k} \geq t\sqrt{N}\big|\mathcal{F}_r\big) \leq e^{-C_1t} \Big) \geq 1-\frac{1}{\log r}.
\end{equation}

\item \textit{Conditional left tail decay:} There exist $C_1, r_0, t_0$ such that for each $r \geq r_0$, $N\geq r$ and $t\geq t_0$
\begin{equation}\label{cond_low_tail}
\mathbb{P}\Big(\mathbb{P}\big(\log W_{N} \leq -t\sqrt{N} \big|\mathcal{F}_r\big) \leq e^{-C_1t} \Big) \geq 1-\frac{1}{\log r}.
\end{equation}

\item  \textit{Conditional Variance bound:} 
There exists $C_1, r_0 >0$, such that for all $r\geq r_0$, 
\begin{equation}\label{as_var}
\mathbb{P}\Big(\Var(\log W_r | \mathcal{F}_r) \geq C_1 r\Big) \geq 1-\frac{1}{\log r}.
\end{equation}
almost surely.

\item  \textit{Right tail lower bound:} For each fixed $t_0$ sufficiently large, there exist $N_0, \epsilon_0$ (depending on $t_0$) such that for each $i\in \Z$ and for each integer $N$ with $|N| \geq N_0$,
\begin{equation}\label{low_bd}
\mathbb{P}\Big( \log W_{N+i}-\log W_{i} \geq t_0\sqrt{|N|}\Big) \geq \epsilon_0 >0.
\end{equation}

\item \textit{FKG inequality:} The random sequence $\mathcal{X}:=\{\hat{X}_{i}\}_{i\in \Z}$ defined by $\hat{X}_{i}=X_{i}$ for $i\ge 1$ and $\hat{X}_{i}=-X_{i}$ otherwise, satisfies the FKG inequality, i.e., for any two square-integrable coordinate wise increasing real-valued functions $f$ and $g$,\footnote{{In our applications, we only use this hypothesis in the case when the arguments of $f$ is a subset of the arguments of $g$. Thus, it does not necessarily preclude the case of Busemann increment initial condition for general KPZ models where it is predicted that the disjoint Busemann increments along a down-right path are negatively correlated \cite{neg_buse}.}}
\begin{equation}\label{FKG}
\Corr(f(\mathcal{X}),g(\mathcal{X}))\geq 0.
\end{equation}
Further, we also require that for all $k\in \mathbb{Z}_{>0}$, the conditional measure of $\mathcal{X}$, restricted to the coordinates $\{1,2,\ldots, k\}$ conditional on the remaining coordinates also satisfies the FKG inequality. \footnote{It might appear that this assumption is rather strong and difficult to verify,
but observe that if $X_{i}$ are independent then this is immediate. We shall also provide examples beyond the independent case which satisfy this and the other hypotheses; see Remark \ref{example}.}
\end{enumerate}
\end{itemize}

Note that the assumptions \eqref{up_bd}, \eqref{low_tail} and \eqref{low_bd} give bounds on non-trivial diffusive growth of $\log W$ (and \eqref{cond_up_bd}, \eqref{cond_low_tail} and \eqref{as_var} gives conditional versions of the same). Many of the assumptions become trivial if the independence of $X_{i}$ is assumed and it is not hard to see that \textbf{Assumption A} implies \textbf{Assumption B}; see Proposition \ref{aimpliesa1}. Note also that, \textbf{Assumption A}  includes that case of stationary inverse-gamma polymer with parameter $\rho=\mu/2$; see Appendix \ref{ver_weights} for a proof of this and Section \ref{stat_poly} for more details on the stationary polymer model. See Remark \ref{example} for some examples of initial conditions that satisfy \textbf{Assumption B} but not \textbf{Assumption A}. We shall see later that not all parts of the hypotheses are used in all our arguments. Also, some of the hypotheses are not quantitatively optimal; see Remark \ref{ext} for a discussion regarding this.

\subsection{Main results on time correlation of the free energies}
We now state our main results for the temporal correlations in the inverse-gamma polymer with general initial conditions. Recall that the correlation coefficient  of two random variables $\zeta$ and $\eta$ is defined by 
\[  \textup{$\mathbb{C}$orr}(\zeta, \eta)=
\frac{\Cov(\zeta, \eta)}{\Var(\zeta)^{1/2}\,\Var(\eta)^{1/2}}
=\frac{\mathbb E[\zeta \eta] - \mathbb E[\zeta]  \cdot  \mathbb E[\eta]}
{\mathbb E[ \,|\zeta-\mathbb E\zeta|^2\,]^{1/2}\,\mathbb E[ \,|\eta-\mathbb E\eta|^2\,]^{1/2}}.
\]
Our main result establishes the time correlation exponents $1/3$ and $2/3$ for a pair of free energies in the inverse-gamma polymer with the above initial conditions, based on on the separation of their endpoints. 

\begin{theorem}\label{thm_r_large}
Consider the inverse-gamma polymer with boundary condition $W$ on $\mathcal{L}_{\bf 0}$ that satisfies \textbf{Assumption B1}.
There exist positive constants $C_1, C_2, c_0, N_0$ such that, whenever $N\geq N_0$ and $N/2 \le r \leq N-c_0$,
we have 
$$ 1-C_1\Big(\frac{N-r}{N}\Big)^{2/3} \leq \textup{$\mathbb{C}$orr}\Big(\fgs_{(r,r)}, \fgs_{(N,N)}\Big) \leq 1-C_2\Big(\frac{N-r}{N}\Big)^{2/3}.$$ 
\end{theorem}

\begin{theorem}\label{thm_r_small}  
Consider the inverse-gamma polymer with boundary condition $W$ on $\mathcal{L}_{\bf 0}$ that satisfies \textbf{Assumptions B1} and \textbf{B2}. There exist positive constants $C_3, c_0,  N_0$ such that, whenever $N\geq N_0$ and $c_0 \leq r \leq N/2$,  we have 
$$  \textup{$\mathbb{C}$orr}\Big(\fgs_{(r,r)}, \fgs_{(N,N)}\Big) \leq C_3\Big(\frac{r}{N}\Big)^{1/3}.$$ 
Furthermore, if $W$ satisfies  \textbf{Assumptions B1}, \textbf{B2} and \textbf{B3}, we have a matching lower bound. There exist positive constants $C_4, c_0,  N_0$ such that, whenever $N\geq N_0$ and $c_0 \leq r \leq N/2$,  we have 
$$ C_4\Big(\frac{r}{N}\Big)^{1/3} \leq \textup{$\mathbb{C}$orr}\Big(\fgs_{(r,r)}, \fgs_{(N,N)}\Big).$$ 
\end{theorem}

We shall informally refer to the setups of Theorem \ref{thm_r_large} and Theorem \ref{thm_r_small} by the large $r$ regime and the small $r$ regime respectively. Notice that since \textbf{Assumption A} implies  \textbf{Assumption B} (Proposition \ref{aimpliesa1}), both Theorem \ref{thm_r_large} and Theorem \ref{thm_r_small} remain valid for initial conditions satisfying \textbf{Assumption A} as well. We record the following corollary. 

\begin{corollary}
Consider the inverse-gamma polymer with boundary condition $W$ on $\mathcal{L}_{\bf 0}$ that satisfies \textbf{Assumption A}.
There exist positive constants $C_1, C_2, c_0, N_0$ such that,
\begin{itemize}
\item for $N\geq N_0$ and $N/2 \le r \leq N-c_0$,
$$1-C_1\Big(\frac{N-r}{N}\Big)^{2/3} \leq \textup{$\mathbb{C}$orr}\Big(\fgs_{(r,r)}, \fgs_{(N,N)}\Big) \leq 1-C_2\Big(\frac{N-r}{N}\Big)^{2/3},$$ 
\item for $N\geq N_0$ and $c_0 \leq r \leq N/2$,  we have 
$$ C_2\Big(\frac{r}{N}\Big)^{1/3} \leq \textup{$\mathbb{C}$orr}\Big(\fgs_{(r,r)}, \fgs_{(N,N)}\Big) \leq C_1\Big(\frac{r}{N}\Big)^{1/3}.$$ 
\end{itemize}
\end{corollary}

As mentioned before, the initial conditions satisfying \textbf{Assumption A} include the stationary initial condition with parameter $\rho=\mu/2$, (see Section \ref{stat_poly} for precise definitions). A different argument giving the upper bound in the small $r$ regime for the special case of this stationary initial condition using duality is given in Theorem \ref{stat_r_small}. We finish this subsection with a few remarks regarding examples and potential extensions of our results.

\begin{remark}
    As mentioned already, our arguments carry over verbatim to the case of exponential last passage percolation on $\Z^2$ with a class of ``diffusive" initial conditions upon changing the free energy to the last passage time and making other appropriate changes to the definitions. In fact, many of the arguments become simpler in this case, since one can consider geodesics instead of the quenched polymer measures. To avoid repetition, we shall not write down the details in this case, but we shall formulate a precise statement in this setting; see Theorem \ref{thm_r_large_lpp} and Theorem \ref{thm_r_small_lpp}. 
\end{remark}

\begin{remark}
    As mentioned above, \textbf{Assumption A} includes the case of stationary initial condition with parameter $\mu/2$. This is related to the fact that the corresponding characteristic direction, in this case, is the diagonal direction $(1,1)$. One should be able to extend our results to the stationary initial conditions with parameters $\rho\in (0,\mu), \rho\neq \mu/2$ (see Section \ref{stat_poly} for precise definitions) provided the endpoint of the polymers varies along the corresponding characteristic direction. In this case, our arguments should apply to initial conditions satisfying \textbf{Assumption A}, but with an appropriate drift (i.e., $\mathbb E X_{i}=- \Psi_0(\mu-\rho)+ \Psi_0(\rho)$). However, we have not attempted to verify all the details in this case. 
\end{remark}

\begin{remark}
    \label{example}
    \textbf{Assumption B} contains many more interesting cases beyond the stationary one; we discuss a few examples. 
    \begin{itemize}
        \item First consider the case when $X_i$'s are independent. It is clear that independent sequences satisfy parts assumptions \eqref{mix}, \eqref{FKG} from \textbf{Assumption B}, and the other parts follow under conditions on the tails of $\{X_i\}_{i\in \mathbb{Z}}$. In particular, let $\{\sigma_{i}\}_{i\in \Z}$ denote a positive sequence uniformly bounded away from $0$ and $\infty$. Denoting by $\hat{X}^{\mu/2}_{i}$ the increment sequence in the stationary case, it is easy to check that $X_{i}:=\sigma_{i}\hat{X}^{\mu/2}_{i}$ satisfies \textbf{Assumption B}. Compare this to \cite{ferocctimecorr}, where the case $\sigma_{i}\equiv \sigma$ was considered in the large $r$ regime.
\item Consider next the ``mixtures" between deterministic and stationary increments. We can consider (a) $X_{i}=\hat{X}^{\mu/2}_{i}$ if $i$ is even and 0 otherwise, or (b) $X_{i}=1,-1,\hat{X}^{\mu/2}_{i}$ if $i=0,1,2$ modulo $3$ respectively. It is easy to check that $X_i$'s defined as above satisfy all the hypotheses in \textbf{Assumption B}. 
\item Finally let us consider $X_{i}$'s which are not independent; the hardest assumption to verify in this case is \eqref{FKG}. Although there should be many examples satisfying these hypotheses, we shall restrict to Gaussian vectors $X_{i}$ since it is a well-known fact that such a vector satisfies the FKG inequality if and only if the pairwise correlations are non-negative. For a sequence $\{U_{i}\}_{i\in \Z}$ of i.i.d.\ standard Gaussians, define $X_{i}=U_{i}+U_{i+1}$ for all $i\ge 0$, and $X_{i}=U_{i}+U_{i-1}$ for $i<0$. Since this sequence is  $1$-dependent, it satisfies the mixing hypothesis \eqref{mix}. By calculating pairwise correlations (both unconditional and conditional) using standard formulas for Gaussian vectors, together with the fact described above, one can easily verify parts \eqref{as_var} and \eqref{FKG} of the hypotheses. Parts \eqref{up_bd}, \eqref{low_tail},  \eqref{cond_up_bd}, \eqref{cond_low_tail}, and \eqref{low_bd} are consequences of standard estimates and are easy to verify. Therefore, the example described above (and similar ones) satisfy \textbf{Assumption B}. 
\end{itemize}
\end{remark}

\begin{remark}
\label{ext}
    The quantitative estimates in several parts of \textbf{Assumption B} are not optimal. The reader can check that the exponent 100 in \eqref{mix} can be replaced by a smaller number and the exponential bound in \eqref{up_bd} can be replaced by a sufficiently high degree polynomial without making any major changes to the arguments. 
    It is possible that \eqref{up_bd}, \eqref{low_bd} and \eqref{as_var} can be removed by replacing the mixing condition in \eqref{mix} with a stronger hypothesis, but we have not managed to locate a result that fits our set-up and hence have not attempted to find a more concise replacement for these hypotheses.
\end{remark}

\subsection{Organization of the paper}
The rest of the paper is organized as follows. In Section \ref{s:prelim}, we introduce basic notations and recall certain basic estimates about the point-to-point polymer from the literature. Section \ref{est_poly_bdry} contains several important estimates for polymers with an initial condition whose proofs are provided later. Section \ref{sec_r_large} and Sec \ref{sec_r_small} prove Theorems \ref{thm_r_large} and \ref{thm_r_small} respectively. Section \ref{sec_r_small} is the heart of the new technical contributions of this paper; Sections \ref{r_small_up} and \ref{r_small_low} provide the proofs of the upper bound and lower bound in Theorem \ref{thm_r_small} respectively. Section \ref{stat_poly} introduces the stationary polymer and recalls some of the known estimates about it. Section \ref{s:duality} provides an alternative proof of the upper bound in Theorem \ref{thm_r_small} for the stationary polymer using duality; we believe this is an argument of independent interest. Section \ref{s:est_proof} provides the proof of the technical estimates from Section \ref{est_poly_bulk} and Section \ref{est_poly_bdry}. Finally, Section \ref{s:lpp} precisely formulates time correlation results for exponential LPP with diffusive initial conditions and describes how the arguments in the inverse-gamma polymer case carry over to the zero temperature set-up. 

\subsection*{Acknowledgements}  RB was partially supported by a MATRICS grant (MTR/2021/000093) from SERB, Govt.~of India, DAE project no.~RTI4001 via ICTS, and the Infosys Foundation via the Infosys-Chandrasekharan Virtual Centre for Random Geometry of TIFR. XS was partially supported by the Wylie research fund from the University of Utah.

\section{Preliminaries}
\label{s:prelim}
\subsection{Notation}
Fix ${\bf a}\in \mathbb{Z}^2$, let
$\mathcal{L}_{{\bf a}}=\{{\bf a}+(j,-j): j\in\Z\}$.
For $k\in \mathbb{R}_{\geq 0}$, define
$\mathcal{L}^k_{{\bf a}} = \{{\bf x}\in \mathcal{L}_{{\bf a}} : |{\bf x}-{\bf a}|_\infty \leq k  \}.$
For ${\bf a}, {\bf b}\in \mathbb{Z}^{2}$ and  $k\in \mathbb{R}_{\geq 0}$,  
$R_{{\bf a}, {\bf b}}^k$ denotes the parallelogram spanned by the four corners ${\bf a} \pm (k,-k)$ and ${\bf b} \pm (k,-k)$. For any subset $A\subset \mathbb{Z}^2$ and any $\triangle \in \{>, \geq, <,  \leq\}$, define
$A^\triangle = \bigcup_{x\in A} (x + \mathbb{Z}^2_{\triangle 0})$.

For a collection of directed paths $\mathfrak{A}$, let   $Z(\mathfrak{A})$   be the free energy obtained by summing over all the paths in $\mathfrak{A}$.
For $A, B \subset \mathbb{R}^2$, let 
$Z_{A, B}$ denote the partition function obtained by summing over all directed paths starting from integer points 
$$A^\circ = \{{\bf a}\in \mathbb{Z}^2 : {\bf a}+[0,1)^2 \cap A \neq \emptyset\}$$ and ending in 
$$B^\circ = \{{\bf b}\in \mathbb{Z}^2 : {\bf b}+[0,1)^2 \cap B \neq \emptyset\}.$$

Let $A, B \subset \mathbb{R}^2$, ${\bf c},{\bf d} \in \mathbb{Z}^2$ and $h> 0$. We define three specific partition functions: 
\begin{align*}
   Z_{A, B}^{\textup{in}, R^{h}_{{\bf c}, {\bf d}}} &= \text{sum over directed paths from $A$ to $B$  contained inside the parallelogram  $R^{h}_{{\bf c}, {\bf d}}$, } \\
   Z_{A, B}^{\textup{exit}, R^{h}_{{\bf c}, {\bf d}}}&= \text{sum over directed paths from $A$ to $B$  that exit at least one of } \\
   &\quad\; \text{the sides of  $R^{h}_{{\bf c}, {\bf d}}$ parallel to ${\bf d}-{\bf c}$. }\\
    Z_{A, B}^{\textup{out}, R^{h}_{{\bf c}, {\bf d}}}&= \text{sum over directed paths from $A$ to $B$  that avoid $R^{h}_{{\bf c}, {\bf d}}$.}
\end{align*}


Integer points on the diagonal or the anti-diagonal are abbreviated as $a=(a,a)$ and $\olsi{a} = (a,-a)$. Common occurrences of this include  $Z^W_{(r,r)} = Z^W_r$, $\wt Z_{\olsi{r},N}=\wt Z_{(r,-r), (N,N)}$,  $\wt Z_{{\bf p}, N+\olsi{k}}=\wt Z_{{\bf p}, (N+k,N-k)}$,  $\mathcal{L}_a^k=\mathcal{L}_{(a,a)}^k$ and $R_{a, b}^k = R^k(a, b)=R_{(a,a), (b,b)}^k$.

Finally, we point out two conventions about constants and  integer rounding. First, generic positive constants will be denoted by $C, C',$ in the calculations of the proofs. They may change from line to line without change in notation. Second, we drop the integer floor function   to simplify  notation. For example, if we divide the line segment from $(0,0)$ to $(N, N)$ in $5$  equal pieces, we denote the free energy of the first segment by  $\log Z_{0, N/5}$ even if  $N/5$ is not an integer.

\subsection{The shape function for the bulk polymer}

When the positive weights $\{Y_{\bf z}\}_{{\bf z}\in \mathbb{Z}^2}$ are chosen as a collection of i.i.d.~positive random variables on some probability space $(\Omega, \mathbb{P})$. 
Under a mild moment assumption such as
$$\mathbb{E}\big[|\log Y_{\bf z}|^p\big] < \infty \quad  \textup{ for some }p>2,$$
a law of large numbers type result called the \textit{shape theorem}  holds for the free energy   (Section 2.3 of \cite{Jan-Ras-20-aop}):  
there exists a concave, positively homogeneous and deterministic continuous function $\Lambda: \mathbb{R}^2_{\geq 0} \rightarrow \mathbb{R}$ that satisfies 
\begin{equation}\label{shape}
\lim_{n\rightarrow \infty} \sup_{ |{\bf z}|_1 \geq n} \frac{|\log \wt{Z}_{0, {\bf z}} - \Lambda({\bf z})|}{|{\bf z}|_1} = 0 \qquad \textup{$\mathbb{P}$-almost surely}.
\end{equation}

For general i.i.d.~weights, regularity properties of $\Lambda$ such as strict concavity or differentiability, are conjectured to hold at least for all continuous weights.
There is a special case, first observed in \cite{poly2}, that if the i.i.d.\ weights  have the inverse-gamma distribution \eqref{ig4}, then $\Lambda$ can be computed explicitly. Using this fact, several estimates were done (only using calculus) in Section 3.2 of \cite{bas-sep-she-23}, and we recall one of the results from there. 

To simplify the notation, let $\Lambda_N= \Lambda((N,N))$.

\begin{proposition}[{\cite[Proposition 3.5]{bas-sep-she-23}}]\label{reg_shape}
There exist positive constants $C_1, N_0, \epsilon_0$ such that for each $N\geq N_0$ , $h\leq \epsilon_0N^{1/3}$ and each ${\bf p}\in \mathcal{L}_{N}^{h N^{2/3}}$, we have 
$$\big|\Lambda({\bf p}) - \Lambda_N\big| \leq C_1 h^2 N^{1/3}.$$
\end{proposition}

\subsection{Estimates for the bulk polymer}\label{est_poly_bulk}

We recall estimates from Section 3.3 of \cite{bas-sep-she-23}. Note that these estimates were originally proved for partition function $Z_{0, \bbullet}$ which included the weight at the starting point, i.e.  
$$\log Z_{0, \bbullet} =  \log Y_{{\bf 0}} + \log \wt{Z}_{0, \bbullet}.$$
However, because $Y_{{\bf 0}}\sim \textup{Ga}^{-1}(\mu)$, the exact same estimates also hold for $\wt{Z}_{0, \bbullet}$.

Recall the shape function for the bulk polymer  $\Lambda_N = \Lambda((N,N))$ from \eqref{shape}. 
Our first two propositions give the upper and lower bounds for the right tail.
\begin{proposition}\label{ptl_upper}
There exist positive constants $C_1, C_2, N_0$ such that for each $N \geq N_0$, $t\geq 1$ and $1\leq h \leq e^{C_1{\min\{t^{3/2}, tN^{1/3}\}}}$, we have 
$$\mathbb{P}\Big(\log \wt{Z}_{\mathcal{L}_0^{hN^{2/3}},\mathcal{L}_{N}}  - \Lambda_N \geq tN^{1/3}\Big) \leq 
e^{-C_2\min\{t^{3/2}, tN^{1/3}\}}.  
$$
\end{proposition}

\begin{proposition}\label{up_lb}
There exist positive constants $C_1, N_0, t_0, \epsilon_0$ such that for each $N\geq N_0$, $t_0 \leq t \leq \epsilon_0 N^{2/3}$, we have
$$\mathbb{P}(\fg_{0, N} - \Lambda_N \geq tN^{1/3}) \geq e^{-C_1 t^{3/2}}.$$
\end{proposition}

The next proposition gives the upper bound for the left tail.
\begin{proposition}\label{low_ub}
There exist positive constants $C_1, N_0 $ such that for each $N\geq N_0$, $t\geq 1$ , we have 
$$\mathbb{P}(\fg_{0,  N} - \Lambda_N \leq -tN^{1/3}) \leq e^{-C_1 \min\{t^{3/2}, tN^{1/3}\}}.$$
\end{proposition}


Then, we have a variance bound for the free energy in the bulk, which follows from the estimates for the left and right tails above. 
\begin{proposition}\label{var_i}There exists positive constants $C_1, C_2, N_0$ such that for each $N\geq N_0$, we have 
$$ C_1 N^{2/3} \leq \Var(\log \wt{Z}_{0,N}) \leq C_2N^{2/3}.$$
\end{proposition}

The next two propositions summarize the loss of free energy from paths having too much transversal fluctuation.

\begin{proposition}\label{trans_fluc_loss}
There exist positive constants $C_1, C_2, N_0$ such that for each $N\geq N_0$, $t\geq 0$ and $s\geq 0$ we have
$$\mathbb{P}\Big(\fg_{\mathcal{L}^{sN^{2/3}}_0, \mathcal{L}_{N}\setminus\mathcal{L}_{N}^{(s+t)N^{2/3}}}    - \Lambda_N \geq -C_1 t^2N^{1/3} \Big) \leq e^{-C_2t^3}.$$
\end{proposition}

\begin{proposition}\label{trans_fluc_loss3}
There exist positive constants $C_1, C_2,   N_0$ such that for each $N\geq N_0$, $1 \leq t \leq  N^{1/3}$ and $0< s < e^{t}$,  we have
$$\mathbb{P}\Big(\log \wt{Z}^{\textup{exit}, R^{(s+t)N^{2/3}}_{0, N}}_{\mathcal{L}_{0}^{sN^{2/3}}, \mathcal{L}_N^{sN^{2/3}}} - \Lambda_N \geq -C_1 t^2N^{1/3}\Big) \leq e^{-C_2t^{3}}.$$
\end{proposition}

The proposition below shows when constrain our paths to a rectangle that obeys the KPZ scale, the free energy will not be too small. 

\begin{proposition}\label{wide_similar}
For each positive $a_0$, there exist positive constants $C, t_0$ such that for each $0< \theta \leq  100$, there exists a positive constant $ N_0$ such that for each $N\geq N_0$, $t\geq t_0$ and  ${\bf p} \in \mathcal{L}_N^{a_0\theta N^{2/3}}$, we have
$$\mathbb{P}\Big(\log \wt{Z}^{\textup{in}, {\theta N^{2/3}}}_{0, {\bf p}} - \Lambda_N \leq -tN^{1/3}\Big) \leq  \tfrac{\sqrt{t}}{\theta}e^{-C\theta t}.$$
\end{proposition}

The next two propositions in this section describe the local  fluctuations of the free energy profile along an anti-diagonal line. The left tail estimate in Proposition \ref{min} is new, and its proof will appear in Section \ref{a:min}.  
\begin{proposition}\label{max_all_t}
There exist positive constants $t_0, N_0$ such that for each $N\geq N_0$, $t\geq t_0$, and each $a \in \mathbb{Z}_{\geq 0}$, we have
$$\mathbb{P}\Big(\log \wt{Z}_{0,\mathcal{L}^a_{N}}\ - \log \wt{Z}_{0, N} \geq t\sqrt{a}\Big) \leq 
e^{-t^{1/10}}.
$$ 
\end{proposition}

\begin{proposition}\label{min}
There exist positive constants $t_0, r_0$ such that for each $r\geq r_0$, $t\geq t_0$, and each $N \geq r$, we have
$$\mathbb{P}\Big(\min_{|k|\leq t^{1/20}r^{2/3}}\log \wt{Z}_{0,{N}+\olsi{k}}\ - \log \wt{Z}_{0, N} \leq -tr^{1/3}\Big) \leq 
e^{-t^{1/10}}.
$$ 
\end{proposition}

The last two propositions look at the local fluctuation around the starting point instead of the endpoint. Recall $\wt{Z}$ is defined to include the weight at the end point but not the starting point and we basically show this does not affect the local fluctuation. The proof of Proposition \ref{max_all_t1} below can be found in Section \ref{max_all_t1_proof}. Moreover, we omit the proof Proposition \ref{min1}, as it is similar to how we obtained Proposition \ref{max_all_t1} from \ref{max_all_t}.
\begin{proposition}\label{max_all_t1}
There exist positive constants $t_0, N_0$ such that for each $N\geq N_0$, $t\geq t_0$, and each $a \in \mathbb{Z}_{\geq 0}$, we have
$$\mathbb{P}\Big(\log \wt{Z}_{\mathcal{L}^a_{0}, N}\ - \log \wt{Z}_{0, N} \geq t\sqrt{a}\Big) \leq 
e^{-t^{1/10}}.
$$ 
\end{proposition}

\begin{proposition}\label{min1}
There exist positive constants $t_0, r_0$ such that for each $r\geq r_0$, $t\geq t_0$, and each $N\geq r$, we have
$$\mathbb{P}\Big(\min_{|k|\leq t^{1/20}r^{2/3}}\log \wt{Z}_{\olsi{k}, N}\ - \log \wt{Z}_{0, N} \leq -tr^{1/3}\Big) \leq 
e^{-t^{1/10}}.
$$ 
\end{proposition}

\subsection{Estimates for the polymer with boundary}\label{est_poly_bdry}
This section summarizes the results for the polymer with general initial conditions. The proofs of these results use the coupling with stationary polymer, thus they are postponed to Section \ref{s:est_proof}.

Recall the shape function for the bulk polymer $\Lambda((N,N)) = \Lambda_N$. 
For $k \in \mathbb{Z}$, let us define $\{\tau_{\mathcal{L}_N^s} = k\} = \bigcup_{{\bf v} \in \mathcal{L}_N^s}\mathbb{X}_{(k, -k), {\bf v}}$, which will often be referred to as  the collection of paths with \textit{exit time} $k$. The partition function over $\{\tau_{\mathcal{L}_N^s} = k\}$ is given by 
$$Z^W_{\mathcal{L}_N^s} (\tau_{\mathcal{L}_N^s} = k)= Z^W_{\mathcal{L}_N^s} (\tau = k)=  W_k \cdot \wt{Z}_{(k, -k), {\mathcal{L}_N^s}},$$
where we omit the subscript under $\tau$ when it is clear from the partition function. The quenched path measure $Q^W_{\mathcal{L}_N^s}\{\tau = k\} = \frac{Z^W_{\mathcal{L}_N^s} (\tau = k)}{Z^W_{\mathcal{L}_N^s}}.$

\begin{proposition}\label{exit_est}
Fix $\epsilon > 0$ and suppose boundary condition $W$ satisfies assumption \eqref{up_bd}. There exist positive constants $C_1, C_2 ,  N_0, t_0$ such that for each $N\geq N_0$ and $t\geq t_0$, we have
$$\mathbb{P}\Big(\fgs_{\mathcal{L}_N^{tN^{2/3}}}(|\tau| > (1+\epsilon) tN^{2/3}) - \Lambda_N \geq -C_1t^2 N^{1/3}\Big) \leq e^{-C_2t^{3/2}}.$$
\end{proposition}

\begin{proposition}\label{b_up_tail}
Fix $\epsilon >0$ and suppose boundary condition $W$ satisfies assumption \eqref{up_bd}. There exist positive constants $C_1, N_0, t_0$ such that for each $N\geq N_0$, $t\geq t_0$, we have
\begin{enumerate}[(i)]
\item $\mathbb{P}\Big(\fgs_{ \mathcal{L}_{ N }^{tN^{2/3}}} - \Lambda_N \geq \epsilon t^2 N^{1/3}\Big) \leq e^{-C_1t^{3/2}}.$
\item$\mathbb{P}\Big(\fgs_{ \mathcal{L}_{ N }^{tN^{2/3}}} - \Lambda_N \geq \epsilon t N^{1/3}\Big) \leq e^{-C_1t^{1/2}}.$
\end{enumerate}
\end{proposition}

\begin{proposition}\label{b_low_tail}
Fix $\epsilon >0$ and suppose boundary condition $W$ satisfies assumption \eqref{up_bd}. There exist positive constants $C_1, N_0, t_0$ such that for each $N\geq N_0$, $t\geq t_0$, we have
$$\mathbb{P}\Big(\fgs_{ N} - \Lambda_N \leq -\epsilon t N^{1/3}\Big) \leq e^{-C_1\min\{t^{3/2}, t N^{1/3}\}}.$$
\end{proposition}

\begin{theorem}\label{var_s} Suppose boundary condition $W$ satisfies assumption \eqref{up_bd}. There exists positive constants $C_1, C_2, N_0$ such that for each $N\geq N_0$, we have 
$$C_1 N^{2/3} \leq \Var(\log {Z}^W_{N}) \leq C_2N^{2/3}.$$
\end{theorem}

\begin{theorem}\label{exit_q}
Fix $\epsilon > 0$ and suppose boundary condition $W$ satisfies assumption \eqref{up_bd}. There exist positive constants $C_1, C_2 ,  N_0, t_0$ such that for each $N\geq N_0$, $t\geq t_0$ and $0\leq s\leq tN^{2/3}$, we have
$$\mathbb{P}\Big(Q^W_{\mathcal{L}_N^{s}}(|\tau| > (1+\epsilon) tN^{2/3}) \geq e^{-C_1t^2 N^{1/3}}\Big) \leq e^{-C_2t^{3/2}}.$$
\end{theorem}

\begin{theorem}\label{nest}
Suppose boundary condition $W$ satisfies assumption \eqref{up_bd}. There exist positive constants $c_0, t_0, N_0, $ such that for each $N\geq N_0$, $N/2 \leq r \leq N-c_0$,  $t\geq t_0$, we have
\begin{align*}
&\mathbb{P} \Big( \log Z^{W}_{N}  - [\log  Z^{W}_{ r}  + \log \wt{Z}_{r,N}  ]\geq t(N-r)^{1/3}\Big) \leq e^{-t^{1/10}}.
\end{align*}
\end{theorem}

\section{Time correlations for large $r$: Proof of Theorem \ref{thm_r_large}}\label{sec_r_large}

The argument proving Theorem \ref{thm_r_large} is similar to the corresponding argument in \cite{bas-sep-she-23}, which proved an analogous estimate for the point-to-point polymer; however, some of the estimates for point-to-point polymer needs to be replaced by corresponding estimates for polymers with an initial condition stated in Section \ref{est_poly_bdry}.

Let us assume  $c_0 \leq N-r \leq N/2$. Recall the following identity
\begin{equation}\label{var_id}
\begin{aligned}
    \Var(U-V) &\geq  \inf_{\lambda\in \mathbb R}\Var(U-\lambda V) = (1-\textup{$\mathbb{C}$orr}^2(U,V))\Var(U).
\end{aligned}\end{equation}
We will upper and lower bound  the quantity $1-\textup{$\mathbb{C}$orr}^2(U,V)$ for $U = \log Z^{W}_{N} $ and $V= \log Z^{W}_{r}$ using the next two lemmas.



\begin{lemma}
   \label{l:vardiffub}
   There exists $C>0$ such that in the above setup we have 
   \begin{equation}\label{nest_var}
\Var( \log Z^W_{N}  - \log Z^W_{r} ) \leq C(N-r)^{2/3}.
\end{equation}
\end{lemma}

\begin{proof}
 We apply the inequality $\Var(A) \leq 2(\Var(B) + \mathbb{E}[(A-B)^2])$ to $A = \log Z^W_{N}  - \log Z^W_{r}$ and $
B = \log \wt{Z}_{r, N} $.  We have
$\Var(B) \leq C(N-r)^{2/3}$ by Proposition \ref{var_i}, and $\mathbb{E}[(A-B)^2]\leq C(N-r)^{2/3}$ follows from Theorem \ref{nest}. This completes the proof of the lemma. 
\end{proof}

\begin{lemma}
    \label{l:vardifflb}
    There exists $C>0$ such that in the above set-up we have \begin{equation} \label{lam_var_lb}
\Var( \log Z^W_{N}  - \lambda \log Z^W_{ r}  ) \geq C(N-r)^{2/3}
\end{equation}
for all $\lambda\in \R$.
\end{lemma}

\begin{proof}
    Let $\mathcal{F}$ be the $\sigma$-algebra of the weights in $\mathcal{L}_r^\leq$. Note that $\log Z^W_{r}$ is $\mathcal F$-measurable. 
    
    Then,  
\begin{align} 
\Var(\log Z^W_{N}  | \mathcal F)  
&= \Var(\log Z^W_{N}  - \log Z^W_{r}  | \mathcal F) \nn \\
\label{lb_here}  &
=  \mathbb{E}\Big[\Big(\log Z^W_{N}  - \log Z^W_{r} - \mathbb{E}[\log Z^W_{N}  - \log Z^W_{r} |\mathcal{F}]\Big)^2\Big| \mathcal{F}\Big].
\end{align}
We develop a lower bound for \eqref{lb_here} on a positive probability $\mathcal{F}$ measurable set. By Theorem \ref{nest}, 
\begin{equation} \label{expct_bd}
\Big|\mathbb{E}[\log Z^W_{N}  - \log Z^W_{r} ] - \mathbb{E}[\log \wt{Z}_{r,N}]\Big| \leq C(N-r)^{1/3}.
\end{equation}
In Proposition \ref{up_lb} the centering  $\Lambda_N$ can be replaced with $\mathbb E[\log \wt{Z}_{0, N}]$ because $\mathbb E[\log \wt{Z}_{0, N}]\le \Lambda_N$ by superadditivity. 
Then, Proposition \ref{up_lb} and \eqref{expct_bd} give
\begin{align*}
 e^{-ct^{3/2}}  &\le   \mathbb{P}\bigl(\log \wt{Z}_{r, N} - \mathbb E[\log \wt{Z}_{r, N}] \geq t(N-r)^{1/3}\bigr) \\
 &\le \mathbb{P}\bigl(\log \wt{Z}_{r, N}  - \mathbb{E}[\log Z^W_{N}  - \log Z^W_{r} ] \geq (t-C)(N-r)^{1/3}\bigr). 
\end{align*}
Let $s_0$ be a large constant and define the event 
\be\label{ArN}  A_{r,N}=\bigl\{ \log \wt{Z}_{r,N}   - \mathbb{E}[\log Z^W_{N}  - \log Z^W_{r} ] \geq s_0(N-r)^{1/3}\bigr\},
\ee
Then, $A_{r,N}$ is independent of $\mathcal{F}$ and  $\mathbb P(A_{r,N})$ is bounded below independently of $r$ and $N$ for $N-r$ sufficiently large.   
Next, using Chebyshev's inequality
 we get
\begin{align*}
&\mathbb{P}\Big(\Big|\mathbb{E}[\log Z^W_{N}  - \log Z^W_{r} |\mathcal{F}]-\mathbb{E}[\log Z^W_{N}  - \log Z^W_{r} ]\Big|> t(N-r)^{1/3}\Big) \\
&\leq  \frac{\Var(\mathbb{E}[\log Z^W_{N}  - \log Z^W_{r} |\mathcal{F}])}{t^2 (N-r)^{2/3}}\\ 
\quad &\leq\frac{\Var(\log Z^W_{N}  - \log Z^W_{r} )}{t^2 (N-r)^{2/3}} 
\leq C/t^2 \qquad \textup{by \eqref{nest_var}}.
\end{align*}
By choosing $t$ and $s_0$ large enough, 
  there is an event $B_{r,N}\in \mathcal{F}$, with positive probability  bounded below  independently of $N$ and $r$, on which  
\be\label{BrN} 
\Big|\mathbb{E}[\log Z^W_{N}  - \log Z^W_{r} |\mathcal{F}]-\mathbb{E}[\log Z^W_{N}  - \log Z^W_{r} ]\Big|\leq \frac{s_0}{10}(N-r)^{1/3}.
\ee
 
On $A_{r,N}\cap B_{r,N}$ we have the following bound, using first superadditivity  $\log Z^W_{N}  - \log Z^W_{r}  \geq \log \wt{Z}_{r,N} $, then \eqref{BrN} and finally \eqref{ArN}: 
\begin{align*}
& \log Z^W_{N}  - \log Z^W_{r} - \mathbb{E}[\log Z^W_{N}  - \log Z^W_{r} |\mathcal{F}] \\
& \geq   \log \wt{Z}_{r,N} - \mathbb{E}[\log Z^W_{N}  - \log Z^W_{r} ] - \frac{s_0}{10}(N-r)^{1/3}  
\geq \frac{9s_0}{10}(N-r)^{1/3}.
\end{align*}
Square this bound and insert  it inside   the conditional expectation on line  \eqref{lb_here}. 
Continuing from that line, we then have 
\begin{align*} 
\Var(\log Z^W_{N}  | \mathcal F) 
\ge C (N-r)^{2/3}  \,
\mathbb{E}[ \ind_{A_{r,N}}\ind_{B_{r,N}} |  \mathcal F  ]
\ge  C (N-r)^{2/3} \, 
 \ind_{B_{r,N}}. 
\end{align*} 
By the law of total variance,   for all $\lambda\in\R$,  
\begin{align*}
&\Var( \log Z^W_{N}  - \lambda \log Z^W_{ r}  )\\
&= \mathbb{E}\big[\Var( \log Z^W_{N}  - \lambda \log Z^W_{r} | \mathcal F ) \big] + \Var\big[ \mathbb{E}(\log Z^W_{N}  - \lambda \log Z^W_{ r} |\mathcal F ) \big]\\
& \geq \mathbb{E}\big[\Var( \log Z^W_{N}  - \lambda \log Z^W_{r} | \mathcal F) \big] \\
&= \mathbb{E}\big[\Var( \log Z^W_{N} | \mathcal F) \big]
\geq C(N-r)^{2/3} \, \mathbb P(B_{r,N}) 
\geq C(N-r)^{2/3} . 
\end{align*}
This completes the proof of the lemma. 
\end{proof}


We can now complete the proof of Theorem \ref{thm_r_large}.

\begin{proof}[Proof of Theorem \ref{thm_r_large}]
    Using \eqref{var_id} together with the bounds on $\Var(\log Z^{W}_{N})$ from Theorem \ref{var_s} it follows from Lemma \ref{l:vardiffub} and Lemma \ref{l:vardifflb} that there exists $C'_1,C'_2>0$ such that in the set-up of Theorem \ref{thm_r_large} 
    \begin{equation}
        \label{e:sq}
         C'_1\left(\frac{N-r}{N}\right)^{2/3}\le 1-\textup{$\mathbb{C}$orr}^2(\log Z^{W}_{N},\log Z^{W}_{r}) \le C_2'\left(\frac{N-r}{N}\right)^{2/3}.
    \end{equation}
    Since $\textup{$\mathbb{C}$orr}(\log Z^{W}_{N},\log Z^{W}_{r})\le 1$, The lower bound in \eqref{e:sq} immediately implies the upper bound in Theorem \ref{thm_r_large}. 
    To prove the lower bound in Theorem \ref{thm_r_large}, notice that it suffices to prove it only in the case where $(N-r)/N$ is sufficiently small; the remaining case can be proved by adjusting the constants as necessary. If $(N-r)/N$ is suffficiently small, it follows from Lemma \ref{l:vardiffub} together with Theorem \ref{var_s} that $\textup{$\mathbb{C}$orr}(\log Z^{W}_{N},\log Z^{W}_{r})\ge 0$; and the lower bound in Theorem \ref{thm_r_large} is immediate from the upper bound in \eqref{e:sq}. 
\end{proof}

\begin{remark}
    Note that Theorem \ref{thm_r_large} is valid for any such initial condition satisfying \textbf{Assumption B1}, random as well as deterministic. In particular, this remains true for the flat initial condition, i.e., $W\equiv 1$. A similar result, for a narrower class of initial conditions, was obtained in \cite{ferocctimecorr} in the zero temperature setup; see Section \ref{s:lpp} for more details. 
\end{remark}

\section{Time correlations for small $r$: Proof of Theorem \ref{thm_r_small}}\label{sec_r_small}
Notice that by Theorem \ref{var_s}, Theorem \ref{thm_r_small} is equivalent to showing, for some $C,C'>0$
$$C'r^{2/3}\le \Cov(\fgs_{r}, \fgs_{N})\le Cr^{2/3}.$$
Separate proofs of the upper and the lower bound in the display above are given in Sections \ref{r_small_up} and \ref{r_small_low} respectively. 

\subsection{Upper bound in Theorem \ref{thm_r_small}} \label{r_small_up}

First, note that if $r\geq N/1000$, our estimate follows directly from Cauchy-Schwarz inequality and Theorem \ref{var_s} as
$$\Cov(\fgs_{r}, \fgs_{N}) \leq \Var(\fgs_{ r})^{1/2} \Var(\fgs_{N})^{1/2} \leq Cr^{2/3}.
$$ 
Thus for the rest of this section, we will always assume that  $0 \leq r \leq  N/1000$. 

We start with defining some notation. Let
$Z^{W, \textup{in}, jr^{2/3}}_{r}$ denote the partition function where the paths are contained inside $R_{0,r}^{jr^{2/3}}$. By convention, we set $\log Z^{W, \textup{in}, 0}_{r} = 0$. Notice also that $Z^{W, \textup{in}, r}_{r}=Z^{W}_{r}$ by definition. 
Let $W'$ be an independent copy of $W$, and define another initial condition $W^{*, j}$ (which interpolates between $W$ and $W'$) such that  
\begin{align*}
\log W^{*, j}_{i+1} -\log W^{*, j}_{i}&= \log W'_{i+1}-\log W'_{i} \text{ for $-jr^{2/3} \leq i \leq jr^{2/3}-1$}\\
\log W^{*, j}_{i+1} -\log W^{*, j}_{i}&= \log W_{i+1} -\log W_{i}\text{ otherwise.}
\end{align*}
Let $Z^{W^{*,j}, \textup{out},k}_{N}$ denote the partition function using initial condition $W^{*, j}$ and the paths inside $\mathbb{Z}^2 \setminus R_{0,r}^{k}$. We set $Z^{W^{*,0}, \textup{out}, 0}_{N}=Z^{W}_{N}$ by convention.

Without loss of generality, we shall assume that both $r^{2/3}$ and $r^{1/3}$ are integers. To simplify notation, for $1< j \le r^{1/3}, 1 < k \leq r^{1/3}$, let 
\begin{align*}
U_j&=\log Z^{W, \textup{in}, jr^{2/3}}_{r}- \log Z^{W, \textup{in}, (j-1)r^{2/3}}_{r}\\
V_k &= \log Z^{W^{*,(k-1)r^{2/3}},\textup{out}, (k-1)r^{2/3}}_{N} - \log Z^{W^{*,kr^{2/3}},\textup{out}, kr^{2/3}}_{N}.
\end{align*}
Then we can rewrite the covariance as follows
\begin{align}
\Cov(\log Z^{W}_{r}, \log Z^{W}_{ N})
= \Big(\sum_{j=1}^{r^{1/3}} \sum_{k=1}^{r^{1/3}} \Cov(U_j, V_k)\Big)+\Cov(Z^{W}_{r},Z^{W^{*,r},\textup{out},r}_{N}) .\label{cov_decomp}
\end{align}
A similar but different decomposition of covariance was also used in a related problem in \cite{BBF22}. The proof shall proceed by bounding the two terms separately. To bound the first term we shall use bounds on the variances of the random variables appearing in the above decomposition followed by an application of the Cauchy-Schwarz inequality and the mixing hypothesis \eqref{mix} in \textbf{Assumption B2}. The second term shall be bounded by a different decomposition. We shall need the following propositions which are proved in the subsequent subsections. 

\begin{proposition}\label{r_proof}
Suppose the boundary condition $W$ satisfies assumption \eqref{up_bd}. There exist positive constants $C_1, C_2, r_0$ such that for each $r\geq r_0$ and $0\le j \leq r^{1/3}-1$
$$\Var\Big(\log Z^{W, \textup{in}, (j+1)r^{2/3}}_{r}- \log Z^{W, \textup{in}, jr^{2/3}}_{r}\Big) \leq C_2e^{-C_1 j}r^{2/3}.$$
\end{proposition}

\begin{proposition}\label{N_proof}
Suppose the boundary condition $W$ satisfies assumptions \eqref{up_bd} and \eqref{low_tail}.
There exist positive constants $C, r_0, N_0$ such that for each $r\geq r_0$, $N \geq N_0\vee 1000r$ and $0\le k \leq r^{1/3}-1$,
$$\Var\Big( \log Z^{W^{*, kr^{2/3}},\textup{out}, kr^{2/3}}_{N} - \log Z^{W^{*, (k+1)r^{2/3}}, \textup{out},  (k+1)r^{2/3}}_{N}\Big) \leq C {(k+1)}^{40}r^{2/3}.$$
\end{proposition}

\begin{proposition}
    \label{p:3rd}
Suppose the boundary condition $W$ satisfies assumptions \eqref{up_bd} and \eqref{low_tail}.
    There exist positive constants $C$ and $r_0$ such that for each $r\geq r_0$ 
    $$\Cov(Z^{W}_{r},Z^{W^{*,r},\textup{out}, r}_{N})\le Cr^{2/3}.$$
\end{proposition}

Notice that Proposition \ref{p:3rd} is trivial in the case where $X_{i}:=\log W_{i}-\log W_{i-1}$ is an independent sequence; so the proof becomes substantially simpler. Assuming Propositions \ref{r_proof}, \ref{N_proof} and \ref{p:3rd} we can now complete the proof of the upper bound. 

\begin{proof}[Proof of Theorem \ref{thm_r_small}, upper bound] 
Recall the decomposition \eqref{cov_decomp}. Observe that by Proposition \ref{p:3rd}, it suffices to upper bound the first sum thereby $Cr^{2/3}$. For this, we further decompose the first sum in \eqref{cov_decomp} as 
\begin{align}
\sum_{j=1}^{r^{1/3}} \sum_{k=1}^{r^{1/3}} \Cov(U_j, V_k)
= \sum_{j=1}^{r^{1/3}} \sum_{k=1}^{j} \Cov(U_j, V_k)+
\sum_{j=1}^{r^{1/3}} \sum_{k=j+1}^{r^{1/3}} \Cov(U_j, V_k).\label{cov_decomp2}
\end{align}
Notice that, by the Cauchy-Schwarz inequality together with Propositions \ref{r_proof} and \ref{N_proof}, the first term on the right of \eqref{cov_decomp2} is upper bounded by 
$$\sum_{j=1}^{r^{1/3}}\sum_{k=1}^{j}\sqrt{\Var(U_j)}\sqrt{ \Var(V_k)} \leq Cr^{2/3}\sum_{j=1}^{r^{1/3}} e^{-Cj} \sum_{k=1}^{j} k^{20}  \leq Cr^{2/3}.$$
Therefore it suffices to prove a similar bound for the second sum. For this, we shall show that for $r^{1/3}\ge k\ge j+1$, 
\begin{equation}
    \label{e:2bd}
    \Cov(U_j,V_{k})\le \frac{1}{((k-j-1)r^{2/3}+1)^{100}} \sqrt{\Var(U_j)}\sqrt{\Var(V_{k})}.
\end{equation}
Indeed, observe that $(k-j-1)r^{2/3}+1\ge (k-j)$ and also that 
$$\sum_{k=j+1}^{r^{1/3}} \frac{k^{20}}{(k-j)^{100}}\le C(j+1)^{21}$$ 
for some $C>0$. Therefore, by \eqref{e:2bd} and Propositions \ref{r_proof} and \ref{N_proof}, it follows that the second term in \eqref{cov_decomp2} is upper bounded by $Cr^{2/3}$ for some $C>0$. 

It remains to show \eqref{e:2bd}. Consider the $\sigma$-algebras,
\begin{align}\label{def_f12}
\begin{split}
\mathcal{F}^{j}_1 &= \sigma\Big(\{Y_{\bf z}\}_{{\bf z}\in R_{0, r}^{jr^{2/3}}}\Big),\\
\mathcal{F}^{k}_2 &= \sigma \Big(\{Y_{\bf z}\}_{{\bf z}\in \mathbb{Z}^2 \setminus R_{0,r}^{kr^{2/3}}} \bigcup \{\log W'_{i+1}-\log W'_{i}\}_{i \in [-kr^{2/3}, kr^{2/3}-1]}\Big).
\end{split}
\end{align}
Notice that if $k\ge j+2$, then $\mathcal{F}^{j+1}_1$ and $\mathcal{F}^{k-1}_2$ are independent, further $U_{j}$ is $\mathcal{F}^{j+1}_1$ measurable and $V_{k}$ is $\mathcal{F}^{k-1}_2$ measurable. 

Fix $1\le j,k \le r^{1/3}$ with $k\ge j+2$, and set $\mathcal{F} = \sigma(\mathcal{F}^{j+1}_1, \mathcal{F}^{k-1}_2)$. By the above observation, 
$\mathbb{E}[U_{j}|\mathcal{F}]=\mathbb{E}[U_{j}\mid \mathcal{F}^{j+1}_1]$ and $\mathbb{E}[V_{k}|\mathcal{F}]=\mathbb{E}[V_{k}\mid \mathcal{F}^{k-1}_2]$. Hence, these two random variables are independent. Observe also that conditional on a fixed realisation of $\mathcal{F}$, $U_{j}$ and $V_{k}$ are functions of $X_{i}=\log W_{i}-\log W_{i-1}$ restricted to disjoint index sets $[-(j+1)r^{2/3}+1,(j+1)r^{2/3}]$ and $\mathbb{Z} \setminus [-(k-1)r^{2/3}+1,kr^{2/3}]$ respectively, and these sets have mutual distance $(k-j-2)r^{2/3}+1$. 
Therefore, using assumption \eqref{mix} 
\begin{align}
\Cov(U_j, V_k) 
&=\mathbb{E}[\Cov(U_j, V_k |\mathcal{F})] +\Cov(\mathbb{E}[U_j|\mathcal{F}],  \mathbb{E}[V_k|\mathcal{F}]) \nonumber\\
&= \mathbb{E}[\Cov(U_j, V_k|\mathcal{F})] \\
\textup{by \eqref{mix}} \qquad & \leq \mathbb{E}\Big[\frac{1}{((k-j-2)r^{2/3}+1)^{100}}\sqrt{\Var(U_j|\mathcal{F})}\sqrt{\Var(V_k|\mathcal{F})}\Big]\nonumber\\
&\leq \frac{1}{((k-j-2)r^{2/3}+1)^{100}}\sqrt{\Var(U_j)}\sqrt{\Var(V_k)}\label{corr_bound}
\end{align}
where in the last step we have used the Cauchy-Schwarz inequality together with the fact the expectation of conditional variance is bounded by the unconditional variance. This establishes \eqref{e:2bd}. With this, we have finished the proof of the upper bound. 

\end{proof}

\subsubsection{Proof of Proposition \ref{r_proof}}

Since $\log Z^{W}_{r} \geq \log Z^{W, \textup{in}, (j+1)r^{2/3}}_{r} \geq \log Z^{W, \textup{in}, jr^{2/3}}_{r}$, it suffices for us to show for $0\leq j \leq r^{1/3}$,
$$\mathbb{E}\Big[\Big(\log Z^{W}_{r}  - \log Z^{W, \textup{in}, jr^{2/3}}_{r}\Big)^2\Big] \leq e^{-Cj} r^{2/3}.$$
To show this, we will first prove show that 
\begin{equation}\label{event_A}
\mathbb{P}\Big(\log Z^{W}_{r}  - \log Z^{W, \textup{in}, jr^{2/3}}_{r} \geq e^{-C'j^2 r^{1/3}}\Big) \leq e^{-Cj^3}.
\end{equation}
Let us rewrite the event inside  \eqref{event_A}. Let $Z^{W, \textup{exit}, jr^{2/3}}_{r}$ denote the partition function over paths from $\mathcal{L}_0$ to $(r,r)$ that exit the two diagonal sides of $R_{0,r}^{jr^{2/3}}$. In our calculation below, the third equality holds because $Z^W_r = Z^{W, \textup{in}, jr^{2/3}}_{r}+Z^{W, \textup{exit}, jr^{2/3}}_{r}$. Also note Taylor's theorem states that $1-e^{z} \approx -z$ when $|z|$ is small, so the last set inclusion below holds provided that that $r \geq r_0$ is sufficiently large. We have  
\begin{align*}
 \{\text{the event in \eqref{event_A}}\}
 &= \Big\{\frac{Z^{W}_{r}}{Z^{W, \textup{in}, jr^{2/3}}_{r}} \geq \exp(e^{-C'j^2 r^{1/3})}\Big\}\\
& = \Big\{\frac{Z^{W, \textup{in}, jr^{2/3}}_{r}}{Z^{W}_{r}} \leq \exp(-e^{-C'j^2 r^{1/3})}\Big\}\\
& = \Big\{\frac{Z^{W, \textup{exit}, jr^{2/3}}_{r}}{Z^{W}_{r}} \geq 1-\exp(-e^{-C'j^2 r^{1/3})}\Big\}
\subset \Big\{\frac{Z^{W, \textup{exit}, jr^{2/3}}_{r}}{Z^{W}_{r}} \geq \tfrac{1}{2}e^{-C'j^2 r^{1/3}}\Big\}.
\end{align*}
Thus, to show \eqref{event_A}, it suffices to show 
\begin{equation}\label{show_A}
\mathbb{P}\Big(\frac{Z^{W, \textup{exit}, jr^{2/3}}_{r}}{Z^{W}_{r}} \geq e^{-C'j^2 r^{1/3}}\Big) \leq e^{-Cj^3}.
\end{equation}
Note here we drop the constant $\tfrac{1}{2}$ by modifying $C'$.

To get \eqref{show_A}, let $\mathfrak{A}$ be the collection of polymer paths from $\mathcal{L}_0$ to $(r,r)$ which exit $R_{0,r}^{jr^{2/3}}$. We break $\mathfrak{A}$ into two groups: 
$$\mathfrak{A}_1 =  \bigcup_{|k|\geq \tfrac{j}{2}r^{2/3}}\mathfrak{A} \cap \mathbb{X}_{\olsi{k}, N} \qquad \text{and} \qquad \mathfrak{A}_2 =  \bigcup_{|k|< \tfrac{j}{2}r^{2/3}}\mathfrak{A} \cap \mathbb{X}_{\olsi{k}, N}.$$
We will show that for $i = 1,2$ and $C'>0$ sufficiently small,
\begin{equation}\label{est_i12}
\mathbb{P}\big(\log Z^{W}_{r}(\mathfrak{A}_i) - \log Z^{W}_{r} \geq -C' j^2 r^{1/3} \big) \leq e^{-Cj^3}.
\end{equation}
Note for $i=1$, this estimate follows directly from Theorem \ref{exit_q}. For $i=2$, by a union bound,
\begin{align*}
\eqref{est_i12} &\leq \mathbb{P}\Big(\log Z^{W}_{r}(\mathfrak{A}_2) - \Lambda_N \geq -2C' j^2 r^{1/3} \Big) + \mathbb{P}\Big(\log Z^{W}_{r} - \Lambda_N \leq - C' j^2 r^{1/3} \Big) \\
&\leq \mathbb{P}\Big(\log \wt{Z}^{\textup{exit}, jr^{2/3}}_{\mathcal{L}_{0}^{\frac{j}{2}r^{2/3}}, r} - \Lambda_N \geq -3C' j^2 r^{1/3} \Big) + \mathbb{P}\Big(\max_{|k| \leq \frac{j}{2}r^{2/3}} \log W_k \geq \tfrac{1}{3}C' j^2 r^{1/3} \Big)\\
& \qquad \qquad \qquad \qquad + \mathbb{P}\Big(\log Z^{W}_{r} - \Lambda_N \leq -C' j^2 r^{1/3} \Big),
\end{align*}
and all three probabilities are bounded by $e^{-Cj^3}$ using Proposition \ref{trans_fluc_loss3}, assumption \eqref{up_bd} and Proposition \ref{b_low_tail}. Then, because $Z^{W, \textup{exit}, jr^{2/3}}_{r} =  Z^{W}_{r}(\mathfrak{A}_1) +  Z^{W}_{r}(\mathfrak{A}_2)$, \eqref{est_i12} implies \eqref{show_A}. And we have finished the proof for \eqref{event_A}. 

Next, let us denote the event inside the probability of \eqref{event_A} as $A$. Then, 
\begin{align*}
&\mathbb{E}\Big[\big(\log Z^{W}_{r}  - \log Z^{W, \textup{in}, jr^{2/3}}_{r}\big)^2\Big]\\
& \leq  \mathbb{E}\Big[\big(\log Z^{W}_{r}  - \log Z^{W, \textup{in},jr^{2/3}}_{r}\big)^2 \mathbbm{1}_{A^c}\Big] + \mathbb{E}\Big[\big(\log Z^{W}_{r}  - \log Z^{W, \textup{in}, jr^{2/3}}_{r}\big)^2\mathbbm{1}_A\Big]\\
&\leq e^{-Cj}+ \mathbb{E}\Big[\big(\log Z^{W}_{r}  - \log Z^{W, \textup{in}, jr^{2/3}}_{r}\big)^4\Big]^{1/2}\mathbb{P}(A)^{1/2}.
\end{align*}
To finish the proof, it remains to show that
$$\mathbb{E}\Big[\big(\log Z^{W}_{r}  - \log Z^{W, \textup{in}, jr^{2/3}}_{r}\big)^4\Big] \leq C r^{4/3}.$$
And this follows directly from Proposition \ref{wide_similar} and Proposition \ref{b_up_tail} which essentially provide an exponential upper bound for the right tails of $\log Z^{W}_{r}-\log Z^{W, \textup{in}, jr^{2/3}}_{r}$.

\subsubsection{Proof of Proposition \ref{N_proof}}
We shall first prove the following proposition. 

\begin{proposition}
    \label{p:difference}
  Suppose the boundary condition $W$ satisfies assumption \eqref{up_bd} and \eqref{low_tail}. Then there exist positive constants $C$ and  $r_0$ such that for each $r\geq r_0$ and $0\le k\le r^{1/3}$, we have 
\begin{equation}
     \mathbb{E}[(\log Z^{W}_{N} - \log Z^{W, \textup{out},kr^{2/3}}_{N})^2]\leq C(k+1)^{40} r^{2/3}.\label{hard_est}
\end{equation}
\end{proposition}
\begin{proof}
First, note that when $k=0$, the expectation is zero by our convention. We may now assume that $k\geq 1$.

Let $Z^{W, \textup{touch},kr^{2/3}}_{N}$ denote the partition function that sums over all paths from $\mathcal{L}_0$ to $(N,N)$ which intersect the rectangle $R^{kr^{2/3}}_{0, r}$, then 
\begin{equation}\label{bet_max}
\max\Big\{\log Z^{W, \text{out},kr^{2/3}}_{N}, \log  Z^{W, \textup{touch},kr^{2/3}}_{N}\Big\} \leq \log Z^{W}_{N} \leq\max \Big\{\log Z^{W, \text{out},kr^{2/3}}_{N}, \log Z^{W, \textup{touch},kr^{2/3}}_{N}\Big\} + 2.
\end{equation}
Let us define the event 
$$B = \Big\{\log Z^{W, \text{out},kr^{2/3}}_{N}< \log  Z^{W, \textup{touch},kr^{2/3}}_{N}\Big\}.$$
And we split the expectation in \eqref{hard_est} according to $B$, 
$$\eqref{hard_est} = \mathbb{E}[(\log Z^{W}_{ N} - \log Z^{W, \textup{out},kr^{2/3}}_{N})^2\mathbbm{1}_{B}] + \mathbb{E}[(\log Z^{W}_{N} - \log Z^{W, \textup{out},kr^{2/3}}_{ N})^2\mathbbm{1}_{B^c}].$$
Now, using \eqref{bet_max},  the expectation term above with  $B^c$ is bounded by $2^2$. Since $r$ is sufficiently large, to get \eqref{hard_est}, it remains to prove that 
\begin{equation}\label{on_B}
\mathbb{E}[(\log Z^{W, \textup{touch}, kr^{2/3}}_{N} - \log Z^{W, \textup{out},kr^{2/3}}_{ N})^2\mathbbm{1}_{B}]\leq C(k+1)^{40}r^{2/3}.
\end{equation}
Because we are on the event $B$, it suffices for us to show that for each $t\geq t_0$ sufficiently large  and $1\leq k \leq r^{1/3}$
\begin{equation}\label{B_tail}
\mathbb{P}\Big(\log Z^{W, \textup{touch}, kr^{2/3}}_{N} - \log Z^{W, \textup{out},kr^{2/3}}_{N} \geq {t}k^{20}r^{1/3} \Big) \leq e^{-Ct^{1/100}}.
\end{equation}

First, let us focus on the term $\log Z^{W, \textup{touch}, kr^{2/3}}_{N}$ inside the probability above. Because of the paths from   $\log Z^{W, \textup{touch}, kr^{2/3}}_{N}$ must touch the rectangle $R_{0,r}^{kr^{2/3}}$ by definition, we may replace $\log Z^{W, \textup{touch}, kr^{2/3}}_{N}$ by 
\begin{equation}\label{max_form}\max_{{\bf p}_1\in \mathcal{L}_0^{100r}}\max_{{\bf p}_2\in \mathcal{L}_{8r}^{100r}} \log Z^{W}_{{\bf p}_1} + \log \wt{Z}^{ \textup{touch},kr^{2/3}}_{{\bf p}_1, {\bf p}_2}+ \log \wt{Z}_{{\bf p}_2, N},
\end{equation}
because their absolute difference is at most $100\log r$. Note here that the last free energy $\log \wt{Z}_{{\bf p}_2, N}$ is well defined since $8r \leq N/10$, and this inequality holds because we are considering the case $r\leq N/1000$. 

Next, let us denote the two maximizers in \eqref{max_form} as ${\bf p}_1^*$ and ${\bf p}_2^*$ and define 
$$H=\Big\{\max_{i\in \{1, 2\} }|{\bf p}_i^* \cdot ({\bf e}_1 - {\bf e}_2)| \geq tkr^{2/3}\Big\}.$$
Lemma \ref{lem_H} at the end of the subsection states that 
$\mathbb{P}(H) \leq e^{-C t^{1/100}}.$ With this, let us start the estimate for \eqref{B_tail}, we have
\begin{align*}
\text{right side of \eqref{B_tail}}
& \leq \mathbb{P}\Big(\Big\{\log Z^{W, \textup{touch}, kr^{2/3}}_{N} - \log Z^{W, \textup{out},kr^{2/3}}_{N} \geq  {t} kr^{1/3} \Big\}\cap H^c\Big) + e^{-Ct^{1/100}}.
\end{align*}
Looking at the probability term above, 
\begin{align}
&\mathbb{P}\Big(\Big\{\log Z^{W, \textup{touch}, kr^{2/3}}_{N} - \log Z^{W, \textup{out},kr^{2/3}}_{ N} \geq {t} kr^{1/3} \Big\}\cap H^c\Big)\nonumber\\
& \leq \mathbb{P}\Big( \Big[\max_{{\bf p_1}\in \mathcal{L}_0^{tkr^{2/3}}}\max_{{\bf p_2}\in \mathcal{L}_{8r}^{tkr^{2/3}}} \log Z^{W}_{{\bf p}_1} + \log \wt{Z}_{{\bf p}_1, {\bf p}_2}+ \log \wt{Z}_{{\bf p}_2, N} \Big]\nonumber\\
& - \Big[\log Z^{W}_{ \olsi{2kr^{2/3}}} + \log \wt{Z}^{ \textup{in},R_{\olsi{2kr^{2/3}}, r+\olsi{2kr^{2/3}}}^{kr^{2/3}}}_{\olsi{2kr^{2/3}}, r+\olsi{2kr^{2/3}}}+ \log \wt{Z}_{r+\olsi{2kr^{2/3}}, 8r} + \log \wt{Z}_{8r, N}\Big] \geq \tfrac{1}{2}{t} k^{20}  r^{1/3} \Big)\nonumber\\
& \leq \mathbb{P}\Big( \max_{{\bf p_1}\in \mathcal{L}_0^{tkr^{2/3}}} \log Z^{W}_{{\bf p}_1} - \log Z^{W}_{\olsi{2kr^{2/3}}}  \geq \tfrac{1}{6}t k^{20} r^{1/3} \Big)\label{item111}\\
& + \mathbb{P}\Big( \max_{{\bf p_1}\in \mathcal{L}_0^{tkr^{2/3}}}\max_{{\bf p_2}\in \mathcal{L}_{{t}}^{tkr^{2/3}}}  \log \wt{Z}_{{\bf p}_1, {\bf p}_2} - \log \wt{Z}^{ \textup{in},R_{\olsi{2kr^{2/3}}, r+\olsi{2kr^{2/3}}}^{kr^{2/3}}}_{\olsi{2kr^{2/3}}, r+\olsi{2kr^{2/3}}}- \log \wt{Z}_{r+\olsi{2kr^{2/3}}, 8r}  \geq \tfrac{1}{6}t k^{20} r^{1/3} \Big)\label{item222}\\
& \qquad + \mathbb{P}\Big( \max_{{\bf p_2}\in \mathcal{L}_{{t}}^{tkr^{2/3}}} \log \wt{Z}_{{\bf p}_2, N} - \log \wt{Z}_{8r, N} \geq \tfrac{1}{6}k^{20} {t} r^{1/3} \Big)\label{item333}
\end{align}

Note that by  assumptions \eqref{up_bd} and \eqref{low_tail},
$$\eqref{item111} \leq \mathbb{P}\Big( \max_{{\bf p_1}\in \mathcal{L}_0^{tkr^{2/3}}} \log Z^{W}_{{\bf p}_1} \geq \tfrac{1}{12}tk^{20} r^{1/3} \Big)+\mathbb{P}\Big(\log Z^{W}_{\olsi{2kr^{2/3}}} \leq -\tfrac{1}{12}tk^{20} r^{1/3} \Big) \leq e^{-Ct^{1/10}}.$$ 
And \eqref{item333} $\leq e^{-Ct^{1/10}}$ following from Proposition \ref{max_all_t1}. 
Finally, we estimate \eqref{item222}, 
\begin{align}
\eqref{item222} &\leq 
\mathbb{P}\Big(\Big[\max_{{\bf p_1}\in \mathcal{L}_0^{tkr^{2/3}}}\max_{{\bf p_2}\in \mathcal{L}_{8r}^{tkr^{2/3}}}  \log \wt{Z}_{{\bf p}_1, {\bf p}_2} - \Lambda_{8r} \geq \tfrac{1}{12}tk^{20} r^{1/3} \Big)\label{it1}\\
& \qquad \qquad + \mathbb{P}\Big(\log \wt{Z}^{ \textup{in},R_{\olsi{2kr^{2/3}}, r+\olsi{2rk^{2/3}}}^{kr^{2/3}}}_{\olsi{2kr^{2/3}}, r+\olsi{2kr^{2/3}}} - \Lambda_r \leq -\tfrac{1}{24}t k^{20} r^{1/3}\Big)\label{it2}\\
& \qquad \qquad + \mathbb{P}\Big(\log \wt{Z}_{r+\olsi{2kr^{2/3}}, 8r}  - \Lambda_{7r} \leq -\tfrac{1}{24}tk^{20} r^{1/3}\Big)\label{it3}
\end{align}

The terms \eqref{it1} and \eqref{it2} are upper bounded by $e^{-t^{1/10}}$ following from Proposition \ref{ptl_upper} and Proposition \ref{wide_similar}. For \eqref{it3}, we split the estimate into two cases, depending on the value of $k$. 

Note that if $k \geq r^{1/4}$, then 
$$\eqref{it3} \leq \mathbb{P}\Big(\log \wt{Z}_{r+\olsi{2kr^{2/3}}, r}   \leq - {t} k^{4}r^4\Big).$$
We can chose an up-right path $\gamma$ between $(r+{2kr^{2/3}}, r-{2kr^{2/3}})$ and $(8r, 8r)$ in some deterministic ordering, and then 
$$\mathbb{P}\Big(\log \wt{Z}_{r+\olsi{2kr^{2/3}}, 8r}  \leq - {t} k^{4}r^{4}\Big)\leq \mathbb{P}\Big(\log \wt{Z}_{r+\olsi{2kr^{2/3}}, 8r}(\{\gamma\})  \leq - {t} k^4 r^{4}\Big) \leq e^{-t^{1/10}},$$
where the last inequality holds by  Theorem \ref{max_sub_exp} since $\log \wt{Z}_{r+\olsi{2kr^{2/3}}, 8r}(\{\gamma\})$ is just a sum of i.i.d.~inverse-gamma random variables.

On the other hand if $k \leq r^{1/4}$, then for $r\geq r_0$ sufficiently large, $kr^{2/3} \leq \epsilon_0 r$ where $\epsilon_0$ is the constant from Proposition \ref{reg_shape}. And from the same proposition, we obtain 
$$\Big|\Lambda((7r-{2kr^{2/3}}, 7r+{2kr^{2/3}})) - \Lambda_{7r} \Big| \leq Ck^2 r^{1/3}.$$ Then, we obtain
$$\eqref{it3} \leq \mathbb{P}\Big(\log \wt{Z}_{r+\olsi{2kr^{2/3}}, 8r}  - \Lambda((7r-{2kr^{2/3}}, 7r+{2kr^{2/3}})) \leq -\tfrac{1}{48}t k^{20}  r^{1/3} \Big) \leq e^{-t^{1/10}}$$
where the last inequality comes from Proposition \ref{low_ub}. With this, we have finished the proof of Proposition \ref{p:difference}.

\end{proof}

We can now give a proof of Proposition \ref{N_proof}. 

\begin{proof}[Proof of Proposition \ref{N_proof}]

First, we note that by a union bound $W^{*, j}$ also satisfies assumptions \eqref{up_bd} and \eqref{low_tail}. Let $K\in \mathbb{Z}_{>0}$ be given,  if $jr^{2/3} \geq K$, then $W^{*, j}_K = W'_K$ and the claim  is clear. Suppose $jr^{2/3} < K$, then 
\begin{align*}\log W^{*, j}_K &= \log W_K - \log W_{jr^{2/3}}  + \log W'_{jr^{2/3}}\\
\max_{0\leq k \leq K} \log W^{*, j}_k &\leq \max_{0\leq k \leq jr^{2/3}} \log W'_k + \max_{0\leq k \leq K} \log W_k - \log W_{jr^{2/3}},
\end{align*}
and the claim follows from a union bound. 

Then, using Proposition \ref{p:difference} and a triangle inequality we get  \begin{align}\mathbb{E}\Big[\Big(\log &Z^{W^{*, kr^{2/3}},\textup{out}, kr^{2/3}}_{N} - \log Z^{W, \textup{out},kr^{2/3}}_{N}\Big)^2\Big]^{1/2}\nonumber\\
&\leq
\mathbb{E}\Big[\Big(\log Z^{W^{*, kr^{2/3}},\textup{out}, kr^{2/3}}_{N} - \log Z^{W^{*, kr^{2/3}}}_{N}\Big)^2\Big]^{1/2}+\mathbb{E}\Big[\Big(\log Z^{W}_{N} - \log Z^{W, \textup{out},kr^{2/3}}_{N}\Big)^2\Big]^{1/2}\nonumber\\
&\qquad \qquad +
\mathbb{E}\Big[\Big(\log Z^{W^{*, kr^{2/3}}}_{N} - \Lambda_N\Big)^2\Big]^{1/2}+\mathbb{E}\Big[\Big(\log Z^{W}_{N}- \Lambda_N\Big)^2\Big]^{1/2}.\label{l2_terms}
\end{align}
Since $W^{*,j}$ and $W$ both satisfy the hypothesis of Proposition \ref{p:difference}, it follows that the first two terms above are bounded by $C(k+1)^{20}r^{1/3}$
while the last two terms are bounded by $Cr^{1/3}$ using Proposition \ref{b_up_tail} and Proposition \ref{b_low_tail}. It therefore follows that 
\begin{equation}
    \mathbb{E}[(\log Z^{W^{*, kr^{2/3}},\textup{out}, kr^{2/3}}_{N} - \log Z^{W,\textup{out},kr^{2/3}}_{N})^2] \leq C(k+1)^{40} r^{2/3}\label{easy_est}
\end{equation}
for some $C>0$. Combining \eqref{easy_est} with Proposition \ref{p:difference} gives 
$$\Var\Big( \log Z^{W^{*, kr^{2/3}}, \textup{out}, kr^{2/3}}_{N} - \log Z^{W}_{N}\Big) \leq C_1 (k+1)^{40}r^{2/3}$$ for each $k\le r^{1/3}$. 
 Using the fact 
 $$\Var(A+B) \leq \Var(A) + \Var(B) + 2\sqrt{\Var(A)}\sqrt{\Var(B)}\le 4\max\{\Var(A), \Var(B)\}$$ 
 where $A =  \log Z^{W^{*, kr^{2/3}}, \textup{out}, kr^{2/3}}_{N} - \log Z^{W}_{N}$ and $B=   \log Z^{W}_{N} - \log Z^{W^{*, (k+1)r^{2/3}}, \textup{out}, (k+1)r^{2/3}}_{N}$, this immediately completes the proof of Proposition \ref{N_proof}. 
\end{proof}

Recall ${\bf p}_1^*$ and ${\bf p}_2^*$ are the two maximizers in \eqref{max_form}. The following lemma was used in the proof of Proposition \ref{p:difference}.

\begin{lemma}\label{lem_H}
Let $k \in \mathbb{Z}_{>0}$ and  
$H=\{\max_{i\in \{1, 2\} }|{\bf p}_i^* \cdot ({\bf e}_1 - {\bf e}_2)| \geq tkr^{2/3}\}.$ There exist positive constants $C_1, t_0$ such that for each $t\geq t_0$, it holds that  
\begin{equation}\label{def_H}
\mathbb{P}(H) \leq e^{-C_1t^{1/100}}.
\end{equation}
\end{lemma}
\begin{proof}
Recall $\olsi{a} = (a, -a)$, and for integers $|u|, |v| \geq t$, let us define 
\begin{align*}
I_u = \mathcal{L}_{\olsi{ukr^{2/3}}}^{kr^{2/3}} \quad \text{ and }\quad 
J_v = \mathcal{L}_{8r + \olsi{vkr^{2/3}}}^{kr^{2/3}}.
\end{align*}
And we will show that 
\begin{equation}\label{loc_p}
\mathbb{P} ({\bf p}_1^* \in I_u, {\bf p}_2^* \in J_v)  \leq e^{-C(k(|u|+|v|))^{1/10}}.
\end{equation}
Once we have \eqref{loc_p}, by a union bound, we obtain \eqref{def_H} since
$$\mathbb{P}(H) \leq \sum_{|u|, |v| \geq t} \mathbb{P} ({\bf p}_1^* \in I_u, {\bf p}_2^* \in J_v) \leq \sum_{|u|, |v| \geq t} e^{-C(k(|u|+|v|))^{1/10}} \leq e^{-Ct^{1/100}}.$$

To show \eqref{loc_p}, let $c_0$ be some positive constant which we fix later when estimating \eqref{item_2} , and it holds that 
\begin{align}
&\mathbb{P} ({\bf p}_1 \in I_u, {\bf p}_2 \in J_v)\nonumber\\
& \leq \mathbb{P}\Big(\max_{{\bf p}_1 \in I_u} \{\log Z^{W}_{{\bf p}_1}\} + \max_{{\bf p}_1 \in I_u} \max_{{\bf p}_2 \in J_v}  \{\log \wt{Z}^{\textup{touch}, kr^{2/3}}_{{\bf p}_1, {\bf p}_2}\} + \max_{{\bf p}_2 \in J_v} \{\log \wt{Z}_{{\bf p}_2, N}\}\nonumber\\
& \qquad \qquad \qquad \qquad\qquad \qquad \qquad \qquad \geq \log \wt{Z}_{0, 8r} + \log \wt{Z}_{8r, N}\Big)\nonumber\\
& \leq \mathbb{P}\Big(\max_{{\bf p}_1 \in I_u} \log Z^{W}_{{\bf p}_1} \geq c_0(|u|+|v|)kr^{1/3}\Big) \label{item_1}\\
& \qquad + \mathbb{P}\Big(\max_{{\bf p}_1 \in I_u} \max_{{\bf p}_2 \in J_v}  \wt{Z}^{\textup{touch}, kr^{2/3}}_{{\bf p}_1, {\bf p}_2} - \log \wt{Z}_{0,  8r} \geq - 2c_0(|u|+|v|)kr^{1/3}\Big) \label{item_2}\\
& \qquad +  \mathbb{P}\Big(\max_{{\bf p}_2 \in J_v} \wt{Z}_{{\bf p}_2, N} - \log \wt{Z}_{8r, N} \geq c_0(|u|+|v|)kr^{1/3} \Big). \label{item_3}
\end{align}
First, we see that  \eqref{item_1} $\leq e^{-C(k(|u|+|v|))^{1/10}}$ from assumption \eqref{up_bd}. And \eqref{item_3} $\leq e^{-C(k(|u|+|v|))^{1/10}}$ following from Proposition \ref{max_all_t1}. 
It remains to show \eqref{item_2}. First, using Proposition \ref{low_ub} and a union bound, it suffices for us to upper bound 
\begin{equation}\label{item_22}
\mathbb{P}\Big(\max_{{\bf p}_1 \in I_u} \max_{{\bf p}_2 \in J_v}  \wt{Z}^{\textup{touch}, kr^{2/3}}_{{\bf p}_1, {\bf p}_2} -\Lambda_{8r} \geq - 3c_0k(|u|+|v|)r^{1/3}\Big).
\end{equation}

If $u$ and $v$ from \eqref{item_22} have opposite signs, then the paths between $I_u$ and $J_v$ must have high transversal fluctuation of order at least $\tfrac{1}{10}k(|u|+|v|)r^{2/3}$. Thus provided $c_0$ is fixed sufficiently small, by Proposition \ref{trans_fluc_loss}, we have 
\eqref{item_22} $\leq  e^{-C(k(|u|+|v|))^{1/10}}$.
Without the loss of generality, we will assume that $u$ and $v$ are both positive. We will split the remaining estimate into two cases, whether $|u-v| \geq \frac{u+v}{10}$ or $|u-v| \leq \frac{u+v}{10}$. In the first case, if $|u-v| \geq \frac{u+v}{10}$, again the paths between $I_u$ and $J_v$ must have a high transversal fluctuation of order at least $\tfrac{1}{10}k(|u|+|v|)r^{2/3}$.  And by Proposition \ref{trans_fluc_loss},
\eqref{item_22} $\leq  e^{-C(k(|u|+|v|))^{1/10}}$. Lastly, if $|u-v| \leq \frac{u+v}{10}$, because the paths between $I_u$ and $J_v$ must touch $R^{kr^{2/3}}_{0, r}$, they must also have a high transversal fluctuation of order at least $\tfrac{1}{10}k(|u|+|v|)r^{2/3}$. In this case, Proposition \ref{trans_fluc_loss3} gives \eqref{item_22} $\leq  e^{-C(k(|u|+|v|))^{1/10}}$. With this, we have finished the proof of \eqref{loc_p}.
\end{proof}

\subsubsection{Proof of Proposition \ref{p:3rd}}

Recall the boundary condition $W'$ which was an independent copy of $W$. Notice that $Z_{N}^{W',\textup{out}, r}$ is independent of $U_{j}$ for all $j\le r^{1/3}$, 
\begin{equation}
    \label{e:3rd}
   \Cov(Z^{W}_{r},Z^{W^{*,r},\textup{out},r}_{N})  = \sum_{j=1}^{r^{1/3}}\Cov(U_{j},Z^{W^{*, r},\textup{out}, r}_{N}-Z^{W',\textup{out},r}_{N}).
\end{equation}
We shall show that for $j\le r^{1/3}$
\begin{equation}
    \label{e:3bd}
    \Cov(U_j,Z^{W^{*,r},\textup{out}, r}_{N}-Z^{W',\textup{out},r}_{N})\le \frac{1}{(r-jr^{2/3}+1)^{100}} \sqrt{\Var(U_j)}\sqrt{\Var(Z^{W^{*,r},\textup{out}, r}_{N}-Z^{W',\textup{out},r}_{N})}.
\end{equation}
This indeed suffices. Note that by an argument identical to the one used to obtain \eqref{l2_terms}, we get $\Var(Z^{W^{*, r},\textup{out}, r}_{N}-Z^{W',\textup{out},r}_{N})\le C(r^{1/3}+1)^{40} r^{2/3}$, and we get that \eqref{e:3rd} is upper bounded by $Cr^{2/3}$ for some $C>0$ upon using Proposition \ref{r_proof} and \eqref{e:3bd}. 

It remains to show \eqref{e:3bd}, and the argument for this is similar to \eqref{e:2bd}. Recall the $\sigma$-algebras $\mathcal{F}_1^j$ and $\mathcal{F}_2^k$ defined in  \eqref{def_f12}.
To establish \eqref{e:3bd}, we repeat the argument for \eqref{e:2bd} but with $\mathcal{F}_2^{k-1}$ replaced by $\mathcal{F}_2$ which is generated by 
$\mathcal{F}^{r^{1/3}}_{2}$ together with $\log W'_{i+1}-\log W'_{i}$ for all $i\in \Z$. Observe that conditional on $\mathcal{F}':=\sigma(\mathcal{F}^{j+1}_1,\mathcal{F}_2)$, now, $Z^{W^{*,r},\textup{out}, r}_{N}-Z^{W',\textup{out},r}_{N}$ is a function of $X_{i}$ restricted to the index set $\Z\setminus [-r+1,r]$ which has distance $r-jr^{2/3}+1$ from the set of indices that $U_{j}$ is a function of. By employing an argument identical to the one presented for \eqref{e:2bd}, we establish \eqref{e:3bd}, albeit without providing specific details. This completes the proof of the upper bound of Theorem \ref{thm_r_small}.

\subsection{Lower bound in Theorem \ref{thm_r_small}}\label{r_small_low}

Given a polymer with boundary $\mathcal{L}_0$, let us define a $\sigma$-algebra $\mathcal{F}$, generated by the weights $\{Y_{\bf z}\}_{{\bf z} \in \mathcal{L}_{0}^>}$ and $\{X_{k}=\log W_{k} - \log W_{k-1} \}_{k \not \in [1, r^{2/3}]}$.
Next, we will define the following events which are $\mathcal{F}$-measurable.

In the definition below,  $M$ is a large positive constant, and the parameters $r\geq r_0$ and $N \geq r$. The events $\mathcal{B}_1$ and $\mathcal{B}_2$ say that the bulk free energy from $(0,0)$ to $(r,r)$ and $(N, N)$ can not be too large or too small
\begin{align*}
\mathcal{B}_1 &= \{|\log \wt{Z}_{0, r} - \Lambda_r| \leq M^{20} r^{1/3}\};\\
\mathcal{B}_2 &= \{|\log \wt{Z}_{0, N} - \Lambda_N| \leq M^{20} N^{1/3}\}.
\end{align*}
The events $\mathcal{B}_3$ and $\mathcal{B}_4$ say that the bulk free energy starting anywhere in a segment around the origin can not be too different from starting from the origin,
\begin{align*}
\mathcal{B}_3 &= \Big\{\max_{0\leq k \leq M r^{2/3}} \log \wt{Z}_{\olsi{k},r} - \log \wt{Z}_{0, r} \leq M^{20} r^{1/3} \Big\};\\
\mathcal{B}_4 &= \Big\{\max_{0\leq k \leq M N^{2/3}}\log \wt{Z}_{\olsi{k},N} - \log \wt{Z}_{0, N} \leq M^{20} N^{1/3}\Big\}.
\end{align*}
We also have control on the local fluctuation of $\log \wt{Z}_{\bbullet, N}$ on the scale $r$,
\begin{align*}
\mathcal{B}_5 &= \Big\{\max_{0 \leq k \leq {r^{2/3}}}\log \wt{Z}_{\olsi{k}. N} - \min_{0\leq j \leq Mr^{2/3}}\log \wt{Z}_{\olsi{j}, N} \leq M^{20} r^{1/3}\Big\}.
\end{align*}
Recall the definition of $\tau$ at the beginning of Section \ref{est_poly_bdry}. The next two events $\mathcal{B}_6$ and $\mathcal{B}_7$ say that the restricted free energy for the paths with non-positive exit times (with boundary) can not be too large.
\begin{align*}
\mathcal{B}_6 &=  \{\fgs_{ r}(\tau\leq 0) - \Lambda_r \leq M^{20}r^{1/3}\};\\
\mathcal{B}_7 &=  \{\fgs_{N}(\tau\leq 0) - \Lambda_N \leq M^{20}N^{1/3}\}.
\end{align*}
Finally, we define two rare events $\mathcal{B}_8$ and $\mathcal{B}_9$ which say the boundary weights along parts of the boundary are large:
\begin{align*}
\mathcal{B}_8 &= \Big\{\log W_{Mr^{2/3}} - \log W_{r^{2/3}}  \geq M^{30} r^{1/3} \Big\};\\
\mathcal{B}_{9} &= \Big\{\log W_{MN^{2/3}}  - \log W_{r^{2/3}} \geq M^{30} N^{1/3} \Big\}.
\end{align*}

Let us denote the intersection of the nine events defined above as 
\begin{equation}\label{def_B}
\mathcal{B} = \bigcap_{i=1}^{9} \mathcal{B}_i,
\end{equation}
and the following lemma lower bounds the probability of the intersection.
\begin{lemma}\label{lbB}
There exists a positive constant $M_0$ such that for each $M\geq M_0$, there exist positive constants $r_0$ and $\epsilon$ such that for each $r\geq r_0$ and $N \geq r$, we have 
$$\mathbb{P}(\mathcal{B}) \geq \epsilon.$$
\end{lemma}
\begin{proof}
Note that $\mathcal{B}_{i}$, for $i= 1, \dots, 8$ are all high probability events if $M$ is large. To see this fact, $\mathbb{P}(\mathcal{B}_1^c \cup \mathcal{B}_2^c)$ can be upper bounded using Proposition \ref{ptl_upper} and Proposition \ref{low_ub}.
The probability $\mathbb{P}(\mathcal{B}_3^c \cup \mathcal{B}_4^c) $ can be upper bounded using Proposition \ref{max_all_t1}.
The probability $\mathbb{P}(\mathcal{B}_5^c)$ can be upper bounded using the triangle inequality, Proposition \ref{max_all_t1} and Proposition \ref{min1}. The probability $\mathbb{P}(\mathcal{B}_6^c\cup \mathcal{B}_7^c)$ can be upper bounded using Proposition \ref{b_up_tail}. 
Thus, 
\begin{equation}\label{bd17}\mathbb{P}\Big(\bigcap_{i=1}^7 \mathcal{B}_i\Big) \geq 0.9
\end{equation}
provided that $M$ is fixed large enough.

Next, we are going to use the assumptions \eqref{mix}, \eqref{FKG} and \eqref{low_bd} to finish the proof of the lower bound. Let $\mathcal{F}'$ denote the $\sigma$-algebra generated by all bulk weights $\{Y_{\bf z}\}_{z\in \mathcal{L}_0^{>}}$. For simplicity of the notation, let us denote by $\mathcal{B}_{i,j}=\mathcal{B}_{i}\cap \mathcal{B}_{j}$, $\mathcal{B}_{i,j,k}=\mathcal{B}_{i}\cap \mathcal{B}_{j}\cap \mathcal{B}_{k}$ and so on. Then we have 
\begin{align}
\mathbb{P}(\mathcal{B}_{1, \dots, 5} \cap \mathcal{B}_{6,7} \cap \mathcal{B}_{8,9})
&=\mathbb{E}[\mathbb{P}(\mathcal{B}_{1, \dots, 5} \cap \mathcal{B}_{6,7} \cap \mathcal{B}_{8,9}|\mathcal{F}')]\nonumber\\
&=\mathbb{E}[\mathbbm{1}_{\mathcal{B}_{1, \dots, 5} }\mathbb{P}(\mathcal{B}_{6,7} \cap \mathcal{B}_{8,9}|\mathcal{F}')].\label{cond_F'}
\end{align}
Now, notice that when conditioned on $\mathcal{F}'$, $\mathbbm{1}_{\mathcal{B}_{6,7}}$ and $\mathbbm{1}_{\mathcal{B}_{8,9}}$ are functions of the increments $\{X_i\}_{i \leq 0}$ and $\{X_i\}_{i \geq r+1}$, respectively. Furthermore, the conditional variances $\Var(\mathbbm{1}_{\mathcal{B}_{6,7}}|\mathcal{F}')$ and $\Var(\mathbbm{1}_{\mathcal{B}_{8,9}}|\mathcal{F}')$ are bounded above by $1$. Then, by assumption \eqref{mix}, we have 
\begin{align*}
\eqref{cond_F'}&\leq \mathbb{E}[\mathbbm{1}_{\mathcal{B}_{1, \dots, 5} }(\mathbb{P}(\mathcal{B}_{6,7}| \mathcal{F}')\mathbb{P}(\mathcal{B}_{8,9}|\mathcal{F}') - 1/r^{50})] \qquad \text{ }\\
& \leq \mathbb{E}[\mathbbm{1}_{\mathcal{B}_{1, \dots, 5} }(\mathbb{P}(\mathcal{B}_{6,7}| \mathcal{F}')\mathbb{P}(\mathcal{B}_{8,9})]- 1/r^{50} \\
&\leq \mathbb{P}(\mathcal{B}_{8,9}) \mathbb{P}(\mathcal{B}_{1, \dots, 7})- 1/r^{50}\\
&\leq \epsilon_0^2\cdot 0.9- 1/r^{50} \end{align*}
where the last inequality follows from \eqref{bd17} and the fact that $\mathbb{P}(\mathcal{B}_{8,9}) \geq \epsilon_0^2$ by assumptions \eqref{low_bd} and  \eqref{FKG}.
Finally, by fixing $r_0$ sufficiently large, we have completed the proof.
\end{proof}
Next, we have the following variance bound for the difference of two free energies when we condition on $\mathcal{F}$.

\begin{lemma}\label{diff_small}
There exist positive constants $C_1, C_2, M_0$  such that for each $M\geq M_0$, there exists a positive constant $r_0$ and $\mathcal{B}'\subset \mathcal{B}$ with $\mathbb{P}(\mathcal{B}') \geq \tfrac{1}{2} \mathbb{P}(\mathcal{B})$ such that for each $r\geq r_0$, $N\geq r$ and each $\omega \in \mathcal{B}'$, we have 
\begin{align*}\Var\big( \fgs_{N} - \fgs_{N} (\tau \geq r^{2/3})\big|\mathcal{F}\big)(\omega) &\leq e^{-C_1M} r^{2/3};\\ \Var\big(\log W_{r^{2/3}}\big|\mathcal{F}\big)(\omega) &\geq C_2 r^{2/3}.
\end{align*}

\end{lemma}

\begin{proof}
First, let us define $\mathcal{B}'$. By assumptions \eqref{cond_up_bd}, \eqref{cond_low_tail} and \eqref{as_var}, there exists a subset $\mathcal{B}' \subset \mathcal{B}$ with probability $\mathbb{P}(\mathcal B') \geq \mathbb{P}(\mathcal B)- 100/\log r$ such that for each $\omega \in \mathcal{B}'$, we have 
\begin{align*}
\mathbb{P}\big(\max_{0\leq k \leq N^{2/3}} \log W_{k} \geq tN^{1/3}\big|\mathcal{F}\big)(\omega) &\leq e^{-C't}\\
\mathbb{P}\big(\log W_{N^{2/3}} \leq -tN^{1/3}\big|\mathcal{F}\big)(\omega) &\leq e^{-C't}\\
\Var\big(\log W_{r^{2/3}} | \mathcal{F}\big)(\omega) &\geq Cr^{2/3}.
\end{align*}
By fixing $r_0$ sufficiently large, we have that $\mathbb{P}(\mathcal B') \geq \tfrac{1}{2}\mathbb{P}(\mathcal B)$, and the second conditional variance lower bound from the lemma follows directly. 

Next, we show the first conditional variance upper bound in the lemma.
Note that because of the event  $\mathcal{B}'$, it holds that
\begin{align*}
\fgs_{N} (\tau \leq 0 ) &\leq \Lambda_N +  M^{20} N^{1/3}\\
\fgs_{N} (\tau \geq r^{2/3}) &\geq \log W_{MN^{2/3}} +\fg_{\olsi{MN^{2/3}}, N}   \geq \Lambda_N- 2M^{20} N^{1/3} + M^{30} N^{1/3}.
\end{align*}
Thus, we have 
\begin{equation}\label{left_small}
\mathbb{P}\Big(\fgs_{N} (\tau \leq 0 ) < \fgs_{N} (\tau \geq r^{2/3})\Big|\mathcal{F}\Big)(\omega) = 1 \qquad \text{ for }\omega \in \mathcal{B}'.
\end{equation}

Let us define $$\mathcal{L}_0^{r^{2/3}, +} =\Big\{{\bf a } \in \mathcal{L}_0^{r^{2/3}} : {\bf a} \cdot {\bf e}_1 \geq 0\Big\}.$$
We have the following bound on the conditional variance. For the calculation below, the fact 
$$
\fgs_{N} \leq \max\Big\{\fgs_{N}(\tau \geq 1), \fgs_{N}(\tau \leq 0)\Big\} + \log 2
$$ 
and \eqref{left_small} gives us the second inequality. For each $\omega \in \mathcal{B}$, it holds that
\begin{align}
&\Var\Big( \fgs_{N} - \fgs_{N} (\tau \geq r^{2/3})\Big|\mathcal{F}\Big)(\omega)\nonumber\\
&\leq \mathbb{E}\Big(\Big( \fgs_{N} - \fgs_{N} (\tau \geq r^{2/3})\Big)^2\Big|\mathcal{F}\Big)(\omega)\nonumber\\
&\leq \mathbb{E}\Big( \Big(\fgs_{N}(\tau \geq 1)  - \fgs_{N} (\tau \geq r^{2/3}) + \log 2\Big)^2\Big|\mathcal{F}\Big)(\omega)\nonumber\\
&\leq \mathbb{E}\Big( \Big(\max\Big\{\fgs_{N}(1\leq \tau <r^{2/3}), \fgs_{N}(\tau \geq r^{2/3})  \Big\} - \fgs_{N} (\tau \geq r^{2/3}) +2\log 2 \Big)^2\Big|\mathcal{F}\Big)(\omega)\nonumber\\
&= \mathbb{E}\Big( \Big(\max\Big\{\fgs_{N}(1\leq \tau <r^{2/3})- \fgs_{N}(\tau \geq r^{2/3}), 0   \Big\} +4 \Big)^2\Big|\mathcal{F}\Big)(\omega)\label{max_zero}
\end{align}

Let us define the event $$\mathcal{A} = \Big\{\fgs_{N} (1\leq \tau < r^{2/3} ) > \fgs_{N} (\tau \geq r^{2/3}) \Big\}.$$
then \eqref{max_zero} can be upper bounded as  as
\begin{align}
\eqref{max_zero} & \leq \mathbb{E}\Big( \Big(\fgs_{N} (1\leq \tau < r^{2/3} ) - \fgs_{N} (\tau \geq r^{2/3}) + 4\Big)^2 \mathbbm{1}_{\mathcal{A}} \Big|\mathcal{F}\Big)(\omega) + 16.\label{has_const}
\end{align}
Next, using the general fact 
$$\mathbb{E}[(X+4)^2] \leq \mathbb{E}[X^2] + 8\mathbb{E}[X^2]^{1/2} + 16,$$
it suffices for us to ignore the constants $4$ and $16$ in \eqref{has_const} and directly bound the expression \eqref{no_const} below. 
In the calculation for \eqref{no_const}, the second inequality follows because on our event $\mathcal{B}$,  
$$\max_{1 \leq k \leq r^{2/3}} \log \wt{Z}_{\olsi{k}, N}  -\log \wt{Z}_{\olsi{r^{2/3}}, N} \leq M^{20} r^{1/3},$$
and for each $\omega \in \mathcal{B}'$,
\begin{align}
&\mathbb{E}\Big( \Big(\fgs_{N} (1\leq \tau < r^{2/3} ) - \fgs_{N} (\tau \geq r^{2/3}) \Big)^2 \mathbbm{1}_{\mathcal{A}} \Big|\mathcal{F}\Big)(\omega) \label{no_const}\\
&\leq \mathbb{E}\Big( \Big(\max_{1 \leq k \leq r^{2/3}} \log W_{k} + \max_{1 \leq k \leq r^{2/3}} \log \wt{Z}_{\olsi{k}, N} - \log Z^{W}_{N}(\tau = r^{2/3}) \Big )^2\mathbbm{1}_{\mathcal{A}} \Big|\mathcal{F}\Big)(\omega)\nonumber\\
& \leq  \mathbb{E}\Big( \max_{1 \leq k \leq r^{2/3}} \log W_{k} - \log W_{r^{2/3}} + 2M^{2}r^{1/3} \Big)^2 \mathbbm{1}_{\mathcal{A}} \Big|\mathcal{F}\Big)(\omega)\nonumber\\
& \leq \sqrt{ \mathbb{E}\Big( \Big(\max_{1 \leq k \leq r^{2/3}} \log W_{k} - \log W_{r^{2/3}}+ 2M^{2}r^{1/3}\Big)^4\Big|\mathcal{F}\Big)(\omega)}\sqrt{\mathbb{P}(\mathcal{A}|\mathcal{F})(\omega)}. \label{cs_eq}
\end{align}

Next, we will bound the conditional expectation and the conditional probability in \eqref{cs_eq}. 
For the conditional expectation, by the definition of $\mathcal{B}'$ (defined at the beginning of the proof) and the fact that the term inside the fourth power is non-negative,
the square root of the conditional expectation above is bounded by  $CM^{4} r^{2/3}$.
For the conditional probability, let us define another set $\mathcal{A}'$ which contains $\mathcal{A}$,
$$\mathcal{A}' = \Big\{\max_{1 \leq k \leq r^{2/3}} \log W_{k} + \max_{1 \leq k \leq r^{2/3}} \log \wt{Z}_{\olsi{k}, N}   + \log r > \log W_{Mr^{2/3}}  + \log \wt{Z}_{\olsi{Mr^{2/3}}, N}\Big\}.$$
On the event $\mathcal{B}'$, it holds that 
\begin{align*}
\log W_{Mr^{2/3}} - \log W_{r^{2/3}} & \geq M^{30} r^{1/3},
\\
\max_{1 \leq k \leq r^{2/3}} \log \wt{Z}_{\olsi{k}, N}  - \log \wt{Z}_{\olsi{Mr^{2/3}}, N} &\leq M^{20} r^{1/3}.
\end{align*}
Again by the definition of $\mathcal{B}'$, for $\omega \in \mathcal{B}'$, we have
\begin{align*}
\mathbb{P}(\mathcal{A}|\mathcal{F})(\omega)
&\leq \mathbb{P}(\mathcal{A}'|\mathcal{F})(\omega) 
 \leq \mathbb{P}\Big(\max_{1 \leq k \leq r^{2/3}} \log W_{k} - \log W_{r^{2/3}}\geq \tfrac{1}{2}M^{30} r^{1/3}|\mathcal{F}\Big)(\omega) \leq e^{-CM^{30}}.
\end{align*}
With this, we have finished the proof of our lemma. 
\end{proof}

Recall the $\sigma$-algebra $\mathcal{F}$ which was defined at the beginning of this section and the $\mathcal{F}$-measurable event $\mathcal{B}'$ defined in Lemma \ref{diff_small}. We shall show the following proposition.
\begin{proposition}\label{cond_lb}
There exists positive constants $C, M_0$ such that for each $M\geq M_0$, there exists a $r_0$ such that for each $r_0 \leq r \leq N$, 
$$\Cov(\fgs_{r}, \fgs_{N}| \mathcal{F})(\omega)  \geq Cr^{2/3} \qquad \text{ for each $\omega \in \mathcal{B}'$}.$$ 
\end{proposition}

\begin{proof}
We bound the conditional covariance using the Cauchy-Schwarz inequality as follows.
\begin{align}
&\Cov(\fgs_{r}, \fgs_{N}|\mathcal{F})\nonumber\\
& = \Cov\Big(\fgs_{r}(\tau\geq r^{2/3}), \fgs_{N}(\tau\geq r^{2/3})\Big|\mathcal{F}\Big) \nonumber\\
&\qquad \qquad  + \Cov\Big(\fgs_{r}- \fgs_{r}(\tau\geq r^{2/3}), \fgs_{N} - \fgs_{N}(\tau\geq r^{2/3})\Big|\mathcal{F}\Big)\nonumber\\
& \qquad \qquad \qquad \qquad  + \Cov\Big(\fgs_{r} - \fgs_{r}(\tau\geq r^{2/3}), \fgs_{N}(\tau\geq r^{2/3})\Big|\mathcal{F}\Big) \nonumber\\
&\qquad \qquad \qquad \qquad  + \Cov\Big(\fgs_{r}(\tau\geq r^{2/3}), \fgs_{N} - \fgs_{N}(\tau\geq r^{2/3})\Big|\mathcal{F}\Big)\nonumber\\
& \geq \Cov\Big(\fgs_{r}(\tau\geq r^{2/3}), \fgs_{N}(\tau\geq r^{2/3})\Big|\mathcal{F}\Big) \label{term1}\\
&\qquad  - \sqrt{\Var\Big( \fgs_{r} - \fgs_{r} (\tau \geq r^{2/3})\Big|\mathcal{F}\Big)}\sqrt{\Var\Big(\fgs_{N} - \fgs_{N}(\tau\geq r^{2/3})\Big| \mathcal{F}\Big)} \label{term2}\\
& \qquad \qquad  - \sqrt{\Var\Big(\fgs_{r} -\fgs_{r}(\tau\geq r^{2/3})\Big|\mathcal{F}\Big)}\sqrt{\Var\Big( \fgs_{N}(\tau\geq r^{2/3})\Big|\mathcal{F}\Big)}\label{term3}\\
&\qquad \qquad  - \sqrt{\Var\Big(\fgs_{r}(\tau\geq r^{2/3})\Big|\mathcal{F}\Big)}\sqrt{\Var\Big(\fgs_{N} - \fgs_{N}(\tau\geq r^{2/3})\Big|\mathcal{F}\Big)}\label{term4}
\end{align}
We note that \eqref{term1} is lower bounded by $Cr^{2/3}$, since both free energies can be rewritten as 
$$\log W_{r^{2/3}} + [\text{some $\mathcal{F}$-measurable r.v.}],$$ 
and hence the covariance is equal to $\Var(\log W_{r^{2/3}}|\mathcal{F})$ which is lower bounded by $Cr^{2/3}$ on $\mathcal{B}'$ by Lemma \ref{diff_small}.

Again by Lemma \ref{diff_small}, on the event $\mathcal{B}'$, \eqref{term2}, \eqref{term3} and \eqref{term4} (without the negative signs in the front) are all upper bounded by $Ce^{-C'M}r^{2/3}$. Hence, $$\Cov(\fgs_{r}, \fgs_{N}| \mathcal{F})(\omega)  \geq Cr^{2/3} - C'e^{-C''M}r^{2/3} \qquad \text{ for each $\omega \in \mathcal{B}'$},$$with this, we have finished the proof of this proposition.
\end{proof}

We can finally complete the proof of the lower bound in Theorem \ref{thm_r_small}.

\begin{proof}[Proof of Theorem \ref{thm_r_small}, lower bound]
Using Lemma \ref{lbB}, Proposition \ref{cond_lb} and applying the FKG inequality assumption in \textbf{Assumption B3}  twice, we have 
\begin{align*}
&\mathbb{E}[\log Z^{W}_{ N}\log Z^{W}_{r}] \\& =\mathbb{E}\Big[\mathbb{E}[\log Z^{W}_{ N}\log Z^{W}_{r}|\mathcal{F}]\Big]\\
& = \int_{\mathcal{B}'} \mathbb{E}[\log Z^{W}_{ N}\log Z^W_{r}|\mathcal{F}] \,d\mathbb P + \int_{\mathcal{B}'^c} \mathbb{E}[\log Z^{W}_{ N}\log Z^{W}_{r}|\mathcal{F}] \,d\mathbb P\\
&\geq \int_{\mathcal{B}'} \mathbb{E}[\log Z^W_{N}|\mathcal{F}] \mathbb{E}[\log Z^{W}_{r}|\mathcal{F}]\,d\mathbb P + C\epsilon r^{2/3} +  \int_{\mathcal{B}'^c} \mathbb{E}[\log Z^{W}_{ N}|\mathcal{F}] \mathbb{E}[\log Z^{W}_{r}|\mathcal{F}]\,d\mathbb P\\
&=\mathbb{E}\Big[\mathbb{E}[\log Z^{W}_{ N}|\mathcal{F}]\mathbb{E}[\log Z^{W}_{r}|\mathcal{F}]\Big] + C\epsilon r^{2/3}\\
&\geq \mathbb{E}[ \log Z^{W}_{ N} ] \mathbb{E}[\log Z^{W}_{r} ] + C\epsilon r^{2/3}.
\end{align*}
Note that in the first inequality above we use the fact that conditional on $\mathcal{F}$, $\log Z_{N}^{W}$ and $\log Z^{W}_{r}$ are both increasing in $\{X_{i}=\log W_{i}-\log W_{i-1}\}_{i=1}^{r^{2/3}}$. By assumption \eqref{FKG}, the conditional law of $\{X_{i}\}_{i=1}^{r^{2/3}}$ satisfied the FKG inequality and therefore
$\mathbb{E}[\log Z^{W}_{ N}\log Z^{W}_{r}|\mathcal{F}]\ge \mathbb{E}[\log Z^W_{N}|\mathcal{F}] \mathbb{E}[\log Z^{W}_{r}|\mathcal{F}]$.
For the second inequality above, recall the sequence $\mathcal{X}$ as defined in \eqref{FKG}. Note that $\mathbb{E}[\log Z_{N}^{W}|\mathcal{F}]$ and $\mathbb{E}[\log Z_{r}^{W}|\mathcal{F}]$ are both increasing functions of $\{Y_{\mathbf{v}}\}_{{\bf v}\in \mathcal{L}_0^{>}}$ as well as $\mathcal{X}$. Note $\{Y_{\mathbf{v}}\}_{{\bf v}\in \mathcal{L}_0^{>}}$ are independent, so they satisfy the FKG inequality. By assumption \eqref{FKG}, the marginal on $\mathcal{X}$ also satisfies the FKG inequality. Because $\{Y_{\mathbf{v}}\}_{{\bf v}\in \mathcal{L}_0^{>}}$ and $\mathcal{X}$ are independent, using the well-known fact that product of two measures each of which satisfy the FKG inequality also satisfies the FKG inequality (see e.g.\ \cite{Kem77}), it follows that $$\mathbb{E}\big[\mathbb{E}[\log Z^{W}_{ N}|\mathcal{F}]\mathbb{E}[\log Z^{W}_{r}|\mathcal{F}]\big]\ge \mathbb{E}[\mathbb{E}[\log Z^{W}_{ N}|\mathcal{F}]]\cdot \mathbb{E}[\mathbb{E}[\log Z^{W}_{r}|\mathcal{F}]]=\mathbb{E}[ \log Z^{W}_{N} ] \mathbb{E}[\log Z^{W}_{r}].$$ 
This shows $\Cov(\log Z^{W}_{N} , \log Z^{W}_{r}) \geq Cr^{2/3}$, hence the lower bound in our theorem holds.
\end{proof}

\section{Stationary inverse-gamma polymer}
\label{stat_poly}

We start by defining the (ratio) stationary inverse-gamma polymer with a general down-right boundary. 
Let $\mathcal{Y}_{\bf 0} = \{{\bf y}_i\}_{i\in \mathbb{Z}}$ be a bi-infinite down-right path going through $(0,0)$, which means the increment ${\bf y}_{i} - {\bf y}_{i-1} \in \{{\bf e}_1, -{\bf e}_2\}$. And without the loss of generality, let ${\bf y}_0 = (0,0)$.

Next, we define the weights for the stationary polymer. Let $\{Y_{\bf z}\}_{{\bf z}\in \mathcal{Y}_{\bf 0}^>}$ be a collection of i.i.d.~random variables with distribution $\textup{Ga}^{-1}(\mu)$, and they will be the bulk weights of the polymer. Independent of the bulk weights, the boundary edge weights $G = \{G_{{\bf y}_{i}, {\bf y}_{i-1}}\}_{i\in \mathbb{Z}}$ are attached to the collection of unit edges $\{[\![{\bf y}_{i}, {\bf y}_{i-1}]\!]\}_{i\in \mathbb{Z}}$. The distribution of $G$ is given by  
\begin{alignat*}{2}
G_{{\bf y}_{i}, {\bf y}_{i-1}} &\sim \textup{Ga}^{-1}(\mu-\rho) &&\qquad \text{if }{\bf y}_{i} - {\bf y}_{i-1} = {\bf e}_1\\
G_{{\bf y}_{i}, {\bf y}_{i-1}} &\sim \textup{Ga}^{-1}(\rho) &&\qquad \text{if }{\bf y}_{i} - {\bf y}_{i-1} = -{\bf e}_2,
\end{alignat*}
and all $G_{{\bf y}_i, {\bf y}_{i-1}}$ are independent. This special choice of weights is referred to as \textit{$\rho$-boundary weights}.

Next, we will define the partition function. Set $H^{\mathcal{Y}_{\bf 0}, \rho}_{\bf 0} = 1$,.
For each ${\bf y}_m$ with $m > 0$, define
$$H^{\mathcal{Y}_{\bf 0}, \rho}_{0, {\bf y}_m} = \prod_{n=1}^m \wt{G}_{{\bf y}_{n-1}, {\bf y}_n} \qquad \text{where } \wt{G}_{{\bf y}_{n-1}, {\bf y}_n}  = \begin{cases}
G_{{\bf y}_{n-1}, {\bf y}_n} \quad & \text{if ${\bf y}_n-{\bf y}_{n-1} = {\bf e}_1$,}\\
1/G_{{\bf y}_{n-1}, {\bf y}_n} \quad & \text{if ${\bf y}_n-{\bf y}_{n-1} = - \bf e_2$.}
\end{cases}$$ 
For each ${\bf y}_{m}$ with $m < 0$, define
$$H^{\mathcal{Y}_{\bf 0}, \rho}_{0,{\bf y}_{m}} = \prod_{n=0}^{m+1} \wt{G}_{{\bf y}_n, {\bf y}_{n-1}} \qquad \text{where } \wt{G}_{{\bf y}_n, {\bf y}_{n-1}}  = \begin{cases}
1/G_{{\bf y}_n, {\bf y}_{n-1}} \quad & \text{if ${\bf y}_n-{\bf y}_{n-1} = \bf e_1$,}\\
G_{{\bf y}_n, {\bf y}_{n-1}} \quad & \text{if ${\bf y}_n-{\bf y}_{n-1} = -\bf e_2$.}
\end{cases}$$ 

Recall the bulk partition function $\wt Z$ defined in Section \ref{def_poly_ad},  the partition function with the $\rho$-boundary weights is then defined by 
$$Z^{\mathcal{Y}_{\bf 0}, \rho}_{0, {\bf v}} = \sum_{k\in \mathbb{Z}} H^{\mathcal{Y}_{\bf 0}, \rho}_{{\bf y}_k} \cdot \wt{Z}_{{\bf y}_k, {\bf v}} \qquad \text{ for ${\bf v}\in \mathcal{Y}_{\bf 0}^\geq$}.$$
The above definition can be generalized to arbitrary starting point ${\bf a}$ and used to define the partition function $Z^{ \mathcal{Y}_{\bf a}, \rho}_{{\bf a}, {\bf v}}$.

The name (ratio) stationary inverse-gamma polymer is justified by the next theorem, which was stated originally for s southwest boundary (where $\mathcal{Y}_{\bf 0}$ is formed by the boundary of the first quadrant). The proof uses a now well-known ``corner flipping" induction, and the exact same proof also applies to the general down-right boundary.
\begin{theorem}[{\cite[Thm.\ 3.3]{poly2} and \cite[Eqn.\ (3.6)]{Geo-etal-15}}]
\label{stat} 
Fix $\rho \in (0,\mu)$. For each ${\bf v}\in \mathcal{\mathcal{Y}_{\bf a}}^{\geq}  $,
 we have 
$$\frac{Z^{\mathcal{Y}_{\bf a}, \rho}_{{\bf a}, {\bf v}+{\bf e}_1}}{Z^{\mathcal{Y}_{\bf a}, \rho}_{{\bf a}, {\bf v}}} \sim \textup{Ga}^{-1}(\mu-\rho), \quad \frac{Z^{\mathcal{Y}_{\bf a}, \rho}_{{\bf a}, {\bf v}+{\bf e}_2}}{Z^{\mathcal{Y}_{\bf a}, \rho}_{{\bf a}, {\bf v}}}\sim \textup{Ga}^{-1}(\rho),$$
and the \textit{dual weights} $\{\wc{Y}^\rho_{\bf z}\}_{{\bf z}\in \mathcal{Y}_{\bf a}^\geq}$ is defined as
$$\wc{Y}^\rho_{\bf v} = \frac1{Z^{\mathcal{Y}_{\bf a}, \rho}_{{\bf a},{\bf v} +{\bf e}_1}/Z^{\mathcal{Y}_{\bf a}, \rho}_{{\bf a},{\bf v}}+Z^{\mathcal{Y}_{\bf a}, \rho}_{{\bf a},{\bf v}+{\bf e}_2}/Z^{\mathcal{Y}_{\bf a}, \rho}_{{\bf a},{\bf v}}}\sim\textup{Ga}^{-1}(\mu).$$
Let $\mathcal{W}=\{{\bf w}_i\}_{i\in \mathcal I}$ be any finite or infinite down-right path in $\mathcal{Y}_{\bf a}^>$, indexed by an interval $\mathcal I\subset\Z$. Then, the nearest-neighbor  ratios $\{Z^{\mathcal{Y}_{\bf a}, \rho}_{{\bf a}, {\bf w}_{i+1}}/Z^{\mathcal{Y}_{\bf a}, \rho}_{{\bf a}, {\bf w}_{i}}\}$ along the path and the dual weights below $\mathcal{W}$ $\bigl\{\wc{Y}^{\rho}_{\bf v}:{\bf v}\in \mathcal{Y}_{\bf a}^{\geq}\cap\mathcal{W}^<\bigr\}$ are mutually independent.
\end{theorem}



\subsection{Two special boundaries $\mathcal{S}$ and $\mathcal{S}'$} \label{special_S}

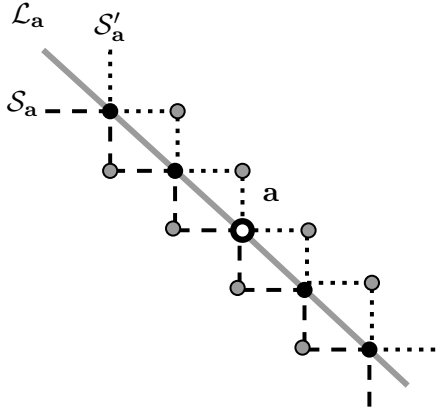
\begin{figure}
\begin{center}

\tikzset{every picture/.style={line width=0.75pt}} 

\begin{tikzpicture}[x=0.75pt,y=0.75pt,yscale=-1,xscale=1]

\draw  [dash pattern={on 5.63pt off 4.5pt}][line width=1.5]  (82.73,103.1) -- (115.47,103.1) -- (115.47,133.2) ;
\draw  [dash pattern={on 5.63pt off 4.5pt}][line width=1.5]  (115.47,133.2) -- (148.2,133.2) -- (148.2,163.3) ;
\draw  [dash pattern={on 5.63pt off 4.5pt}][line width=1.5]  (148,163.2) -- (180.73,163.2) -- (180.73,193.3) ;
\draw  [dash pattern={on 5.63pt off 4.5pt}][line width=1.5]  (180.73,193.3) -- (213.47,193.3) -- (213.47,223.4) ;
\draw  [dash pattern={on 5.63pt off 4.5pt}][line width=1.5]  (213.47,223.4) -- (246.2,223.4) -- (246.2,253.5) ;
\draw [color={rgb, 255:red, 155; green, 155; blue, 155 }  ,draw opacity=1 ][line width=2.25]    (81.67,72.17) -- (181.51,163.25) -- (265.5,241.55) ;
\draw  [fill={rgb, 255:red, 0; green, 0; blue, 0 }  ,fill opacity=1 ] (176.45,163.3) .. controls (176.45,160.13) and (179.03,157.55) .. (182.2,157.55) .. controls (185.37,157.55) and (187.95,160.13) .. (187.95,163.3) .. controls (187.95,166.47) and (185.37,169.05) .. (182.2,169.05) .. controls (179.03,169.05) and (176.45,166.47) .. (176.45,163.3) -- cycle ;
\draw  [fill={rgb, 255:red, 0; green, 0; blue, 0 }  ,fill opacity=1 ] (209.89,193.3) .. controls (209.89,191.32) and (211.49,189.72) .. (213.47,189.72) .. controls (215.44,189.72) and (217.05,191.32) .. (217.05,193.3) .. controls (217.05,195.28) and (215.44,196.88) .. (213.47,196.88) .. controls (211.49,196.88) and (209.89,195.28) .. (209.89,193.3) -- cycle ;
\draw  [fill={rgb, 255:red, 0; green, 0; blue, 0 }  ,fill opacity=1 ] (242.62,223.4) .. controls (242.62,221.42) and (244.22,219.82) .. (246.2,219.82) .. controls (248.18,219.82) and (249.78,221.42) .. (249.78,223.4) .. controls (249.78,225.38) and (248.18,226.98) .. (246.2,226.98) .. controls (244.22,226.98) and (242.62,225.38) .. (242.62,223.4) -- cycle ;
\draw  [fill={rgb, 255:red, 0; green, 0; blue, 0 }  ,fill opacity=1 ] (144.62,133.2) .. controls (144.62,131.22) and (146.22,129.62) .. (148.2,129.62) .. controls (150.18,129.62) and (151.78,131.22) .. (151.78,133.2) .. controls (151.78,135.18) and (150.18,136.78) .. (148.2,136.78) .. controls (146.22,136.78) and (144.62,135.18) .. (144.62,133.2) -- cycle ;
\draw  [fill={rgb, 255:red, 0; green, 0; blue, 0 }  ,fill opacity=1 ] (111.89,103.1) .. controls (111.89,101.12) and (113.49,99.52) .. (115.47,99.52) .. controls (117.44,99.52) and (119.05,101.12) .. (119.05,103.1) .. controls (119.05,105.08) and (117.44,106.68) .. (115.47,106.68) .. controls (113.49,106.68) and (111.89,105.08) .. (111.89,103.1) -- cycle ;
\draw  [fill={rgb, 255:red, 155; green, 155; blue, 155 }  ,fill opacity=1 ] (111.89,133.2) .. controls (111.89,131.22) and (113.49,129.62) .. (115.47,129.62) .. controls (117.44,129.62) and (119.05,131.22) .. (119.05,133.2) .. controls (119.05,135.18) and (117.44,136.78) .. (115.47,136.78) .. controls (113.49,136.78) and (111.89,135.18) .. (111.89,133.2) -- cycle ;
\draw  [fill={rgb, 255:red, 155; green, 155; blue, 155 }  ,fill opacity=1 ] (143.62,162.3) .. controls (143.62,160.32) and (145.22,158.72) .. (147.2,158.72) .. controls (149.18,158.72) and (150.78,160.32) .. (150.78,162.3) .. controls (150.78,164.28) and (149.18,165.88) .. (147.2,165.88) .. controls (145.22,165.88) and (143.62,164.28) .. (143.62,162.3) -- cycle ;
\draw  [fill={rgb, 255:red, 155; green, 155; blue, 155 }  ,fill opacity=1 ] (176.15,192.3) .. controls (176.15,190.32) and (177.76,188.72) .. (179.73,188.72) .. controls (181.71,188.72) and (183.31,190.32) .. (183.31,192.3) .. controls (183.31,194.28) and (181.71,195.88) .. (179.73,195.88) .. controls (177.76,195.88) and (176.15,194.28) .. (176.15,192.3) -- cycle ;
\draw  [fill={rgb, 255:red, 155; green, 155; blue, 155 }  ,fill opacity=1 ] (208.89,222.4) .. controls (208.89,220.42) and (210.49,218.82) .. (212.47,218.82) .. controls (214.44,218.82) and (216.05,220.42) .. (216.05,222.4) .. controls (216.05,224.38) and (214.44,225.98) .. (212.47,225.98) .. controls (210.49,225.98) and (208.89,224.38) .. (208.89,222.4) -- cycle ;
\draw  [dash pattern={on 1.69pt off 2.76pt}][line width=1.5]  (116.73,103.1) -- (149.47,103.1) -- (149.47,133.2) ;
\draw  [dash pattern={on 1.69pt off 2.76pt}][line width=1.5]  (149.47,133.2) -- (182.2,133.2) -- (182.2,163.3) ;
\draw  [dash pattern={on 1.69pt off 2.76pt}][line width=1.5]  (182,163.2) -- (214.73,163.2) -- (214.73,193.3) ;
\draw  [dash pattern={on 1.69pt off 2.76pt}][line width=1.5]  (214.73,189.72) -- (247.47,189.72) -- (247.47,219.82) ;
\draw  [dash pattern={on 1.69pt off 2.76pt}][line width=1.5]  (247.47,223.4) -- (280.2,223.4) -- (280.2,224.4) ;
\draw  [fill={rgb, 255:red, 155; green, 155; blue, 155 }  ,fill opacity=1 ] (211.93,163.25) .. controls (211.93,161.27) and (213.53,159.67) .. (215.51,159.67) .. controls (217.49,159.67) and (219.09,161.27) .. (219.09,163.25) .. controls (219.09,165.23) and (217.49,166.83) .. (215.51,166.83) .. controls (213.53,166.83) and (211.93,165.23) .. (211.93,163.25) -- cycle ;
\draw  [fill={rgb, 255:red, 155; green, 155; blue, 155 }  ,fill opacity=1 ] (243.89,189.72) .. controls (243.89,187.74) and (245.49,186.14) .. (247.47,186.14) .. controls (249.44,186.14) and (251.05,187.74) .. (251.05,189.72) .. controls (251.05,191.7) and (249.44,193.3) .. (247.47,193.3) .. controls (245.49,193.3) and (243.89,191.7) .. (243.89,189.72) -- cycle ;
\draw  [fill={rgb, 255:red, 155; green, 155; blue, 155 }  ,fill opacity=1 ] (178.62,133.2) .. controls (178.62,131.22) and (180.22,129.62) .. (182.2,129.62) .. controls (184.18,129.62) and (185.78,131.22) .. (185.78,133.2) .. controls (185.78,135.18) and (184.18,136.78) .. (182.2,136.78) .. controls (180.22,136.78) and (178.62,135.18) .. (178.62,133.2) -- cycle ;
\draw  [fill={rgb, 255:red, 155; green, 155; blue, 155 }  ,fill opacity=1 ] (145.89,103.1) .. controls (145.89,101.12) and (147.49,99.52) .. (149.47,99.52) .. controls (151.44,99.52) and (153.05,101.12) .. (153.05,103.1) .. controls (153.05,105.08) and (151.44,106.68) .. (149.47,106.68) .. controls (147.49,106.68) and (145.89,105.08) .. (145.89,103.1) -- cycle ;
\draw  [fill={rgb, 255:red, 255; green, 255; blue, 255 }  ,fill opacity=1 ] (178.62,163.3) .. controls (178.62,161.32) and (180.22,159.72) .. (182.2,159.72) .. controls (184.18,159.72) and (185.78,161.32) .. (185.78,163.3) .. controls (185.78,165.28) and (184.18,166.88) .. (182.2,166.88) .. controls (180.22,166.88) and (178.62,165.28) .. (178.62,163.3) -- cycle ;
\draw  [dash pattern={on 1.69pt off 2.76pt}][line width=1.5]  (114.67,73) -- (115.47,73) -- (115.47,103.1) ;

\draw (191,140.56) node [anchor=north west][inner sep=0.75pt]    {$\bf a$};
\draw (61.9,91.56) node [anchor=north west][inner sep=0.75pt]    {$\mathcal{S}_{\bf a}$};
\draw (64,46.96) node [anchor=north west][inner sep=0.75pt]    {$\mathcal{L}_{\bf a}$};
\draw (105.9,51.56) node [anchor=north west][inner sep=0.75pt]    {$\mathcal{S} '_{\bf a}$};

\end{tikzpicture}

\captionsetup{width=0.8\textwidth}
\caption{An illustration of  $\mathcal{L}_{\bf a }$, $\mathcal{S}_{\bf a }$ and $\mathcal{S}'_{\bf a}$.} \label{S and L}
\end{center}
\end{figure}

Let $\mathcal{S}_{\bf a}$ (and $\mathcal{S}'_{\bf a}$) denote the bi-infinite down-right staircase path that has ${\bf a}$ as one of its upper right (or lower left) corners (illustrated in Figure \ref{S and L}),
\begin{align*}\label{stair_path}
\mathcal{S}_{\bf a} &= \{\dots, {\bf a} - {\bf e}_1 +\bf e_2, {\bf a} - \bf e_1,{\bf a} , {\bf a}+ \bf e_1, {\bf a}+ \bf e_1 - e_2, \dots \},\\
\mathcal{S}_{\bf a}' &= \{\dots, {\bf a} - {\bf e}_1 +\bf e_2, {\bf a} + \bf e_2,{\bf a} + \bf e_1 , {\bf a}+ \bf e_1 - {\bf e}_2, \dots \}.
\end{align*}
For the rest of the paper, the general down-right boundary $\mathcal{Y}_{\bf a}$ will be chosen to be $\mathcal{S}_{\bf a}$. The other boundary  $\mathcal{S}_{\bf a}'$ will only appear in the proof of Theorem \ref{stat_r_small}. With this in mind, let us simplify our notation by omitting the boundary $\mathcal{S}_{\bf a}$ in the superscript of the partition function,
$$Z^\rho_{\bf a, \bbullet} = Z^{\mathcal{S}_{\bf a},\rho}_{\bf a, \bbullet}.$$ 
And the corresponding quenched measure will be denoted as $Q^\rho_{\bf a, \bbullet}.$

Lastly, we note that $Z^\rho_{\bf a, \bbullet}$ matches the  definition of a polymer with a given initial condition defined in Section \ref{def_poly_ad}. Let us define $\mathcal{S}_{\bf 0} = \{{\bf s_i}\}_{i \in \mathbb{Z}}$ with $\bf s_0 = 0$, then the initial condition $\{W_k\}_{k\in \mathbb{Z}}$ define  in Section \ref{def_poly_ad} is given by 
$$W_k = H^{\mathcal{S}_{\bf 0}, \rho}_{{\bf s}_{2k}} \qquad \text{ for } k \in \mathbb{Z}.$$
Because of this, the stationary boundary is also called the stationary initial condition. And we will also often refer to the staircase boundary  $\mathcal{S}_{\bf 0}$  as the anti-diagonal boundary $\mathcal{L}_0$, which is different from the $\mathcal{S}_{\bf 0}'$ staircase boundary.

\subsection{Exit time estimate}

Define the \textit{characteristic direction} as a function of $\rho \in (0, \mu)$
\begin{equation}\label{char_dir}
{\boldsymbol\xi}[\rho] = \Big(\frac{\Psi_1(\rho)}{\Psi_1(\rho) + \Psi_1(\mu -\rho)}\,,\, \frac{\Psi_1(\mu -\rho)}{\Psi_1(\rho) + \Psi_1(\mu -\rho)}\Big)
\end{equation}
where $\Psi_1$ is the trigamma function defined by 
$\Psi_1(z) = \frac{d^2}{dz^2}\Gamma(z)$. 
The sampled paths between $(0,0)$ and $N\boldsymbol{\xi}[\rho]$ tend to stay on the anti-diagonal boundary for  at most order $N^{2/3}$ number of steps. This is the statement of the next proposition obtained from \cite{ras-sep-she-}. Note the stationary polymer defined in \cite{ras-sep-she-} has a southwest boundary, but the following result is actually proved for the anti-diagonal boundary and then showed that it implies the same result for the southwest boundary. 
\begin{proposition}[{\cite[Lemma 4.6]{ras-sep-she-}}] \label{stat_exit_est}
Fix $\varepsilon\in(0,\mu/2)$. There exist positive constants $C_1, C_2$, $r_0$, $N_0$
depending only on $\varepsilon$ such that for all ${{\rho}} \in [\varepsilon, \mu-\varepsilon]$, $N\geq N_0$ and $r\ge r_0$, we have  
\begin{align*}
&\mathbb{P}\Big(Q^\rho_{0, N\boldsymbol\xi[\rho]}\{|\tau| > rN^{2/3}\} \geq e^{-C_1 r^2 N^{1/3}}\Big) \leq e^{-C_2r^{3}}.
\end{align*}
\end{proposition}

\section{Temporal correlation for the stationary polymer using duality}
\label{s:duality}

In this section, we give a stronger upper bound for the time correlation for the stationary polymer when $r$ is smaller than $N/2$.
The availability of this result (with simpler proof) can be attributed to a special duality arising from the stationary boundary. Essentially, the dual weights $\wc{Y}^{\mu/2}$ defined in Theorem \ref{stat} follow the i.i.d.~inverse-gamma distribution. 
\begin{theorem}\label{stat_r_small}  
For any $r\leq N/2$, we have 
$$ \textup{$\mathbb{C}$ov}\Big(\log Z^{\mu/2}_{0, r}, \log Z^{\mu/2}_{0, N}\Big) \leq \Var\Big(\log Z^{\mu/2}_{0, r}\Big).$$ 
\end{theorem}
This is a stronger version of the upper bound in Theorem \ref{thm_r_small} since by Theorem \ref{var_s}, it holds that  $\Var\big(\log Z^{\mu/2}_{0, r}\big) \leq Cr^{2/3}$.

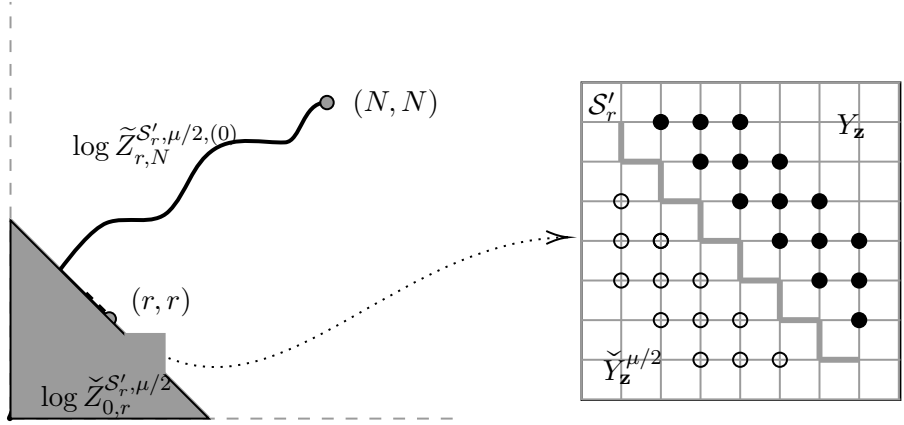
\begin{figure}[t]
\begin{center}

\tikzset{every picture/.style={line width=0.75pt}} 

\begin{tikzpicture}[x=0.75pt,y=0.75pt,yscale=-1,xscale=1]

\draw  [color={rgb, 255:red, 155; green, 155; blue, 155 }  ,draw opacity=0.5 ][dash pattern={on 4.5pt off 4.5pt}] (414,240.7) -- (191,240.7) -- (191,30.4) ;
\draw [color={rgb, 255:red, 155; green, 155; blue, 155 }  ,draw opacity=0.8 ][line width=1.5]    (190.8,140.4) -- (290.8,240.4) ;
\draw  [dash pattern={on 2.53pt off 3.02pt}][line width=2.25] [line join = round][line cap = round] (190.67,240.5) .. controls (194.02,233.8) and (194.17,225.9) .. (196.67,218.83) .. controls (198.21,214.47) and (207.34,220.39) .. (210,215.83) .. controls (212.22,212.03) and (211.9,198.51) .. (212,198.17) .. controls (212.56,196.34) and (214.85,195.44) .. (216.67,194.83) .. controls (218.35,194.27) and (221.23,196.44) .. (222,194.83) .. controls (224.51,189.58) and (221.53,181.56) .. (226,177.83) ;
\draw  [line width=1.5] [line join = round][line cap = round] (215.88,165.73) .. controls (221.19,158.65) and (231.81,144.86) .. (239.65,141.87) .. controls (249.68,138.05) and (262.81,144.02) .. (270.88,136.02) .. controls (280.01,126.95) and (282.21,112.14) .. (293.94,104.63) .. controls (304.62,97.79) and (318.49,101.75) .. (329.42,101.28) .. controls (337.44,100.94) and (338.26,81.19) .. (350,81.19) ;
\draw  [fill={rgb, 255:red, 155; green, 155; blue, 155 }  ,fill opacity=1 ] (347.24,81.18) .. controls (347.24,79.31) and (348.75,77.8) .. (350.62,77.8) .. controls (352.49,77.8) and (354,79.31) .. (354,81.18) .. controls (354,83.05) and (352.49,84.56) .. (350.62,84.56) .. controls (348.75,84.56) and (347.24,83.05) .. (347.24,81.18) -- cycle ;
\draw [line width=1.5]    (215.84,165.76) -- (240.8,190.4) ;
\draw  [fill={rgb, 255:red, 155; green, 155; blue, 155 }  ,fill opacity=1 ] (237.42,190.4) .. controls (237.42,188.53) and (238.93,187.02) .. (240.8,187.02) .. controls (242.67,187.02) and (244.18,188.53) .. (244.18,190.4) .. controls (244.18,192.27) and (242.67,193.78) .. (240.8,193.78) .. controls (238.93,193.78) and (237.42,192.27) .. (237.42,190.4) -- cycle ;
\draw [line width=3]  [dash pattern={on 3.38pt off 3.27pt}]  (228.32,178.08) -- (240.8,190.4) ;
\draw  [color={rgb, 255:red, 0; green, 0; blue, 0 }  ,draw opacity=0 ][fill={rgb, 255:red, 155; green, 155; blue, 155 }  ,fill opacity=0.33 ] (191,140.7) -- (291,240.7) -- (191,240.7) -- cycle ;
\draw  [draw opacity=0] (479,71) -- (639.8,71) -- (639.8,231.3) -- (479,231.3) -- cycle ; \draw  [color={rgb, 255:red, 155; green, 155; blue, 155 }  ,draw opacity=0.3 ] (479,71) -- (479,231.3)(499,71) -- (499,231.3)(519,71) -- (519,231.3)(539,71) -- (539,231.3)(559,71) -- (559,231.3)(579,71) -- (579,231.3)(599,71) -- (599,231.3)(619,71) -- (619,231.3)(639,71) -- (639,231.3) ; \draw  [color={rgb, 255:red, 155; green, 155; blue, 155 }  ,draw opacity=0.3 ] (479,71) -- (639.8,71)(479,91) -- (639.8,91)(479,111) -- (639.8,111)(479,131) -- (639.8,131)(479,151) -- (639.8,151)(479,171) -- (639.8,171)(479,191) -- (639.8,191)(479,211) -- (639.8,211)(479,231) -- (639.8,231) ; \draw  [color={rgb, 255:red, 155; green, 155; blue, 155 }  ,draw opacity=0.3 ]  ;
\draw [color={rgb, 255:red, 155; green, 155; blue, 155 }  ,draw opacity=1 ][line width=2.25]    (519,111) -- (519,131) ;
\draw [color={rgb, 255:red, 155; green, 155; blue, 155 }  ,draw opacity=1 ][line width=2.25]    (539,131) -- (539,151) ;
\draw [color={rgb, 255:red, 155; green, 155; blue, 155 }  ,draw opacity=1 ][line width=2.25]    (559,151) -- (559,171) ;
\draw [color={rgb, 255:red, 155; green, 155; blue, 155 }  ,draw opacity=1 ][line width=2.25]    (579,171) -- (579,191) ;
\draw [color={rgb, 255:red, 155; green, 155; blue, 155 }  ,draw opacity=1 ][line width=2.25]    (599,191) -- (599,211) ;
\draw [color={rgb, 255:red, 155; green, 155; blue, 155 }  ,draw opacity=1 ][line width=2.25]    (539,131) -- (519,131) ;
\draw [color={rgb, 255:red, 155; green, 155; blue, 155 }  ,draw opacity=1 ][line width=2.25]    (559,151) -- (539,151) ;
\draw [color={rgb, 255:red, 155; green, 155; blue, 155 }  ,draw opacity=1 ][line width=2.25]    (579,171) -- (559,171) ;
\draw [color={rgb, 255:red, 155; green, 155; blue, 155 }  ,draw opacity=1 ][line width=2.25]    (599,191) -- (579,191) ;
\draw [color={rgb, 255:red, 155; green, 155; blue, 155 }  ,draw opacity=1 ][line width=2.25]    (619,211) -- (599,211) ;
\draw [color={rgb, 255:red, 155; green, 155; blue, 155 }  ,draw opacity=1 ][line width=2.25]    (499,91) -- (499,111) ;
\draw [color={rgb, 255:red, 155; green, 155; blue, 155 }  ,draw opacity=1 ][line width=2.25]    (519,111) -- (499,111) ;
\draw   (502.58,131) .. controls (502.58,129.03) and (500.97,127.42) .. (499,127.42) .. controls (497.03,127.42) and (495.42,129.03) .. (495.42,131) .. controls (495.42,132.97) and (497.03,134.58) .. (499,134.58) .. controls (500.97,134.58) and (502.58,132.97) .. (502.58,131) -- cycle ;
\draw   (522.58,151) .. controls (522.58,149.03) and (520.97,147.42) .. (519,147.42) .. controls (517.03,147.42) and (515.42,149.03) .. (515.42,151) .. controls (515.42,152.97) and (517.03,154.58) .. (519,154.58) .. controls (520.97,154.58) and (522.58,152.97) .. (522.58,151) -- cycle ;
\draw   (522.58,151) .. controls (522.58,149.03) and (520.97,147.42) .. (519,147.42) .. controls (517.03,147.42) and (515.42,149.03) .. (515.42,151) .. controls (515.42,152.97) and (517.03,154.58) .. (519,154.58) .. controls (520.97,154.58) and (522.58,152.97) .. (522.58,151) -- cycle ;
\draw   (542.58,171) .. controls (542.58,169.03) and (540.97,167.42) .. (539,167.42) .. controls (537.03,167.42) and (535.42,169.03) .. (535.42,171) .. controls (535.42,172.97) and (537.03,174.58) .. (539,174.58) .. controls (540.97,174.58) and (542.58,172.97) .. (542.58,171) -- cycle ;
\draw   (562.58,191) .. controls (562.58,189.03) and (560.97,187.42) .. (559,187.42) .. controls (557.03,187.42) and (555.42,189.03) .. (555.42,191) .. controls (555.42,192.97) and (557.03,194.58) .. (559,194.58) .. controls (560.97,194.58) and (562.58,192.97) .. (562.58,191) -- cycle ;
\draw   (582.58,211) .. controls (582.58,209.03) and (580.97,207.42) .. (579,207.42) .. controls (577.03,207.42) and (575.42,209.03) .. (575.42,211) .. controls (575.42,212.97) and (577.03,214.58) .. (579,214.58) .. controls (580.97,214.58) and (582.58,212.97) .. (582.58,211) -- cycle ;
\draw   (502.58,151) .. controls (502.58,149.03) and (500.97,147.42) .. (499,147.42) .. controls (497.03,147.42) and (495.42,149.03) .. (495.42,151) .. controls (495.42,152.97) and (497.03,154.58) .. (499,154.58) .. controls (500.97,154.58) and (502.58,152.97) .. (502.58,151) -- cycle ;
\draw   (522.58,171) .. controls (522.58,169.03) and (520.97,167.42) .. (519,167.42) .. controls (517.03,167.42) and (515.42,169.03) .. (515.42,171) .. controls (515.42,172.97) and (517.03,174.58) .. (519,174.58) .. controls (520.97,174.58) and (522.58,172.97) .. (522.58,171) -- cycle ;
\draw   (542.58,191) .. controls (542.58,189.03) and (540.97,187.42) .. (539,187.42) .. controls (537.03,187.42) and (535.42,189.03) .. (535.42,191) .. controls (535.42,192.97) and (537.03,194.58) .. (539,194.58) .. controls (540.97,194.58) and (542.58,192.97) .. (542.58,191) -- cycle ;
\draw   (562.58,211) .. controls (562.58,209.03) and (560.97,207.42) .. (559,207.42) .. controls (557.03,207.42) and (555.42,209.03) .. (555.42,211) .. controls (555.42,212.97) and (557.03,214.58) .. (559,214.58) .. controls (560.97,214.58) and (562.58,212.97) .. (562.58,211) -- cycle ;
\draw   (502.58,171) .. controls (502.58,169.03) and (500.97,167.42) .. (499,167.42) .. controls (497.03,167.42) and (495.42,169.03) .. (495.42,171) .. controls (495.42,172.97) and (497.03,174.58) .. (499,174.58) .. controls (500.97,174.58) and (502.58,172.97) .. (502.58,171) -- cycle ;
\draw   (522.58,191) .. controls (522.58,189.03) and (520.97,187.42) .. (519,187.42) .. controls (517.03,187.42) and (515.42,189.03) .. (515.42,191) .. controls (515.42,192.97) and (517.03,194.58) .. (519,194.58) .. controls (520.97,194.58) and (522.58,192.97) .. (522.58,191) -- cycle ;
\draw   (542.58,211) .. controls (542.58,209.03) and (540.97,207.42) .. (539,207.42) .. controls (537.03,207.42) and (535.42,209.03) .. (535.42,211) .. controls (535.42,212.97) and (537.03,214.58) .. (539,214.58) .. controls (540.97,214.58) and (542.58,212.97) .. (542.58,211) -- cycle ;
\draw  [fill={rgb, 255:red, 0; green, 0; blue, 0 }  ,fill opacity=1 ] (522.58,91) .. controls (522.58,89.03) and (520.97,87.42) .. (519,87.42) .. controls (517.03,87.42) and (515.42,89.03) .. (515.42,91) .. controls (515.42,92.97) and (517.03,94.58) .. (519,94.58) .. controls (520.97,94.58) and (522.58,92.97) .. (522.58,91) -- cycle ;
\draw  [fill={rgb, 255:red, 0; green, 0; blue, 0 }  ,fill opacity=1 ] (542.58,111) .. controls (542.58,109.03) and (540.97,107.42) .. (539,107.42) .. controls (537.03,107.42) and (535.42,109.03) .. (535.42,111) .. controls (535.42,112.97) and (537.03,114.58) .. (539,114.58) .. controls (540.97,114.58) and (542.58,112.97) .. (542.58,111) -- cycle ;
\draw  [fill={rgb, 255:red, 0; green, 0; blue, 0 }  ,fill opacity=1 ] (562.58,131) .. controls (562.58,129.03) and (560.97,127.42) .. (559,127.42) .. controls (557.03,127.42) and (555.42,129.03) .. (555.42,131) .. controls (555.42,132.97) and (557.03,134.58) .. (559,134.58) .. controls (560.97,134.58) and (562.58,132.97) .. (562.58,131) -- cycle ;
\draw  [fill={rgb, 255:red, 0; green, 0; blue, 0 }  ,fill opacity=1 ] (582.58,151) .. controls (582.58,149.03) and (580.97,147.42) .. (579,147.42) .. controls (577.03,147.42) and (575.42,149.03) .. (575.42,151) .. controls (575.42,152.97) and (577.03,154.58) .. (579,154.58) .. controls (580.97,154.58) and (582.58,152.97) .. (582.58,151) -- cycle ;
\draw  [fill={rgb, 255:red, 0; green, 0; blue, 0 }  ,fill opacity=1 ] (602.58,171) .. controls (602.58,169.03) and (600.97,167.42) .. (599,167.42) .. controls (597.03,167.42) and (595.42,169.03) .. (595.42,171) .. controls (595.42,172.97) and (597.03,174.58) .. (599,174.58) .. controls (600.97,174.58) and (602.58,172.97) .. (602.58,171) -- cycle ;
\draw  [fill={rgb, 255:red, 0; green, 0; blue, 0 }  ,fill opacity=1 ] (622.58,191) .. controls (622.58,189.03) and (620.97,187.42) .. (619,187.42) .. controls (617.03,187.42) and (615.42,189.03) .. (615.42,191) .. controls (615.42,192.97) and (617.03,194.58) .. (619,194.58) .. controls (620.97,194.58) and (622.58,192.97) .. (622.58,191) -- cycle ;
\draw  [fill={rgb, 255:red, 0; green, 0; blue, 0 }  ,fill opacity=1 ] (542.58,91) .. controls (542.58,89.03) and (540.97,87.42) .. (539,87.42) .. controls (537.03,87.42) and (535.42,89.03) .. (535.42,91) .. controls (535.42,92.97) and (537.03,94.58) .. (539,94.58) .. controls (540.97,94.58) and (542.58,92.97) .. (542.58,91) -- cycle ;
\draw  [fill={rgb, 255:red, 0; green, 0; blue, 0 }  ,fill opacity=1 ] (562.58,111) .. controls (562.58,109.03) and (560.97,107.42) .. (559,107.42) .. controls (557.03,107.42) and (555.42,109.03) .. (555.42,111) .. controls (555.42,112.97) and (557.03,114.58) .. (559,114.58) .. controls (560.97,114.58) and (562.58,112.97) .. (562.58,111) -- cycle ;
\draw  [fill={rgb, 255:red, 0; green, 0; blue, 0 }  ,fill opacity=1 ] (582.58,131) .. controls (582.58,129.03) and (580.97,127.42) .. (579,127.42) .. controls (577.03,127.42) and (575.42,129.03) .. (575.42,131) .. controls (575.42,132.97) and (577.03,134.58) .. (579,134.58) .. controls (580.97,134.58) and (582.58,132.97) .. (582.58,131) -- cycle ;
\draw  [fill={rgb, 255:red, 0; green, 0; blue, 0 }  ,fill opacity=1 ] (602.58,151) .. controls (602.58,149.03) and (600.97,147.42) .. (599,147.42) .. controls (597.03,147.42) and (595.42,149.03) .. (595.42,151) .. controls (595.42,152.97) and (597.03,154.58) .. (599,154.58) .. controls (600.97,154.58) and (602.58,152.97) .. (602.58,151) -- cycle ;
\draw  [fill={rgb, 255:red, 0; green, 0; blue, 0 }  ,fill opacity=1 ] (622.58,171) .. controls (622.58,169.03) and (620.97,167.42) .. (619,167.42) .. controls (617.03,167.42) and (615.42,169.03) .. (615.42,171) .. controls (615.42,172.97) and (617.03,174.58) .. (619,174.58) .. controls (620.97,174.58) and (622.58,172.97) .. (622.58,171) -- cycle ;
\draw  [fill={rgb, 255:red, 0; green, 0; blue, 0 }  ,fill opacity=1 ] (562.58,91) .. controls (562.58,89.03) and (560.97,87.42) .. (559,87.42) .. controls (557.03,87.42) and (555.42,89.03) .. (555.42,91) .. controls (555.42,92.97) and (557.03,94.58) .. (559,94.58) .. controls (560.97,94.58) and (562.58,92.97) .. (562.58,91) -- cycle ;
\draw  [fill={rgb, 255:red, 0; green, 0; blue, 0 }  ,fill opacity=1 ] (582.58,111) .. controls (582.58,109.03) and (580.97,107.42) .. (579,107.42) .. controls (577.03,107.42) and (575.42,109.03) .. (575.42,111) .. controls (575.42,112.97) and (577.03,114.58) .. (579,114.58) .. controls (580.97,114.58) and (582.58,112.97) .. (582.58,111) -- cycle ;
\draw  [fill={rgb, 255:red, 0; green, 0; blue, 0 }  ,fill opacity=1 ] (602.58,131) .. controls (602.58,129.03) and (600.97,127.42) .. (599,127.42) .. controls (597.03,127.42) and (595.42,129.03) .. (595.42,131) .. controls (595.42,132.97) and (597.03,134.58) .. (599,134.58) .. controls (600.97,134.58) and (602.58,132.97) .. (602.58,131) -- cycle ;
\draw  [fill={rgb, 255:red, 0; green, 0; blue, 0 }  ,fill opacity=1 ] (622.58,151) .. controls (622.58,149.03) and (620.97,147.42) .. (619,147.42) .. controls (617.03,147.42) and (615.42,149.03) .. (615.42,151) .. controls (615.42,152.97) and (617.03,154.58) .. (619,154.58) .. controls (620.97,154.58) and (622.58,152.97) .. (622.58,151) -- cycle ;
\draw  [color={rgb, 255:red, 155; green, 155; blue, 155 }  ,draw opacity=1 ][fill={rgb, 255:red, 155; green, 155; blue, 155 }  ,fill opacity=0.2 ] (248.87,197.97) -- (268.73,197.97) -- (268.73,217.83) -- (248.87,217.83) -- cycle ;
\draw  [dash pattern={on 0.84pt off 2.51pt}]  (270.7,210.15) .. controls (332.46,234.89) and (388.9,150.72) .. (470.9,149.22) ;
\draw [shift={(472.13,149.2)}, rotate = 179.54] [color={rgb, 255:red, 0; green, 0; blue, 0 }  ][line width=0.75]    (10.93,-3.29) .. controls (6.95,-1.4) and (3.31,-0.3) .. (0,0) .. controls (3.31,0.3) and (6.95,1.4) .. (10.93,3.29)   ;

\draw (250,172.96) node [anchor=north west][inner sep=0.75pt]    {$( r,r)$};
\draw (362,72.76) node [anchor=north west][inner sep=0.75pt]    {$( N,N)$};
\draw (220.8,90.36) node [anchor=north west][inner sep=0.75pt]    {$\fg_{r,N}^{\mathcal{S}_r', \mu/2, (0)}$};
\draw (204.8,216.76) node [anchor=north west][inner sep=0.75pt]    {$\log \widecheck{Z}_{0,r}^{\mathcal{S}_r', \mu/2}$};
\draw (488.2,203.9) node [anchor=north west][inner sep=0.75pt]    {$\wc{Y}^{\mu/2}_{\bf z}$};
\draw (606.2,85.4) node [anchor=north west][inner sep=0.75pt]    {$Y_{\bf z}$};
\draw (481,74.4) node [anchor=north west][inner sep=0.75pt]    {$\mathcal{S}_{r}'$};

\end{tikzpicture}

\captionsetup{width=0.8\textwidth}
\caption{An illustration of the two partition functions defined with the original and dual weights  } \label{two_lpp}

\end{center}
\end{figure}

\begin{proof}

Fix $0\leq r \leq N/2$. First, we will define a nested free energy.
Fix the staircase $\mathcal{S}_{r}'= \{\bf s'_i\}_{i\in \mathbb{Z}}$ with ${\bf s}_0 = (r,r)$. Let us define the boundary weight $G$ on $\mathcal{S}_r'$ to be  
\begin{alignat*}{2}
G_{{\bf s}'_{i}, {\bf s}'_{i-1}} &={Z^{\mu/2}_{0, {\bf s}'_{i}}}/{Z^{\mu/2}_{0, {\bf s}'_{i-1}}} &&\qquad \text{if }{\bf s}'_{i} - {\bf s}'_{i-1} = {\bf e}_1\\
G_{{\bf s}'_{i}, {\bf s}'_{i-1}} &= {Z^{\mu/2}_{0, {\bf s}'_{i-1}}}/{Z^{\mu/2}_{0, {\bf s}'_{i}}} &&\qquad \text{if }{\bf s}'_{i} - {\bf s}'_{i-1} = -{\bf e}_2,
\end{alignat*}

Theorem \ref{stat} states that the weights $G$ are independent $\text{Ga}^{-1}(\mu/2)$ random variables, attached to the $e_1$ and $e_2$ edges in $\mathcal{S}_r'$. In addition, $G$, $\{Y_{\bf z}\}_{z\in {\mathcal{S}_r'}^>}$ and $\{\wc{Y}^{\mu/2}_{\bf z}\}_{z\in {\mathcal{S}_r'}^<}$ are all independent.

With $G$ and bulk weights $\{Y_{\bf z}\}_{z\in {\mathcal{S}_r'}^>}$, we can define a stationary polymer between $(r,r)$ and $(N,N)$, whose partition function is denoted by $Z^{\mathcal{S}'_r, \mu/2, (0)}_{r, N}$.
Then, Lemma A.6 from \cite{ras-sep-she-} precisely states that 
\begin{equation}\label{nest_eq}
\log Z^{\mu/2}_{0, N} = \log Z^{\mu/2}_{0, r} + \log Z^{\mathcal{S}'_r, \mu/2, (0)}_{r, N}.
\end{equation}

With the same boundary $G$ and the dual weights $\{\wc{Y}^{\mu/2}_{\bf z}\}_{z\in {\mathcal{S}_r'}^<}$, we can define another stationary polymer between $(0,0)$ and $(r,r)$, but the boundary and the bulk are rotated by $180^\circ$. We denote its partition function as $\wc{Z}^{\mathcal{S}'_r, \mu/2}_{0, r}$. Lemma 4.3, originally presented in \cite{Geo-etal-15} for a stationary polymer with a southwest boundary, can be directly applied in this context with an identical proof, yielding the following result: \begin{equation}\label{dual_eq}
\log Z^{\mu/2}_{0, r} = \log Z^{\mathcal{S}_{0}, \mu/2}_{0, r}  = \log \wc{Z}^{\mathcal{S}'_r, \mu/2}_{0, r}.
\end{equation}

With \eqref{nest_eq} and \eqref{dual_eq}, we see that 
\begin{equation}\label{erwrite_cov}
\begin{aligned}
\Cov(\log Z^{\mu/2}_{0,r}, \log Z^{\mu/2}_{0, N}) &=  \Cov(\log Z^{\mu/2}_{0,r}, \log Z^{\mu/2}_{0, r} + \log Z^{\mathcal{S}_r', \mu/2, (0)}_{r, N}) \\
& = \Var(\log Z^{\mu/2}_{0,r}) + \Cov({\log  {Z}}^{\mu/2}_{0,r},\log Z^{\mu/2, (0)}_{r, N})\\
& = \Var(\log Z^{\mu/2}_{0,r}) + \Cov(\log \wc{Z}^{\mathcal{S}'_r, \mu/2}_{0, r},\log Z^{\mathcal{S}'_r, \mu/2, (0)}_{r, N}).
\end{aligned}
\end{equation}
The two partition functions from the last covariance term are illustrated in Figure \ref{two_lpp}. To finish the proof, we will show that 
\begin{equation}\label{negative_fkg}
\Cov(\log \wc{Z}^{\mathcal{S}_{r}', \stat}_{0,r},\log Z^{\mathcal{S}_{r}', \stat, (0)}_{r, N}) \leq 0.
\end{equation}
Note once we have this, we obtain our desired result from \eqref{erwrite_cov},
$$\Cov(\log Z^\stat_{0,r}, \log Z^\stat_{0, N}) \leq \Var(\log Z^\stat_{0,r}).$$

So now it remains to show \eqref{negative_fkg}, let $\mathcal{F}$ and $\wc\mathcal{F}$ denote the $\sigma$-algebras generated by $\{Y_{\bf z}\}_{{\bf z}\in \mathcal{S}_r^>}$ and $\{\wc{Y}^{\mu/2}_{\bf z}\}_{{\bf z}\in \mathcal{S}_r^<}$.
By the law of total covariance, 
\begin{align}
&\Cov(\log \wc{Z}^{\mathcal{S}_{r}',, \stat}_{0,r},\log Z^{\mathcal{S}_{r}',\stat, (0)}_{r, N}) \nonumber\\&= \mathbb{E}[\Cov(\log \wc{Z}^{\mathcal{S}_{r}', \stat}_{0,r},\log Z^{\mathcal{S}_{r}',\stat, (0)}_{r, N} |\mathcal{F}, \wc\mathcal{F})] + \Cov(\mathbb{E}[\log \wc{Z}^{\mathcal{S}_{r}', \stat}_{0,r}|\mathcal{F}, \wc\mathcal{F}], \mathbb{E}[\log Z^{\mathcal{S}_{r}',\stat, (0)}_{r, N}|\mathcal{F}, \wc\mathcal{F}]) \label{total_c}.
\end{align}
Since, $\log \wc{Z}^{\mathcal{S}_{r}',\stat}_{0,r}$ is independent of $\mathcal{F}$,
$$\mathbb{E}[\log \wc{Z}^{\mathcal{S}_{r}',\stat}_{0,r}|\mathcal{F}, \wc\mathcal{F}] = \mathbb{E}[\log \wc{Z}^{\mathcal{S}_{r}',\stat}_{0,r}| \wc\mathcal{F}],$$
and similarly $\log Z^{\mathcal{S}_{r}',\stat, (0)}_{r, N}$ is independent of $\wc\mathcal{F}$, so 
$$\mathbb{E}[\log Z^{\mathcal{S}_{r}',\stat, (0)}_{r, N}|\mathcal{F}, \wc\mathcal{F}] = \mathbb{E}[\log Z^{\mathcal{S}_{r}',\stat, (0)}_{r, N}|\mathcal{F}].$$
Then, by independence, the second covariance term from \eqref{total_c} is zero. And it remains to show that 
$$\mathbb{E}[\Cov(\log \wc{Z}^{\mathcal{S}_{r}',\stat}_{0,r},\log Z^{\mathcal{S}_{r}',\stat, (0)}_{r, N} |\mathcal{F}, \wc\mathcal{F})] \leq 0.$$
This holds because by definition, $\log \wc{Z}^{\mathcal{S}_{r}',\stat}_{0,r}$ and $\log Z^{\mathcal{S}_{r}',\stat, (0)}_{r, N}$ are negatively correlated in their boundary weights. This means if we increase the value of one of the edge weights along the boundary $S_r'$, one of the two free energies increases while the other one decreases. Thus, by the FKG inequality,
$$\Cov(\log \wc{Z}^{\mathcal{S}_{r}',\stat}_{0,r},\log Z^{\mathcal{S}_{r}',\stat, (0)}_{r, N} |\mathcal{F}, \wc\mathcal{F}) \leq 0.$$
We have finished showing \eqref{negative_fkg}, and this finishes the proof.

\end{proof}

\section{Proofs for the estimates in Section \ref{est_poly_bulk} and  Section \ref{est_poly_bdry}}
\label{s:est_proof}

In this section, we prove the various estimates from  Section \ref{est_poly_bulk} and  Section \ref{est_poly_bdry}. The proofs use techniques from the stationary polymer.
\addtocontents{toc}{\protect\setcounter{tocdepth}{-10}}

\subsection{Proof of Proposition \ref{min}} \label{a:min}

\begin{proof}
By symmetry and a union bound, it suffices to consider the case when $k$ is non-negative and upper bound 
\begin{equation}\label{k_est}
\mathbb{P}\Big(\min_{0\leq k \leq t^{1/20}r^{2/3}}\log \wt{Z}_{0,{N}+\olsi{k}} - \log \wt{Z}_{0, N} \leq -tr^{1/3}\Big).
\end{equation}

First, recall $\mu$ is the shape parameter of the i.i.d.~inverse-gamma environment.  Let us define $s = t^{1/20}$ and look at the case that $s\leq \tfrac{\mu}{100} N^{1/3}$. We will use the random walk comparison technique and turn this into an estimate for random walks. Recall Theorem 3.28 from \cite{bas-sep-she-23} which states that there exists an event $A$ with $\mathbb{P}(A) \geq 1- e^{-Cs^3}$, such that for each $0\leq k \leq sr^{2/3}$,
$$\log \wt{Z}_{0,{N}+\olsi{k}}\ - \log \wt{Z}_{0, N} \geq \log \tfrac{9}{10} + \sum_{i=1}^k H_i \qquad \text{ on the event } A,$$
where
$\{H_i\}$ are i.i.d. random variables whose distribution is given by the difference of two independent log-gamma random variables. More precisely, set $\eta = \mu/2 - sN^{-1/3}$, then 
$$H_i \stackrel{d}{=} U-V$$ where $U \sim  \log \textup{Ga}^{-1}(\mu-\eta), V \sim \log \textup{Ga}^{-1}(\eta) $ and $U, V$ are independent.

By an application of Taylor's theorem to the digamma functions, we obtain that $\mathbb{E}[H_i] = - \Psi_0(\mu-\rho)+ \Psi_0(\eta) \geq -C' sN^{-1/3}$ for some positive constant $C'$. Then, it holds that
\begin{align*}
\eqref{k_est} \leq e^{-t^{3/20}} +\mathbb{P}\Big(\min_{0\leq k \leq t^{1/20}r^{2/3}} \sum_{i=1}^k (H_i - \mathbb{E}[H_i]) \leq - \tfrac{1}{2}tr^{1/3} \Big) \leq e^{-t^{1/10}}
\end{align*}
where bounding the probability uses Theorem \ref{max_sub_exp} since  $\pm H_i$'s are sub-exponential random variables (this is verified in Appendix \ref{ver_weights}). 

On the other hand, if $t^{1/20} \geq \tfrac{\mu}{100}N^{1/3}$, then 
\begin{align*}
\eqref{k_est} &\leq \mathbb{P}\Big(\min_{0\leq k \leq t^{1/20}r^{2/3}}\log \wt{Z}_{0,{N}+\olsi{k}} - \log \wt{Z}_{0, N} \leq -\sqrt{t}N^{3}\Big)\\
&\leq  \mathbb{P}\Big(\min_{0\leq k \leq t^{1/20}r^{2/3}}\log \wt{Z}_{0,{N}+\olsi{k}} \leq -\tfrac{1}{2}\sqrt{t}N^{3}\Big) + \mathbb{P}\Big(\log \wt{Z}_{0, N} \geq \tfrac{1}{2}\sqrt{t}N^{3}\Big)\\
& \leq \sum_{k=0}^{t^{1/20}r^{2/3}}\mathbb{P}\Big(\log \wt{Z}_{0,{N}+\olsi{k}} \leq -\tfrac{1}{2}\sqrt{t}N^{3}\Big) + \mathbb{P}\Big(\log \wt{Z}_{0, N} \geq \tfrac{1}{2}\sqrt{t}N^{10}\Big)\\
 & \leq  t^{1/20}r^{2/3} e^{-t^{1/10}N} + e^{-t^{1/10}},
\end{align*}
where the last inequality comes from \cite[Proposition 3.8]{bas-sep-she-23} (a general version of Proposition \ref{low_ub} from this paper, which holds in all directions bounded away from the axes), and Proposition \ref{up_lb}. With this, we have finished the proof of this proposition.
\end{proof}

\subsection{Proof of Proposition \ref{max_all_t1}} \label{max_all_t1_proof}
\begin{proof}
It suffices for us to get the upper bound 
$$\mathbb{P}\Big(\max_{{\bf x} \in \mathcal{L}^a_{0}}\log \wt{Z}_{{\bf x}, N} - \log \wt{Z}_{0, N} \geq t\sqrt{a}\Big) \leq 
e^{-t^{1/10}}.
$$ 
Let $\log Z_{{\bf x}, N} = \log Y_{\bf x} + \log \wt{Z}_{{\bf x}, N}$, and applying a union bound, we have 
\begin{align*}
&\mathbb{P}\Big(\max_{{\bf x} \in \mathcal{L}^a_{0}}\log \wt{Z}_{{\bf x}, N} - \log \wt{Z}_{0, N} \geq t\sqrt{a}\Big)\\
& \leq \mathbb{P}\Big(\max_{{\bf x} \in \mathcal{L}^a_{0}}\log {Z}_{{\bf x}, N} - \min_{{\bf x} \in \mathcal{L}^a_{0}} \log Y_{\bf x} - \log {Z}_{0, N} + \log Y_{(0,0)}\geq t\sqrt{a}\Big)\\
& \leq \mathbb{P}\Big(\max_{{\bf x} \in \mathcal{L}^a_{0}}\log {Z}_{{\bf x}, N}  - \log {Z}_{0, N} \geq \tfrac{1}{2}t\sqrt{a}\Big) + \mathbb{P}\Big(- \min_{{\bf x} \in \mathcal{L}^a_{0}} \log Y_{\bf x} + \log Y_{(0,0)}\geq \tfrac{1}{2}t\sqrt{a}\Big).
\end{align*}
The first probability is bounded by $e^{-t^{1/10}}$ using Proposition \ref{max_all_t} after a $180^\circ$ rotation of the picture. The second probability is bounded by $e^{-t^{1/10}}$ as well, because as shown in Appendix \ref{ver_weights}, $\pm \log Y_{\bf x}$ are sub-exponential random variables.
\end{proof}
\subsection{Proof of Proposition \ref{exit_est}}
\begin{proof}
We may assume that $t \leq N^{1/3}/\epsilon$, since otherwise the collection of paths $\big\{|\tau_{ \mathcal{L}_N^{tN^{2/3}}}| > (1+\epsilon)tN^{2/3}\big\}$ is empty. Also, since $$\fgs_{ \mathcal{L}_N^{tN^{2/3}}}(|\tau| > (1+\epsilon)tN^{2/3})  \leq \max_{{\bf x} \in\mathcal{L}_N^{tN^{2/3}}} \fgs_{ {\bf x}}(|\tau| > (1+\epsilon)tN^{2/3}) + 10\log N,$$
we may prove the estimate for $\max_{{\bf x} \in\mathcal{L}_N^{tN^{2/3}}} \fgs_{ {\bf x}}(|\tau| > 2tN^{2/3})$.
Finally, by a union bound, we may replace $|\tau|$ with just $\tau$.

The rest of the proof follows from a union bound. First,
 we will show that there exist two positive constants $C, C'$ such that for each $(1+\epsilon)t+1 \leq s \leq (t+N^{1/3}/\epsilon)-1$, 
\begin{equation} \label{goal}
\mathbb{P}\Big(\max_{{\bf x} \in\mathcal{L}_N^{tN^{2/3}}} \fgs_{ {\bf x}}((s-1)N^{2/3} \leq \tau \leq (s+1)N^{2/3}) - \Lambda_N \geq -C's^2 N^{1/3}\Big) \leq e^{-Cs^{3/2}}.
\end{equation}
This holds because we have
\begin{align}
&\mathbb{P}\Big(\max_{{\bf x} \in\mathcal{L}_N^{tN^{2/3}}} \fgs_{ {\bf x}}((s-1)N^{2/3} \leq \tau \leq (s+1)N^{2/3}) - \Lambda_N \geq -C's^2 N^{1/3}\Big)\nonumber\\
&\leq \mathbb{P}\Big(\max_{{\bf x} \in\mathcal{L}_N^{tN^{2/3}}} \log \wt{Z}_{\mathcal{L}_{\olsi{sN^{2/3}}}^{N^{2/3}}, {\bf x}}- \Lambda_N \geq -2C's^2 N^{1/3}\Big) \label{est1}\\
& \qquad \qquad + \mathbb{P}\Big(\max_{1\leq k \leq (s+1)N^{2/3}}\fgs_{ \olsi{k}} \geq \tfrac{1}{2}C's^2 N^{1/3}\Big).\label{est2}
\end{align}
Since $s\geq (1+\epsilon)t+1$,  \eqref{est1} is bounded by $e^{-Cs^3}$ using Proposition \ref{trans_fluc_loss}, provided $C'$ in \eqref{est1} is fixed sufficiently small. And \eqref{est2} is bounded by $e^{-Cs^{3/2}}$ using assumption \eqref{up_bd}. Thus, have shown that $\eqref{goal}.$

Finally, by a union bound we complete the proof
\begin{align*}
&\mathbb{P}\Big(\log Z^{W}_{ \mathcal{L}_N^{tN^{2/3}}}(\tau > (1+\epsilon)tN^{2/3}) - \Lambda_N \geq -C't^2 N^{1/3}\Big)\\
& \leq \sum_{s = (1+\epsilon)t+1}^{(t+N^{1/3})-1}
\mathbb{P}\Big(\log Z^{W}_{ \mathcal{L}_N^{tN^{2/3}}}((s-1)N^{2/3} \leq \tau \leq (s+1)N^{2/3}) - \Lambda_N \geq -C's^2 N^{1/3}\Big)\\
& \leq  \sum_{s = (1+\epsilon)t+1}^{(t+N^{1/3})-1} e^{-Cs^{3/2}} \leq e^{-C t^{3/2}}.
\end{align*}
\end{proof}

\subsection{Proof of Proposition \ref{b_up_tail}}
\begin{proof} First, we prove $(i)$.
By Proposition \ref{exit_est}, we have
$$\mathbb{P}\Big(\fgs_{\mathcal{L}_{ N }^{tN^{2/3}}}(|\tau| > 2tN^{2/3}) - \Lambda_N \geq -C't^2 N^{1/3}\Big) \leq e^{-Ct^{3/2}}.$$
Then, it suffices for us to bound
\begin{equation}\label{in10tN}
\mathbb{P}\Big(\fgs_{ \mathcal{L}_{N }^{tN^{2/3}}}(|\tau|\leq 2tN^{2/3}) - \Lambda_N \geq \epsilon t^2 N^{1/3}\Big).
\end{equation}
By a union bound,  
\begin{align}
\eqref{in10tN} 
& \leq \mathbb{P}\Big(\fg_{\mathcal{L}_0, \mathcal{L}_N^{tN^{2/3}}} - \Lambda_N \geq \tfrac{1}{3}\epsilon t^2 N^{1/3}\Big)
\label{est01}\\
& \qquad \qquad +  \mathbb{P}\Big(\max_{1\leq |k| \leq 2tN^{2/3}}\fgs_{ \olsi{k}} \geq \tfrac{1}{3}\epsilon t^{3/2}(t^{1/2} N^{1/3})\Big).\label{est02}
\end{align}
Then, \eqref{est01} is bounded by $e^{-C\min\{t^3, t^2N^{1/3}\}}$ from Proposition \ref{ptl_upper}, and \eqref{est02} is bounded by $e^{-Ct^{3/2}}$ by assumption \eqref{up_bd}. With this, we have shown \eqref{in10tN} and finished proving the first estimate $(i)$ in our proposition. 

The second estimate follows from the exact same argument, the only changes are that the ``$t^2$" in \eqref{est01} becomes ``$t$", and ``$t^{3/2}$" from \eqref{est02} becomes ``$t^{1/2}$". So we omit the details.
\end{proof}

\subsection{Proof of Proposition \ref{b_low_tail}}
\begin{proof} 
This follows directly from the monotonicity that $\fgs_{ N} \geq \log \wt{Z}_{0, N}$ and the upper bound of the left tail for $\log \wt{Z}_{0, N}$ from  Proposition \ref{low_tail}.
\end{proof}

\subsection{Proof of Theorem \ref{var_s}}
By Proposition \ref{b_up_tail} and Proposition \ref{b_low_tail}, 
$$\big|\mathbb{E}[\log Z^W_N] - \Lambda_N\big| \leq CN^{1/3}.$$
Then, these propositions give the upper bound on the variance. 

The lower bound on the variance follows from the fact $\fgs_{ N} \geq \log \wt{Z}_{0, N}$ and Proposition \ref{up_lb}.

\subsection{Proof of Theorem \ref{exit_q}}

By a union bound, 
\begin{align*}
&\mathbb{P}\Big(\fgs_{ \mathcal{L}_N^{s}}(|\tau| > (1+\epsilon) tN^{2/3}) - \fgs_{ \mathcal{L}_N^{s}}\geq -C't^2 N^{1/3}\Big)\\
& \leq \mathbb{P}\Big(\fgs_{ \mathcal{L}_N^{tN^{2/3}}}(|\tau| > (1+\epsilon) tN^{2/3}) - \Lambda_N \geq -{2C'}t^2 N^{1/3}\Big) \\
& \qquad \qquad + \mathbb{P}\Big(\fgs_{ N} - \Lambda_N \leq -{C'}t^2 N^{1/3}\Big).
\end{align*}
By Proposition \ref{exit_est} and Proposition \ref{b_low_tail}, the two probabilities above are bounded by $e^{-Ct^{3/2}}$ provided $C'$ is sufficiently small. We have finished the proof of the theorem.

\subsection{Proof of Theorem \ref{nest}}

The proof of this Theorem  is similar to the proof of Theorem 4.8 from \cite{bas-sep-she-23}, except here we have to handle the addition of the boundary weights.
We start with a proposition that captures the local fluctuation of the free energy profile. 

\begin{proposition}\label{local_fluc}
Suppose the boundary condition $W$ satisfies assumption \eqref{up_bd}. There exist positive constants $C_1, C_2, c_0, N_0, t_0$ such that for each $N\geq N_0$, $t_0 \leq t \leq c_0 N^{1/2}$, and each integer $0\leq a \leq \tfrac{1}{2}t^{2/3}N^{2/3}$, we have
$$\mathbb{P}\Big(\fgs_{\mathcal{L}^a_{N}} - \fgs_{ N} \geq C_1t \sqrt{a}\Big) \leq 
e^{-C_2 t}.
$$ 
\end{proposition}

\begin{proof}
First, by Theorem  \ref{exit_q},
on an event with probability at least $1-e^{-Ct}$, we have
$$Q^W_{ \mathcal{L}^a_{N}}(|\tau|> t^{2/3}N^{2/3}) <1/10,$$
which implies
$$\fgs_{ \mathcal{L}^a_{N}}(|\tau|\leq t^{2/3}N^{2/3}) \geq \fgs_{ \mathcal{L}^a_{N}}(|\tau|> t^{2/3}N^{2/3}).$$
Hence, on the same event, it holds that
\begin{align*}
\fgs_{ \mathcal{L}^a_{N}} \leq \fgs_{ \mathcal{L}^a_{N}}(|\tau|\leq t^{2/3}N^{2/3}) + \log 2.
\end{align*}
Then, to prove the proposition, it suffices for us to bound the following probability
\begin{equation}\label{goal_lf}
\mathbb{P}\Big(\fgs_{\mathcal{L}^a_{N}}(|\tau| \leq t^{2/3}N^{2/3}) - \fgs_{ N}(|\tau| \leq t^{2/3}N^{2/3}) \geq C't \sqrt{a}\Big).
\end{equation}

Next, let 
\begin{equation}\label{fix_lam}
\lambda = \frac{\mu}{2} + q_0t^{2/3}N^{-1/3}
\end{equation}
where the constant $q_0$ will be chosen later below \eqref{lam_ineq}. Once $q_0$ is fixed, we may lower the value of $c_0$ from the assumption of our proposition so that $\lambda \in [\mu/2, 2\mu/3]$.  We enrich our probability space by introducing another set of independent $\lambda$-boundary weights along $\mathcal{L}_0$, which defines a (ratio) stationary polymer whose free energy is $\log Z^\lambda_{0, \bbullet}$.

Let us introduce another free energy $\log Z^*$ which uses these new $\lambda$-boundary weights on $\mathcal{L}_0 \setminus \mathcal{L}_0^{t^{2/3}N^{2/3}}$. Note that by definition,
\begin{align*}&\fgs_{\mathcal{L}^a_{N}}(|\tau| \leq t^{2/3}N^{2/3}) - \fgs_{ N}(|\tau| \leq t^{2/3}N^{2/3}) \\
& \qquad \qquad \qquad= \log Z^*_{\mathcal{L}^a_{N}}(|\tau| \leq t^{2/3}N^{2/3}) - \log Z^*_{N}(|\tau| \leq t^{2/3}N^{2/3}).
\end{align*}
Using a union bound and the fact that 
$$\log Z^*_{\mathcal{L}^a_{N}}(|\tau| \leq t^{2/3}N^{2/3}) \leq \max_{-a\leq k \leq a}\log Z^*_{N+\olsi{k}}(|\tau| \leq t^{2/3}N^{2/3}) + 10 \log a,$$
in order to show \eqref{goal_lf}, it suffices for us to bound 
\begin{equation}\label{star_est}
\mathbb{P}\Big(\max_{0\leq k \leq a}\log Z^*_{ N+\olsi{k}}(|\tau| \leq t^{2/3}N^{2/3}) - \log Z^*_{N}(|\tau| \leq t^{2/3}N^{2/3}) \geq C't \sqrt{a}\Big).
\end{equation}
To bound this, we will compare the probability above with the right tail of the running maximum of a random walk with i.i.d.~steps.

Recall $\log Z^\lambda$ is the free energy with the $\lambda$-boundary weights. In the calculation below, the inequality \eqref{prop_ineq} follows from Proposition \ref{ratio_mono2} and \eqref{prop_ineq1} follows from Proposition \ref{ratio_mono1}. For each integer $k$ with $0 \leq k \leq a$, it holds that
\begin{align}
&\log Z^*_{ N+\olsi{k}}(|\tau| \leq t^{2/3}N^{2/3}) - \log Z^*_{0, N}(|\tau| \leq t^{2/3}N^{2/3})\nonumber \\
& \leq  \log Z^*_{ N+\olsi{k}}(\tau \geq 0) - \log Z^*_{ N}(\tau\geq0) \label{prop_ineq} \\
& \leq  \log Z^*_{ N+\olsi{k}}(\tau \geq t^{2/3}N^{2/3}) - \log Z^*_{ N}(\tau\geq t^{2/3}N^{2/3}) \label{prop_ineq1} \\
& =\log Z^\lambda_{0, N+\olsi{k}}(\tau \geq t^{2/3}N^{2/3}) - \log Z^\lambda_{0, N}(\tau \geq t^{2/3}N^{2/3}).\label{lam_eq1}
\end{align}

Then, we shall show that there exists an event $A$ with probability at least $1-e^{-Ct}$ such that on the event $A$, for each $0\leq k \leq a$ it holds that
\begin{equation}\label{lam_eq2}
\eqref{lam_eq1} \leq  \log Z^\lambda_{0, N+\olsi{k}} - \log Z^\lambda_{0, N} + \log 2.\end{equation}
Before proving \eqref{lam_eq2}, we see that  \eqref{star_est} can then be bounded as
$$\eqref{star_est} \leq \mathbb{P}\Big(\max_{0\leq k \leq a}\log Z^\lambda_{0, N+\olsi{k}} - \log Z^\lambda_{0, N} \geq C't\sqrt{a}\Big) + \mathbb{P}(A^c).$$
By Theorem \ref{stat}, $\{\log Z^\lambda_{0, N+\olsi{k}} - \log Z^\lambda_{0, N} \}_{0\leq k\leq a}$ is a random walk with i.i.d.~steps, whose step distribution is $\log (\text{Ga}^{-1}(\mu-\lambda)) - \log (\text{Ga}^{-1} (\lambda))$. By Taylor's theorem, the expectation of the steps of this random walk is bounded inside a closed interval $[-c'q_0t^{2/3}N^{-1/3}, c'q_0t^{2/3}N^{-1/3}]$, where $c'$ is a constant depending only on $\mu$. We fix the  constant $C'$ in \eqref{star_est} sufficiently large (depending on $q_0$ and $\mu$) so that 
$$ac'q_0t^{2/3}N^{2/3} \leq \tfrac{1}{2}C't\sqrt{a}.$$
Then, in the calculation below, we normalize the random walk $\{\log Z^\lambda_{0, N+\olsi{k}} - \log Z^\lambda_{0, N} \}_{0\leq k\leq a}$ to be mean zero and apply Theorem \ref{max_sub_exp}, 
\begin{align*}
&\mathbb{P}\Big(\max_{0\leq k \leq a}\log Z^\lambda_{0, N+\olsi{k}} - \log Z^\lambda_{0, N} \geq C't\sqrt{a}\Big) \\
& \leq \mathbb{P}\Big(\max_{0\leq k \leq a}\Big\{\log Z^\lambda_{0, N+\olsi{k}} - \log Z^\lambda_{0, N} - \mathbb{E}[\log Z^\lambda_{0, N+\olsi{k}} - \log Z^\lambda_{0, N} ]\Big\}\geq \tfrac{1}{2}C't\sqrt{a}\Big) \\
&\leq e^{-C\min\{t^{2}\!,\, t\sqrt{a}\}}.
\end{align*}
This finishes the proof if we assume \eqref{lam_eq2} holds on a large probability event.

\begin{figure}
\begin{center}

\tikzset{every picture/.style={line width=0.75pt}} 

\begin{tikzpicture}[x=0.75pt,y=0.75pt,yscale=-1,xscale=1]

\draw    (185.67,118.17) -- (315.67,245.17) ;
\draw  [fill={rgb, 255:red, 0; green, 0; blue, 0 }  ,fill opacity=1 ] (326.83,72.58) .. controls (326.83,70.6) and (328.44,69) .. (330.42,69) .. controls (332.4,69) and (334,70.6) .. (334,72.58) .. controls (334,74.56) and (332.4,76.17) .. (330.42,76.17) .. controls (328.44,76.17) and (326.83,74.56) .. (326.83,72.58) -- cycle ;
\draw  [dash pattern={on 0.84pt off 2.51pt}]  (287.35,256.25) -- (330.42,72.58) ;
\draw [shift={(286.67,259.17)}, rotate = 283.2] [fill={rgb, 255:red, 0; green, 0; blue, 0 }  ][line width=0.08]  [draw opacity=0] (8.93,-4.29) -- (0,0) -- (8.93,4.29) -- cycle    ;
\draw  [fill={rgb, 255:red, 0; green, 0; blue, 0 }  ,fill opacity=1 ] (226.83,160.58) .. controls (226.83,158.6) and (228.44,157) .. (230.42,157) .. controls (232.4,157) and (234,158.6) .. (234,160.58) .. controls (234,162.56) and (232.4,164.17) .. (230.42,164.17) .. controls (228.44,164.17) and (226.83,162.56) .. (226.83,160.58) -- cycle ;
\draw  [fill={rgb, 255:red, 0; green, 0; blue, 0 }  ,fill opacity=1 ] (247.08,181.67) .. controls (247.08,179.69) and (248.69,178.08) .. (250.67,178.08) .. controls (252.65,178.08) and (254.25,179.69) .. (254.25,181.67) .. controls (254.25,183.65) and (252.65,185.25) .. (250.67,185.25) .. controls (248.69,185.25) and (247.08,183.65) .. (247.08,181.67) -- cycle ;
\draw  [fill={rgb, 255:red, 0; green, 0; blue, 0 }  ,fill opacity=1 ] (266.83,201.58) .. controls (266.83,199.6) and (268.44,198) .. (270.42,198) .. controls (272.4,198) and (274,199.6) .. (274,201.58) .. controls (274,203.56) and (272.4,205.17) .. (270.42,205.17) .. controls (268.44,205.17) and (266.83,203.56) .. (266.83,201.58) -- cycle ;

\draw  [fill={rgb, 255:red, 255; green, 255; blue, 255 }  ,fill opacity=1 ] (290.83,224.58) .. controls (290.83,222.6) and (292.44,221) .. (294.42,221) .. controls (296.4,221) and (298,222.6) .. (298,224.58) .. controls (298,226.56) and (296.4,228.17) .. (294.42,228.17) .. controls (292.44,228.17) and (290.83,226.56) .. (290.83,224.58) -- cycle ;
\draw (193,160.4) node [anchor=north west][inner sep=0.75pt]    {\small $( 0,0)$};
\draw (116,180.4) node [anchor=north west][inner sep=0.75pt]    {\small $( t^{2/3}N^{2/3} ,-t^{2/3}N^{2/3})$};
\draw (126,203.4) node [anchor=north west][inner sep=0.75pt]    {\small $( 2t^{2/3}N^{2/3} ,-2t^{2/3}N^{2/3})$};
\draw (310.04,168.78) node [anchor=north west][inner sep=0.75pt]    {$\boldsymbol\xi [ \lambda ]$};
\draw (336,60.4) node [anchor=north west][inner sep=0.75pt]    {\small $( N,\ N)$};

\draw (300,211.4) node [anchor=north west][inner sep=0.75pt]    {${\bf x}$};

\end{tikzpicture}

\captionsetup{width=0.8\textwidth}
\caption{The distance between $(0,0)$ and ${\bf x}$ is lower bounded by $c''q_0 t^{2/3}N^{2/3}$. By choosing $q_0$ sufficiently large, the ${\bf x}$ will be to the right and below the point $(2t^{2/3}N^{2/3} ,-2t^{2/3}N^{2/3})$. } \label{fig_fix_q}

\end{center}

\end{figure}
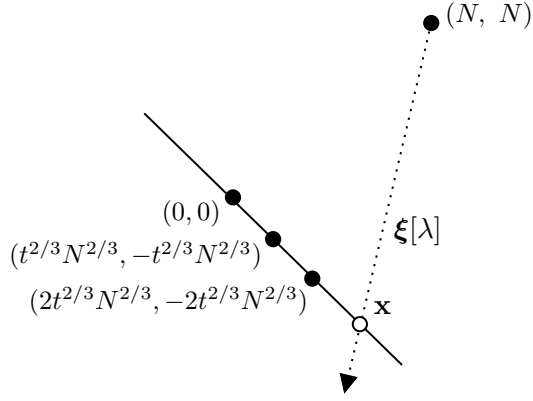

Finally, to complete the proof, we show that there is a large probability event $A$ with $\mathbb{P}(A) \geq 1- e^{-Ct}$ such that \eqref{lam_eq2} holds.
Note that $$\log Z^\lambda_{0, N+\olsi{k}}(\tau \geq t^{2/3}N^{2/3}) - \log Z^\lambda_{0, N}(\tau \geq t^{2/3}N^{2/3}) \leq  \log Z^\lambda_{0, N+\olsi{k}} - \log Z^\lambda_{0, N}(\tau \geq t^{2/3}N^{2/3}),$$
then it suffices to show that on a high probability event $A$, 
$$ \log Z^\lambda_{0, N} \leq  \log Z^\lambda_{0, N}(\tau \geq t^{2/3}N^{2/3}) + \log 2.$$
To obtain the inequality above, it suffices to show that 
\begin{equation}\label{lam_ineq}
\log Z^\lambda_{0, N}(\tau \geq t^{2/3}N^{2/3}) >\log Z^\lambda_{0, N}(\tau < t^{2/3}N^{2/3})
\end{equation}
with high probability. This follows from the standard exit estimate provided that the $q_0$ in \eqref{fix_lam} is fixed sufficiently large. So first, we will fix $q_0$. 

Let ${\bf x}$ denote the point where the $-\boldsymbol\xi[\lambda]$-directed ray from $(N,N)$ crosses $\mathcal{L}_0$, as shown in Figure \ref{fig_fix_q}. By applying Taylor's theorem to the slope of the vector $\boldsymbol\xi[\lambda]$ (defined through \eqref{char_dir}), the distance between $(0,0)$ and the point ${\bf x}$ (shown in Figure \ref{fig_fix_q}) is lower bounded by $c''q_0t^{2/3}N^{2/3}$ for some constant $c''$ depending on $\mu$. Then, we may fix $q_0$ such that ${\bf x}$ is to the right and below the point $(2t^{2/3}N^{2/3} ,-2t^{2/3}N^{2/3})$.

Next, we will define the event $A$. In the calculation below, Proposition \ref{nestedpoly} gives us the first equality. The next inequality is obtained by translating ${\bf x}$ and $(N,N)$ to ${\bf 0}$ and $N\boldsymbol{\xi}[\lambda]$, then observing that the shifted paths $-{\bf x} + \{\tau < t^{2/3}N^{2/3}\}$ are contained inside $\{\tau<-t^{2/3}N^{2/3}\}$. And the third inequality follows from Theorem \ref{stat_exit_est}. With these, we have
\begin{equation}
\begin{split}\label{def_Ac}
&\mathbb{P}(Q^\lambda_{0, N}\{\tau < t^{2/3}N^{2/3}\} \geq 1/10)
= \mathbb{P}(Q^\lambda_{{\bf x}, N}\{\tau < t^{2/3}N^{2/3}\} \geq 1/10)\\
& \leq \mathbb{P}(Q^\lambda_{0, N\boldsymbol\xi[\lambda]}\{\tau < -t^{2/3}N^{2/3}\} \geq 1/10)
\leq e^{-Ct}.
\end{split}
\end{equation}
Define 
$$A = \{Q^\lambda_{0, N}\{\tau \leq t^{2/3}N^{2/3}\} \leq 1/10\}.$$
Then \eqref{def_Ac} shows $\mathbb{P}(A) \geq 1-e^{-Ct}$, and on $A$, \eqref{lam_ineq} holds. With this, we have completed the proof of this proposition.
\end{proof}

Next, we extend the range of $t$ from the previous proposition. 
\begin{proposition}\label{all_t}
Suppose the boundary condition $W$ satisfies assumption \eqref{up_bd}. There exist positive constants $C_1, N_0, t_0$ such that for each $N\geq N_0$, $t\geq t_0$, and each positive integer $ a \leq \min\{\tfrac{1}{2}t^{2/3}N^{2/3}, N\} $, we have
$$\mathbb{P}\Big(\fgs_{\mathcal{L}^a_{N}} - \fgs_{ N} \geq C_1t \sqrt{a}\Big) \leq 
e^{-t^{1/10}}.
$$ 
\end{proposition}

\begin{proof}
First, note that when $t\leq c_0 N^{1/2}$ the result holds from the previous proposition. Now, suppose $t\geq c_0 N^{1/2}$, let $t= zN^{1/2}$ for some real positive number $z\geq c_0$. Then, if $a \leq N^{2/3}$,we have 
\begin{align*}
\mathbb{P}\Big(\fgs_{\mathcal{L}^a_{N}} - \fgs_{ N} \geq C_1t \sqrt{a}\Big) 
& \leq \mathbb{P}\Big(\fgs_{\mathcal{L}^a_{N}} - \fgs_{ N} \geq C_1t \Big)\\
& \leq \mathbb{P}\Big(\fgs_{\mathcal{L}^{N^{2/3}}_{N}} - \fgs_{ N} \geq C_1(zN^{1/6} )N^{1/3}\Big)\\
\text{(by Proposition \ref{b_up_tail}  and Proposition \ref{b_low_tail})\quad } & \leq e^{-CzN^{1/6}} \leq e^{-t^{1/10}}.
\end{align*}
Otherwise, if $N^{2/3} \leq a\leq \min\{\tfrac{1}{2}t^{2/3}N^{2/3}, N\} $, we have
\begin{align*}
\mathbb{P}\Big(\fgs_{\mathcal{L}^a_{N}} - \fgs_{ N} \geq C_1t \sqrt{a}\Big) 
& \leq \mathbb{P}\Big(\fgs_{\mathcal{L}^a_{N}} - \fgs_{ N} \geq C_1tN^{1/3}\Big)\\
\text{by Proposition \ref{b_up_tail} and  Proposition \ref{b_low_tail} \quad } &  \leq e^{-t^{1/10}}.
\end{align*}
\end{proof}

The next proposition captures the transversal fluctuation of the favorite point in the polymer.
Fix $N/2\leq r \leq N$.  Let ${\bf p}_*$ denote the random maximizer in   
$$\max_{{\bf p}\in \mathcal{L}_r} \big\{\fgs_{ {\bf p}}  + \fg_{{\bf p}, N} \big\} = \fgs_{ {\bf p}_*}  + \fg_{{\bf p}_*, N} .$$
The proposition below captures the KPZ transversal fluctuation which says that the maximizer ${\bf p}_*$ cannot be too far from the diagonal on the local scale $(N-r)^{2/3}$.

\begin{proposition}\label{fluc_bound_1}
Suppose the boundary condition $W$ satisfies assumption \eqref{up_bd}. There exist positive constants  $C_1, c_0, t_0, N_0 $ such that for each $N\geq N_0$, $N/2 \leq r \leq N-c_0$ and $t\geq t_0$, we have
$$\mathbb{P}(|{\bf p}_*-(r,r)|_\infty > t(N-r)^{2/3}) \leq e^{-C_1{t}^{3/2}}.$$
\end{proposition}

\begin{proof}
Abbreviate  $J^{h}=\mathcal{L}^{(N-r)^{2/3}}_{r+\olsi{2h(n-r)^{2/3}}}$.  
We bound the probability  as follows. 
\begingroup
\allowdisplaybreaks
\begin{align}
&\mathbb{P}(|{\bf p}_*-(r,r)|_\infty > t(N-r)^{2/3}) \nonumber\\
&\leq \mathbb{P}\Big(\max_{{\bf p}\in \mathcal{L}_r\setminus \mathcal{L}_r^{t(N-r)^{2/3}}}\Big\{\fgs_{{\bf p}}  + \fg_{{\bf p}, N} \Big \}  > \fgs_{ r}  + \fg_{r, N} \Big)\nonumber\\
&\leq \sum_{|h|= \floor{t/2}}^{(N-r)^{1/3}} \mathbb{P}\Big( \max_{{\bf x}\in J^h}\fgs_{ {\bf x}} + \max_{{\bf x}\in J^h} \fg_{{\bf x}, N}  > \fgs_{ r}  + \fg_{r, N}    \Big)\nonumber\\
&=  \sum_{|h|= \floor{t/2}}^{(N-r)^{1/3}}\mathbb{P}\Big(\Big[ \max_{{\bf x}\in J^h} \fgs_{ {\bf x}} - \fgs_{ r} \Big]   + \Big[\max_{{\bf x}\in J^h}\fg_{{\bf x}, N} - \fg_{r, N}  \Big]    > 0\Big)\nonumber\\
&=  \sum_{|h|= \floor{t/2}}^{(N-r)^{1/3}} \mathbb{P}\Big( \max_{{\bf x}\in J^h} \fgs_{ {\bf x} } - \fgs_{ r}  \geq Dh^2(N-r)^{1/3} \Big)\label{far_max}\\
&\qquad \qquad \qquad \qquad + \mathbb{P}\Big(\max_{{\bf x}\in J^h}\fg_{ {\bf x} , N} - \fg_{r, N}  \geq-Dh^2(N-r)^{1/3}\Big).\label{close_max}
\end{align}
\endgroup
where $D$ is a small positive constant that 
we will fix later.

For \eqref{close_max}, provided $t_0$ is fixed sufficiently large,
we may upper bound it using Proposition \ref{trans_fluc_loss} and Proposition \ref{low_ub} as below
\begin{align*}
\eqref{close_max} &= \mathbb{P}\Big([ \max_{{\bf x}\in J^h} \fg_{{\bf x}, N} -\Lambda_{N-r}] -[\fg_{r, N} - \Lambda_r] \geq -Dh^2(N-r)^{1/3} \Big)\\
& \leq \mathbb{P}\Big(\max_{{\bf x}\in J^h}\fg_{{\bf x}, N} -\Lambda_{N-r} \geq -2Dh^2(N-r)^{1/3} \Big) \\
& \qquad \qquad + \mathbb{P}\Big(\fg_{r, N} -\Lambda_{N-r} \leq - Dh^2(N-r)^{1/3} \Big) \\ 
&\leq e^{-C|h|^3}
\end{align*}
provided $D \leq \tfrac{1}{10}C^*$ where $C^* = C_1$ from Proposition \ref{trans_fluc_loss}.

For \eqref{far_max}, we will split the estimate into two cases depending on the value of $N-r$, whether $N-r\leq \epsilon_0 r$ or $N-r\geq \epsilon_0 r$, for $\epsilon_0$ which we will fix below between the math displays \eqref{prob_estt} and  \eqref{fix_D}.

When $N-r\leq \epsilon_0 r$, we upper bound \eqref{far_max} by 
\begin{equation}\label{bound_small_r}
\eqref{far_max} \leq \mathbb{P}\Big(\max_{{\bf x}\in \mathcal{L}_r^{4|h|(N-r)^{2/3}}}\fgs_{ {\bf x} } - \fgs_{ r}  \geq Dh^2(N-r)^{1/3}\Big).
\end{equation}
To apply Proposition \ref{local_fluc}, let us set $a = 4|h|(N-r)^{2/3}$ and $t = 8|h|^{3/2}\epsilon_0^{3/2}$, and we have
\begin{align}\eqref{bound_small_r} =  \mathbb{P}\Big(\max_{{\bf x}\in \mathcal{L}_r^{a}}\log Z_{{\bf x}, N} - \log Z_{r, N}  \geq \tfrac{D}{16\epsilon_0^{3/2}}t\sqrt{a}\Big). \label{prob_estt}
\end{align}
Next, we fix $\epsilon_0$ sufficiently small so that $t  \leq \wt c_0 r^{1/2} $, where $\wt c_0 = c_0$ from Proposition \ref{local_fluc}. Next, we lower the value of $D$ to get
\begin{equation}\label{fix_D}
\tfrac{\frac{1}{10}D}{8\epsilon_0^{3/2}} \leq C^*
\end{equation}
where $C^* = C_1$ appearing in Proposition \ref{local_fluc}. Finally, by Proposition \ref{local_fluc}, the above probability in \eqref{prob_estt} will always be bounded by $e^{-Ct} = e^{-C|h|^{3/2}}$.

On the other hand, let us look at the case when $N-r\geq \epsilon_0r$, and note that now $\epsilon_0$ has already been fixed. Since the value of $N-r$ will always be less than $r$ when $r \geq N/2$, let us set $N-r = zr$ for some positive constant $z$ with $\epsilon_0 \leq z \leq 1$.  We can obtain the following upper bound as
\begin{align*}
&\mathbb{P}\Big( \max_{{\bf x}\in J^h} \log Z^{W}_{ {\bf x}} - \fgs_{ r}  \geq Dh^2(N-r)^{1/3} \Big)\\
& \leq \mathbb{P}\Big( \max_{{\bf x}\in \mathcal{L}_r^{4|h|(N-r)^{2/3} }} \fgs_{ {\bf x}} - \fgs_{ r}  \geq Dh^2(N-r)^{1/3} \Big)\\
& = \mathbb{P}\Big( \max_{{\bf x}\in \mathcal{L}_r^{4|h|(N-r)^{2/3} }} \fgs_{ {\bf x}} - \fgs_{ r}  \geq Dh^2z^{1/3}r^{1/3} \Big)\\
& \leq \mathbb{P}\Big( \max_{{\bf x}\in \mathcal{L}_r^{4|h|(N-r)^{2/3} }} 
 \log Z^{W}_{ {\bf x}} - \Lambda_r  \geq \tfrac{1}{2} Dh^2z^{1/3}r^{1/3} \Big) + \mathbb{P}\Big(  \fgs_{ r}  - \Lambda_r \leq -\tfrac{1}{2}Dh^2z^{1/3}r^{1/3} \Big)\\
& \leq e^{-C|h|^{3/2}} \qquad \text{by Proposition \ref{b_up_tail} and Proposition \ref{b_low_tail}.}
\end{align*}

To summarize, the arguments above show that 
$$\sum_{|h| = \floor{t/2}}^{(N-r)^{1/3}} \eqref{close_max} + \eqref{far_max}\leq  \sum_{|h| = \floor{t/2}}^{\infty}e^{-C|h|^{3/2}} \leq e^{-Ct^{3/2}},$$
with this, we have finished the proof of this proposition.
\end{proof}

Finally, we prove Theorem \ref{nest}.
\begin{proof}[Proof of Theorem \ref{nest}]

Note because of 
$$\log Z^W_{ N} \leq \max_{{\bf p}\in \mathcal{L}_r}\{\log Z^W_{{\bf p}}  + \log \wt{Z}_{{\bf p} , N}\} + 2\log(N-r),$$
it suffices to show that 
\begin{align*}
&\mathbb{P} \Big( \max_{{\bf p}\in \mathcal{L}_r} \Big \{\log Z^W_{{\bf p}}  + \log \wt{Z}_{{\bf p} , N} \Big\} - \Big[\log  Z^W_{r}  + \log \wt{Z}_{r,N} \Big]\geq t(N-r)^{1/3}\Big) \leq e^{- t^{1/10}}.
\end{align*}
By a union bound, we split the above maximum over ${\bf p}\in \mathcal{L}_r$ to ${\bf p}\not\in \mathcal{L}_r^{t(N-r)^{2/3}}$ and ${\bf p}\in \mathcal{L}_r^{t(N-r)^{2/3}}$.  

In the case when ${\bf p}\in \mathcal{L}_r^{t(N-r)^{2/3}}$, we have
\begin{align}
&\mathbb{P} \Big( \max_{{\bf p}\in \mathcal{L}_r^{t(N-r)^{2/3}}}\Big\{\log Z^W_{{\bf p}}  + \log \wt{Z}_{{\bf p} , N}\Big\}  - \Big[\log  Z^W_{r}  + \log \wt{Z}_{r,N}  \Big]\geq t(N-r)^{1/3}\Big)\nonumber\\
& \leq   \mathbb{P} \Big(\max_{{\bf p}\in \mathcal{L}_r^{t(N-r)^{2/3}}} \log Z^W_{{\bf p}}  - \log  \wt{Z}_{0, r}  \geq \tfrac{1}{2}t(N-r)^{1/3}\Big) \label{far_term} \\
&\qquad \qquad + \mathbb{P} \Big(\max_{{\bf p}\in \mathcal{L}_r^{t(N-r)^{2/3}}}\log \wt{Z}_{{\bf p}, N}  - \log \wt{Z}_{r,N} \geq \tfrac{1}{2}t(N-r)^{1/3}\Big). \label{close_term}
\end{align}
Both probabilities \eqref{far_term} and \eqref{close_term} can be upper bounded by $e^{-Ct^{1/10}}$ using Proposition \ref{all_t}.
In the case when ${\bf p}\not\in \mathcal{L}_r^{t(N-r)^{2/3}}$, this follows from Proposition \ref{fluc_bound_1}.
\end{proof}
\addtocontents{toc}{\protect\setcounter{tocdepth}{2}}

\section{Parallel results in exponential LPP}
\label{s:lpp}
As already mentioned in the introduction, the results of this paper all have analogues in the setting of exponential LPP on $\Z^2$. Our proofs carry over almost verbatim to this set-up upon replacing the free energy by the last passage time, the restricted free energy over a class of paths by the maximum passage time among the paths in that class, and the quenched polymer measure by the measure that puts mass one on the geodesic in the LPP setting. We shall not be repeating the arguments, but will formulate proper statements to this effect. 

Let $\{Y_{v}\}_{v\in \mathcal{L}_0^{>}}$ denote a collection of i.i.d.\ rate one exponential random variables. Let $H_{k}$ be a sequence of random variables with $H_0=0$ such that $\{\wt{X}_i:=H_{i}-H_{i-1}\}$ is a collection of independent mean $0$ random variables satisfying the following conditions: 
\begin{enumerate}
    \item[(i)] There exists $\lambda_1,K_1>0$ such that for all $i$,
    $\log \mathbb{E}(\exp(\lambda |\wt{X}_{i}|))\le K_1\lambda^2$ for $\lambda\in [-\lambda_1,\lambda_1]$. 
    \item[(ii)] $\inf_{i} \Var \wt{X}_i>0$. 
\end{enumerate}

The last passage time between ${\bf u}$ and ${\bf v}$ (without any boundary condition), denoted $\wt{T}_{{\bf u},{\bf v}}$ is defined by 
$$\wt T_{{\bf u},{\bf v}}:=\max_{{\bf x}_{\bbullet} \in\mathbb{X}_{{\bf a}, {\bf b}}}\sum_{i=1}^{|{\bf v} - {\bf u}|_1} Y_{{\bf x}_i}.$$
For $\mathbf{v}\in \mathcal{L}^{>}_{0}$, let us define the last passage time to $v$ with initial condition $H$ by 
$$T^H_{\bf v}=\max_{k\in \Z} \left( H_{k}+T_{(k,-k),{\bf v}}\right).$$
We then have the following results corresponding to Theorems \ref{thm_r_large} and \ref{thm_r_small}. 

\begin{theorem}\label{thm_r_large_lpp}
Consider exponential LPP with initial condition $H$ on $\mathcal{L}_{\bf 0}$ satisfying the above hypotheses.
There exist positive constants $C_1, C_2, c_0, N_0$ such that, whenever $N\geq N_0$ and $N/2 \le r \leq N-c_0$,
we have 
$$ 1-C_1\Big(\frac{N-r}{N}\Big)^{2/3} \leq \textup{$\mathbb{C}$orr}\Big(T^{H}_{(r,r)}, T^{H}_{(N,N)}\Big) \leq 1-C_2\Big(\frac{N-r}{N}\Big)^{2/3}.$$ 
\end{theorem}

\begin{theorem}\label{thm_r_small_lpp}  
Consider exponential LPP with initial condition $H$ on $\mathcal{L}_{\bf 0}$ satisfying the above hypotheses. There exist positive constants $C_3, C_4, c_0,  N_0$ such that, whenever $N\geq N_0$ and $c_0 \leq r \leq N/2$,  we have 
$$ C_3\Big(\frac{r}{N}\Big)^{1/3} \leq \textup{$\mathbb{C}$orr}\Big(T^{H}_{(r,r)}, T^{H}_{(N,N)}\Big) \leq C_4\Big(\frac{r}{N}\Big)^{1/3}.$$ 
\end{theorem}

As the reader might have noticed, our arguments for the inverse-gamma polymer essentially used only the curvature of the limit shape (Proposition \ref{reg_shape}), and the moderate deviation estimates for the free energy, see Section 3.3.1 of \cite{bas-sep-she-23} for a discussion about this. All the other estimates were then developed using these ingredients together with certain random walk comparison results. The curvature of limit shape for exponential LPP is known \cite{Joh-00}, as are the moderate deviation estimates \cite{led-rid-2021}. Many of the required auxiliary estimates have already been developed in the exponential LPP setting (see \cite{timecorriid, timecorrflat}). Random walk comparisons in the zero temperature setting first appeared in the seminal work \cite{Cat-Gro-05}, since then it has been employed in many places \cite{balzs2019nonexistence, cuberoot, seppcoal}. By quoting these auxiliary estimates as necessary, the proofs of Theorems \ref{thm_r_large_lpp} and \ref{thm_r_small_lpp} can be completed following the proofs of Theorems \ref{thm_r_large} and \ref{thm_r_small} almost verbatim. 

Before finishing this section, we point out that a variant of Theorem \ref{thm_r_large_lpp} (in the $n\to \infty$ limit with $\tau=r/n$ close to $1$) was established in \cite{ferocctimecorr} in the special case where $\wt X_{i}$ are i.i.d.\ having the same distribution as $\sigma(U_{i}-V_{i})$ where $U_{i},V_{i}$ are two independent sequences of i.i.d. random variables distributed as $\exp(1/2)$ and $\sigma\in [0,\infty)$. Theorem \ref{thm_r_large_lpp} extends this to more general initial conditions and to the pre-limiting set-up. Theorem \ref{thm_r_small_lpp} appears to be new except in the stationary case with density $1/2$ (i.e., when $H_{k}$ is a two sided random walk with i.i.d.\ increments with distribution $W-W'$ where $W,W'$ are independent copies of $\mbox{Exp}(1/2)$ random variables) which was also dealt with in \cite{ferocctimecorr}, again only in the $n\to \infty$ limit. 

We also emphasize that several extensions of Theorems \ref{thm_r_large} and \ref{thm_r_small} remain valid in this set-up as well. We can extend the hypothesis on the initial condition to an analogue of \textbf{Assumption B1, Assumption B2, Assumption B3} or relax even further as discussed in Remark \ref{ext}. It should also be possible cover stationary initial conditions with different densities for time correlations if the endpoint varies along the corresponding characteristic direction, but we shall not get into the details here. 

\appendix
\section{Appendix}
\subsection{Sub-exponential random variables.}\label{sub_exp_sec}
First, we state a general result for the running maximum of sub-exponential random variables.
Recall that a random variable $X_1$ is sub-exponential if there exist two positive constants $K_0$ and $\lambda_0$ such that 
\be\label{sub_exp}
\log(\mathbb{E}[e^{\lambda (X_1-\mathbb{E}[X_1])}]) \leq K_0 \lambda^2 \quad \textup{ for $\lambda \in [0, \lambda_0]$}.
\ee
Let $\{X_i\}$ be a sequence of i.i.d.~sub-exponential random variables with the parameters $K_0$ and $\lambda_0$. Define $S_0 = 0$ and $S_k = X_1 + \dots + X_k - k\mathbb{E}[X_1]$ for $k\geq 1$. The following theorem gives an upper bound for the right tail of the running maximum.  
\begin{theorem}[{\cite[Theorem D.1]{bas-sep-she-23}}]\label{max_sub_exp}
Let the random walk $S_k$ be defined as above. Then,
$$\mathbb{P} \Big(\max_{0\leq k \leq n} S_k \geq t\sqrt{n}\Big) \leq 
\begin{cases}
e^{-t^2/(4K_0)} \quad  & \textup{if $t \leq 2\lambda_0 K_0 \sqrt n$} \\
e^{-\frac{1}{2}\lambda_0 t\sqrt{n}} \quad  & \textup{if $t \geq 2\lambda_0 K_0 \sqrt n$}
\end{cases}.
$$ 
\end{theorem}

Next, we will state a lower bound for the right tail probability which was obtained from \cite{non_asy_rw} using Paley–Zygmund inequality. Let $\{Z_i\}$ be a sequence of i.i.d.~random variables such that there exist parameters $K_1$ and $\lambda_1$ with 
\begin{equation}\label{lower_sub}
\tfrac{1}{K_1} \lambda^2 \leq \log(\mathbb{E}[e^{\lambda (Z_1-\mathbb{E}[Z_1])}]) \leq K_1 \lambda^2 \quad \textup{ for $\lambda \in [0,\lambda_1]$}.
\end{equation}
Similar to before, define $H_0 = 0$ and $H_k = Z_1 + \dots + Z_k - k\mathbb{E}[Z_1]$ for $k\geq 1$, then the following holds. 
\begin{theorem}[{\cite[Theorem 4]{non_asy_rw}}]\label{ul_bd_exp}
Let the random walk $H_n$ be defined as above. Then, there exists positive constants $C_1, C_2, C_3, C_4$ depending on $K_1$ and $\lambda_1$ such that  
$$C_1e^{-C_2 t^2} \leq \mathbb{P}(H_n \geq t\sqrt{n}) \leq e^{-C_3 t^2} \qquad\text{ for each $0\leq t \leq C_4 \sqrt n$.}
$$
\end{theorem}

\subsubsection{Verifying \eqref{lower_sub} for stationary model weights} \label{ver_weights}
Since assumption \eqref{lower_sub} is strictly stronger than \eqref{sub_exp}, we will just verify  \eqref{lower_sub} for our applications in this paper. 
First, we have a proposition which shows that both $\log(\textup{Ga})$ and $\log(\textup{Ga}^{-1}) = -\log(\textup{Ga})$
satisfies \eqref{lower_sub}. 
\begin{proposition}\label{Ga_sub_exp}
Fix $\epsilon \in (0, \mu/2)$. There exists positive constants $K_1, \lambda_1$ depending on $\epsilon$ such that for each $\alpha \in [\epsilon, \mu-\epsilon]$, let $X\sim \textup{Ga}(\alpha)$ and we have
$$
\tfrac{1}{K_1} \lambda^2 \leq \log(\mathbb{E}[e^{\pm\lambda(\log(X) - \Psi_1(\alpha))}]) \leq K_1 \lambda^2 \qquad \textup{ for $\lambda \in [0, \lambda_1]$}.
$$
\end{proposition}
\begin{proof}
First, note that $\mathbb{E}[X^{\pm \lambda}] = \frac{\Gamma(\alpha\pm\lambda)}{\Gamma(\alpha)}$, provided that $\alpha \pm \lambda > 0$. 
Then, the proof essentially follows from Taylor's theorem,
\begin{align*}
\log(\mathbb{E}[e^{\pm\lambda(\log(X) - \Psi_1(\alpha))}]) & = \log(\mathbb{E} [X^{\pm\lambda}]e^{\mp\lambda \Psi_1(\alpha)})\\
 \qquad & = \log(\Gamma(\alpha \pm \lambda)) -[\log(\Gamma(\alpha)) \pm \lambda \Psi_1(\alpha)]\\
\textup{(recall $\log(\Gamma(\alpha))' = \Psi_1(\alpha)$) } \quad & =\frac{\Psi_1'(\alpha)}{2}\lambda^2 + o(\lambda^2).
\end{align*}
The last line is bounded between $\tfrac{1}{K_1} \lambda^2$ and $K_1 \lambda^2$, provided $\lambda_1$ is fixed sufficiently small. And the constant $K_1$ can be chosen uniformly for all $\alpha$ from the compact interval $[\epsilon, \mu-\epsilon]$ because $\Phi_1$ is a smooth function on $\mathbb{R}_{\geq 0}$.
\end{proof}



\subsection{Monotonicity results for the polymer model}
Recall the definition of a polymer with initial condition $W$ from Section \ref{def_poly_ad} and the exit time $\tau$ defined at the beginning of Section \ref{est_poly_bdry}.
We state two monotonicity results for the ratios of partition functions. In this section, the weights can take any arbitrary positive values.

The first proposition gives certain monotonicity between the exit times and ratios of the partition function. The proof follows from a ``corner flipping" induction. For a similar result in the setting of two different southwest boundaries (instead of for different exit times), see Lemma A.1 from \cite{Bus-Sep-22}.
\begin{proposition}\label{ratio_mono2}
Fix a positive initial condition $W$ with $W_0 = 1$. Let for each ${\bf a} \in \mathbb{Z}^2_{>0}$ and $k \geq 1$, 
$$\frac{Z^W_{{\bf a}}(|\tau| \leq k)}{Z^W_{{\bf a}-{\bf e_1}}(|\tau| \leq k)} \leq  \frac{Z^W_{{\bf a}}(\tau \geq 0)}{Z^W_{{\bf a}-{\bf e_1}}(\tau \geq 0)} \qquad \text{ and } \qquad \frac{Z^W_{{\bf a}}(|\tau| \leq k)}{Z^W_{{\bf a}-{\bf e_2}}(|\tau| \leq k)} \geq  \frac{Z^W_{{\bf a}}(\tau \geq 0)}{Z^W_{{\bf a}-{\bf e_2}}(\tau \geq 0)}.$$
\end{proposition}
\begin{proof}
Fix $W$, let us define two more boundary weights $W^{(1)}$ and $W^{(2)}$, given by 
$$
W^{(1)}_{i} = \begin{cases}
W_{i} \quad & \text{for } |i| \leq k\\
0 \quad & \text{otherwise } 
\end{cases} \qquad \text{ and } \qquad
W^{(2)}_{i} = \begin{cases}
W_{i} \quad & \text{for }i \geq 0 \\
0 \quad & \text{for }i \leq -1
\end{cases}.
$$
Then, $Z^W_{{\bf x}}(|\tau| \leq k) = Z_{{\bf x}}^{W^{(1)}}$ and  $Z_{ {\bf x}}^W(\tau \geq 0) = Z_{{\bf x}}^{W^{(2)}}$. 
In the ratio calculation below, let us follow the convention $0/0 = 0$, then for each $i\in \mathbb{Z}_{>0}$, it holds that 
\begin{align}\frac{\log Z^{W^{(1)}}_{(i,-i)}}{\log Z^{W^{(1)}}_{(i-1,-i+1)}} = \frac{W^{(1)}_i}{W^{(1)}_{i-1}} &\leq \frac{W^{(2)}_i}{W^{(2)}_{i-1}} = \frac{\log Z^{W^{(2)}}_{(i,-i)}}{\log Z^{W^{(2)}}_{(i-1,-i+1)}}\label{base1}\\
\frac{\log Z^{W^{(1)}}_{(-i,i)}}{\log Z^{W^{(1)}}_{(-i+1,i-1)}} = \frac{W^{(1)}_{-i}}{W^{(1)}_{-i+1}} &\geq \frac{W^{(2)}_{-i}}{W^{(2)}_{-i+1}} = \frac{\log Z^{W^{(2)}}_{(-i,i)}}{\log Z^{W^{(2)}}_{(-i+1,i-1)}}\label{base2}.
\end{align}
These will be the base case of our induction. 

Next, using the inductive relation 
$
Z^{W}_{{\bf x}}=(Z^{W}_{{\bf x}-{\bf e}_1}+Z^{W}_{{\bf x}-{\bf e}_2})Y_{\bf x},
$
and the base case \eqref{base1}, \eqref{base2}, we have 
\begin{align*}\frac{Z^{W^{(1)}}_{{\bf x}}}{Z^{W^{(1)}}_{{\bf x}-{\bf e_1}}}=\Big(1+\frac{Z^{W^{(1)}}_{{\bf x}-{\bf e}_2}}{Z^{W^{(1)}}_{{\bf x}-{\bf e_1}}}\Big)Y_{\bf x} \leq \Big(1+\frac{Z^{W^{(2)}}_{{\bf x}-{\bf e}_2}}{Z^{W^{(2)}}_{{\bf x}-{\bf e_1}}}\Big)Y_{\bf x} = \frac{Z^{W^{(2)}}_{{\bf x}}}{Z^{W^{(2)}}_{{\bf x}-{\bf e_1}}}\\
\frac{Z^{W^{(1)}}_{{\bf x}}}{Z^{W^{(1)}}_{{\bf x}-{\bf e_2}}}=\Big(\frac{Z^{W^{(1)}}_{{\bf x}-{\bf e}_1}}{Z^{W^{(1)}}_{{\bf x}-{\bf e_2}}} + 1\Big)Y_{\bf x} \geq \Big(\frac{Z^{W^{(2)}}_{{\bf x}-{\bf e}_1}}{Z^{W^{(2)}}_{{\bf x}-{\bf e_2}}} + 1\Big)Y_{\bf x} = \frac{Z^{W^{(2)}}_{{\bf x}}}{Z^{W^{(2)}}_{{\bf x}-{\bf e_2}}}.
\end{align*}

\end{proof}

The second proposition is similar to the first one, and it has appeared as Lemma A.4 from \cite{ras-sep-she-} for the polymer model with south-west boundary. The same ``corner flipping" induction there also applies to the anti-diagonal boundary, so we omit the proof. 
\begin{proposition}\label{ratio_mono1}
Fux a positive initial condition $W$ with $W_0 = 1$. Let for each ${\bf a} \in \mathbb{Z}^2_{>}0$ and $k\geq 1$, 
$$\frac{Z^W_{{\bf a}}(\tau \geq 0)}{Z^W_{{\bf a}-{\bf e_1}}(\tau \geq 0)} \leq  \frac{Z^W_{{\bf a}}(\tau \geq k)}{Z^W_{{\bf a}-{\bf e_1}}(\tau \geq k)} \qquad \text{ and } \qquad \frac{Z^W_{{\bf a}}(\tau \geq 0)}{Z^W_{{\bf a}-{\bf e_2}}(\tau \geq 0)} \geq  \frac{Z^W_{{\bf a}}(\tau \geq k)}{Z^W_{{\bf a}-{\bf e_2}}(\tau \geq k)}.$$
\end{proposition}

The last result is a specific instance of Lemma A.7 from \cite{ras-sep-she-}. While this result applies to positive weights of any value, we will present it within the context of the stationary inverse-gamma polymer. Because its statement is better suited to the definition of the partition function as outlined in Section \ref{stat_poly} (instead of Section \ref{def_poly_ad}).

\begin{proposition} \label{nestedpoly}
Fix ${\bf a} \in \mathbb{Z}^2$. 
Then, for each $\rho \in (0, \mu)$, $k \in \mathbb{Z}$, ${\bf z}\in \mathcal{L}_{\bf a}^{\geq}$, ${\bf w} \in \mathcal{L}_{\bf a}^>$ and $i=1,2$
$$Q^{\rho}_{{\bf a}, {\bf w}}\{\text{paths go through $[\![{\bf z}, {\bf z}+e_i]\!]$}\} = Q^{\rho}_{{\bf a} + \olsi{k}, {\bf w}}\{\text{paths go through $[\![{\bf z}, {\bf z}+e_{i}]\!]$}\}.$$
\end{proposition}

\subsection{Assumption A implies Assumption B}
\begin{proposition}
    \label{aimpliesa1}
    Let $X_i$ be a sequence of independent mean $0$ random variables satisfying the hypotheses of {\rm \textbf{Assumption A}}. Then $\{W_{k}\}_{k\in \Z}$ defined by $W_0=1$ and $X_{i}=\log W_{i}-\log W_{i-1}$ satisfy {\rm\textbf{Assumption B1}, \textbf{Assumption B2}} and {\rm\textbf{Assumption B3}}. 
\end{proposition}

\begin{proof}
Since $X_{i}$ are independent, many of the hypotheses in \textbf{Assumption B2} and \textbf{Assumption B3} are redundant. Indeed, \eqref{cond_up_bd} and \eqref{cond_low_tail} follow from \eqref{up_bd} and \eqref{low_tail} respectively. The mixing hypothesis \eqref{mix} is also a trivial consequence of independence. The FKG inequality hypothesis also follows from
the standard FKG inequality for product measures. In \eqref{as_var} the conditioning is redundant and $\Var (\log W_{N})\ge C_1N$ follows from the hypothesis that $\inf_{i} \Var X_{i}>0$. It therefore only remains to verify \eqref{up_bd}, \eqref{low_tail} and \eqref{low_bd}.

For \eqref{up_bd} and \eqref{low_tail}, notice that since $\mathbb{E} X_{i}=0$, the first hypothesis of \textbf{Assumption A} implies that $|X_i|$ are uniformly sub-exponential. Using Theorem \ref{max_sub_exp}, which is a consequence of Doob maximal inequality, we get \eqref{up_bd} and \eqref{low_tail}. For \eqref{low_bd}, we appeal to a Berry-Essen theorem for sums of non i.i.d. variables (see e.g.\ \cite{BMN77}). Since $\mathbb{E} X_i=0$ and  $\sup_{i}\mathbb{E} |X_{i}|^{3}<\infty$ (by the subexponentiality hypothesis), it follows that for $Z\sim N(0,1)$,
$$\sup_{t\in \R} \mathbb{P}(\log W_{n}\ge t\sqrt{\Var (\log W_{n})})-\mathbb{P}(Z\ge t)=O(n^{-1/2}).$$
Since $\Var \log W_{n}\in [C_1N,C_2N]$ for some $C_1>0$ (by the second hypothesis in \textbf{Assumption A}) and $C_2<\infty$ (by the first hypothesis in \textbf{Assumption A}), \eqref{low_bd} follows. 
\end{proof}

\bibliographystyle{amsplain}
\bibliography{time}

\end{document}